\documentclass[reqno]{amsart}

\usepackage{multirow}
\usepackage{hyperref}
\usepackage{enumerate}
\usepackage{amsfonts,amssymb,amsmath,amsthm}
\usepackage{epsfig}
\usepackage{graphics}
\usepackage[normalem]{ulem}
\usepackage{color}
\usepackage{version}
\usepackage{caption}
\usepackage{subcaption}
\usepackage{numprint}
\usepackage{thm-restate}
\usepackage{lipsum}

\usepackage{hyperref}

\usepackage{cleveref}
\usepackage{bm}
\usepackage{dsfont}
\usepackage{color}
\usepackage{tikz}
\usetikzlibrary{cd}

\newcommand{\Seven}{S_{\text{even}}}
\newcommand{\Sodd}{S_{\text{odd}}}

\DeclareMathOperator{\Tr}{Tr}

\input xy 
\xyoption{all}
\numberwithin{equation}{section}

\definecolor{OrangeRed}{cmyk}{0,0.6,1,0}            
\definecolor{DarkBlue}{cmyk}{1,1,0,0.20}
\definecolor{DarkGreen}{cmyk}{1,0,0.6,0.2}
\definecolor{myblue}{rgb}{0.66,0.78,1.00}
\definecolor{Violet}{cmyk}{0.79,0.88,0,0}
\definecolor{Lavender}{cmyk}{0,0.48,0,0}

\newtheorem{thm}{Theorem}[section]
\newtheorem{theorem}[thm]{Theorem}
\newtheorem*{theorem*}{Theorem}
\newtheorem{main theorem}[thm]{Main Theorem}
\newtheorem{corollary}[thm]{Corollary}
\newtheorem*{main}{Main Theorem}
\newtheorem{lemma}[thm]{Lemma}

\newtheorem{prop}[thm]{Proposition}

\theoremstyle{definition}
\newtheorem{question}{Question}

\newtheorem{definition}[thm]{Definition}
\newtheorem{remark}[thm]{Remark}

\def\bcases{\begin{cases}}

\def\ecases{\end{cases}}

\newcommand{\bea}{\begin{eqnarray*}}
\newcommand{\eea}{\end{eqnarray*}}

\newcommand{\be}{\begin{equation}}
\newcommand{\ee}{\end{equation}}

\DeclareMathOperator{\lab}{lab}
\DeclareMathOperator{\ext}{ext}
\DeclareMathOperator{\inte}{int}

\newcommand{\diam}{\mathrm{diam}}
\newcommand{\dist}{\operatorname{dist}}
\renewcommand{\epsilon}{\varepsilon}

\DeclareMathOperator{\TorusCompatContours}{\Upsilon}
\DeclareMathOperator{\SetOfContours}{\Omega}
\DeclareMathOperator{\Balancedtori}{\mathbf{T}}
\DeclareMathOperator{\torus}{\mathcal{T}}

\definecolor{gG}{RGB}{ 60, 186,  84}
\definecolor{gY}{RGB}{244, 194,  13}
\definecolor{gB}{RGB}{48., 88.6667, 158.}
\definecolor{gR}{RGB}{219,  50,  54}

\definecolor{c0}{RGB}{255,0,0}
\definecolor{c1}{RGB}{0,0,255}

\begin{document}

\title{On boundedness of zeros of the independence polynomial of tori}

\author[D. de Boer]{David de Boer}
\author[P. Buys]{Pjotr Buys}
\author[H. Peters]{Han Peters}
\author[G. Regts]{Guus Regts}

\thanks{$\ddagger$ DdB and PB were funded by the Netherlands Organisation of Scientific Research (NWO): 613.001.851. GR was funded by the Netherlands Organisation of Scientific Research (NWO): VI.Vidi.193.068}
\address[David de Boer, Pjotr Buys, Han Peters, Guus Regts]{Korteweg de Vries Institute for Mathematics, University of Amsterdam. P.O. Box 94248  
1090 GE Amsterdam  
The Netherlands}
\email{\{daviddeboer2795,pjotr.buys,hanpeters77,guusregts\}@\texttt{gmail.com} }

\date{\today}


\begin{abstract}
We study boundedness of zeros of the independence polynomial of tori for sequences of tori converging to the integer lattice. We prove that zeros are bounded for sequences of balanced tori, but unbounded for sequences of highly unbalanced tori.
Here balanced means that the size of the torus is at most exponential in the shortest side length, while highly unbalanced means that the longest side length of the torus is super exponential in the product over the other side lengths cubed.
We discuss implications of our results to the existence of efficient algorithms for approximating the independence polynomial on tori.

This project was partially inspired by the relationship between zeros of partition functions and  holomorphic dynamics, a relationship that in the last two decades played a prominent role in the field. Besides presenting new results, we survey this relationship and its recent consequences.
\end{abstract}

\maketitle

\section{Introduction}\label{sec: introduction}

\subsection{Main results} The independence polynomial of a finite simple graph $G = (V,E)$ is defined by
$$
Z_G(\lambda) = \sum_I \lambda^{|I|},
$$
where the summation runs over all \emph{independent} subsets $I \subseteq V$. Besides its relevance in graph theory, the independence polynomial is studied extensively in the statistical physics literature, where it appears as the partition function of the hard-core model, and in theoretical computer science, where one is primarily interested in the (non-)existence of efficient algorithms for the computation or approximation of $Z_G$.

From the physical viewpoint it is particularly interesting to consider sequences of graphs $G_n$ that converge to a regular lattice. We will consider the integer lattice, and focus on sequences of $d$-dimensional tori converging to $\mathbb Z^d$ for $d \geq 2$, i.e. tori whose minimal cycle lengths tend to infinity. 
Write $\mathbb{Z}_{n}$ for $\mathbb{Z}/n \mathbb{Z}$.
A $d$-dimensional torus with side lengths $\ell_1,\ldots,\ell_d$ is the Cartesian product $\mathbb{Z}_{\ell_1} \times \cdots \times \mathbb{Z}_{\ell_d}$.
For technical reasons explained below we only consider tori for which all side lengths are even and call those tori \emph{even}.
The main result of this paper is the following:

\begin{main}\label{thm:main}
Let $\mathcal{F}$ be a family of even $d$-dimensional tori. If $\mathcal{F}$ is balanced, then the zeros of the independence polynomials $\{Z_{\mathcal{T}}:\mathcal{T} \in \mathcal{F}\}$ are uniformly bounded. If $\mathcal{F}$ is highly unbalanced, then the zeros are not uniformly bounded.
\end{main}

Here we say that a family of $d$-dimensional tori $\mathcal{F}$ is balanced if there exists a $C > 0$ such that for all $\mathcal{T} \in \mathcal{F}$ we have that $\ell_d \le \mathrm{Exp}(C\cdot \ell_1)$, where $\ell_1 \leq \cdots \leq \ell_d$ denote the side lengths of $\mathcal{T}$. On the other hand we say that the family is \emph{highly unbalanced} if there is no uniform constant $C>0$ such that $\ell_d \le \mathrm{Exp}(C \cdot \left(\ell_1 \cdots \ell_{d-1}\right)^3)$ for all $\mathcal{T} \in \mathcal{F}$.

We remark that a family that is not balanced is not necessarily highly unbalanced, hence the addition of the adjective highly. It is not clear to the authors that either estimate is sharp, and it would be interesting if one or both of the results could be sharpened in order to obtain a conclusive statement for all families of even tori.

The motivation for our main result comes from three different directions: (1) statistical physics, (2) the existence of efficient approximation algorithms and (3) the relationship with holomorphic dynamics. The first two of these will be briefly discussed in this introduction. The last will be discussed more extensively in the next section.

\subsection{Motivation from statistical physics}
Understanding the location and distribution of zeros of the independence polynomial plays a prominent role in statistical physics. For a sequence of graphs $G_n = (V_n, E_n)$ and for $\lambda \ge 0$ the free energy per site (also called the \emph{pressure}) is defined by
\begin{equation}\label{eq:free energy}
\rho(\lambda) := \lim_{n \rightarrow \infty} \frac{\log  Z_{G_n}(\lambda)}{|V_n|},
\end{equation}
whenever this limit exists. It was shown by Yang and Lee~\cite{YangLee} that the free energy per site exists for induced subgraphs $G_n$ of $\mathbb Z^d$ that converge in the sense of van Hove, i.e. sequences of graphs for which
$$
\frac{|\partial V_n|}{|V_n|}\rightarrow 0.
$$
It turns out that the limit also exists and agrees for many other sequences of graphs, including cylinders, i.e. products of paths and cycles, and tori, i.e. products of cycles, as long as the length of the shortest path or cycle diverges. This motivates the notion of sequences of tori converging to $\mathbb Z^d$. However, we emphasize that the convergence above occurs specifically for real parameters $\lambda\ge 0$.

\medskip
Independence polynomials have positive coefficients and their zeros therefore never lie on the positive real axis. The location of the complex zeros is however closely related to the behavior of the normalized limit $\rho(\lambda)$. Let $G_n$ be again a sequence of graphs converging to $\mathbb Z^d$ in the sense discussed above. Yang and Lee~\cite{YangLee} showed that if there exists a zero-free neighborhood of the parameter $\lambda_0 \ge 0$,  then $\rho$ is analytic near $\lambda_0$. 
In the other direction, knowledge of the distribution of the zeros can be used to characterize the regularity near phase transitions: parameters $\lambda_0$ where the free energy is not analytic.

As remarked above, the limit behavior on the positive real axis of the normalized logarithm of the independence polynomials is to a large extent independent from the sequence of graphs. 
The motivating question for this work is to what extent this remains true for the distribution and location of the complex zeros of the independence polynomial. In particular we focus on the question whether the zero sets are uniformly bounded or not.

Jauslin and Lebowitz~\cite{JauslinLebowitzhigh} and He and Jauslin~\cite{he2024high} showed that for a variety of non-sliding hard-core particle models on $\mathbb{Z}^d$ (which can be thought of as ordinary hard-core models on a graph obtained from $\mathbb{Z}^d$ by adding certain edges) the pressure is analytic in $z=1/\lambda$ near $z=0$, which according to~\cite[Section 1.4]{JauslinLebowitzhigh} implies that the zeros of the independence polynomials of the associated finite graphs are bounded.
It was shown by Helmuth, Perkins and the last author~\cite{PirogovSinaiWillGuusTyler} that for sequences of \emph{padded} induced subgraphs of $\mathbb Z^d$ the zeros of their independence polynomials are uniformly bounded. We recall that an induced subgraph of $\mathbb Z^d$ is said to be padded if all of its boundary points share the same parity.



Our main result shows that the boundedness of zeros for tori requires additional assumptions on the `balancedness' of the tori. 
We note that in the statistical physics literature it is shown that various other properties of models also show a dependence on a notion of balancedness of the volumes, see for example~\cite{BorgImbireunbalanced,Beauunbalanced}.

\subsection{Implications for efficient approximation algorithms}

The distribution of zeros of the independence polynomial is not only closely related to the analyticity of the limiting free energy, but also to the existence of efficient algorithms for the approximation of the independence polynomial.
Indeed, let $\mathcal{G}$ be a class of bounded degree graphs. 
Then if $Z_G(\lambda)\neq0$ for all $G\in \mathcal{G}$ and $\lambda$ in some open set $U$ containing $0$, then by Barvinok's interpolation method~\cite{BarBook} and follow up work of Patel and the last author ~\cite{PatelRegts17} there exists an algorithm that for each $\lambda\in U$ and $\varepsilon>0$ computes on input of an $n$-vertex graph $G$ from $\mathcal{G}$ a number $\xi$ such that 
\begin{equation}
 |\xi-Z_G(\lambda)|\leq \varepsilon |Z_G(\lambda)|\label{eq:approx}   
\end{equation} in time polynomial in $n/\varepsilon$. 
Such an algorithm is called a \emph{Fully Polynomial Time Approximation Scheme} or \emph{FPTAS} for short.


Recently, Helmuth, Perkins and the last author~\cite{PirogovSinaiWillGuusTyler} were able to extend this algorithmic approach to zero-free regions that do not contain the point $0$, but rather the point $\infty$, for certain subgraphs of the integer lattice. See also~\cite{JKP20} for extensions of this to other families of  bounded degree graphs.
The algorithmic results from~\cite{PirogovSinaiWillGuusTyler} also apply to the torus with all side lengths equal and of even length $n$, but the resulting algorithm is technically not an FPTAS, since it restricts the choice of $\varepsilon$ to be at least $e^{-c n}$ for some constant $c>0$.
The results of the present paper allow to remedy this and moreover extend it to non-positive evaluations and the collection of all balanced tori (at the cost of decreasing the domain).

The following result is almost a direct corollary of our main result combined with the algorithmic approach from~\cite{PirogovSinaiWillGuusTyler}; we will provide details for its proof in Section~\ref{sec:algorithms}.
\begin{prop}\label{prop:algorithm}
Let $d\in \mathbb{Z}_{\geq 2}$ and let $\Balancedtori_d$ be a family of balanced even $d$-dimensional tori. 
Then there exists a $\Lambda>0$ such that for each $\lambda\in \mathbb{C}$ with $|\lambda|>\Lambda$ there exists an FPTAS for approximating $Z_\mathcal{T}(\lambda)$ for $\torus \in \Balancedtori_d$.
\end{prop}

The interpolation method crucially depends on there being an open set not containing any zeros of the independence polynomial for graphs in the given family.
There is essentially no way to circumvent this, at least for the family of all graphs of a given maximum degree $d\geq 3$, $\mathcal{G}_d$. 
Indeed, it was shown in~\cite{BGGS20,BGGS23,de2021zeros} that the closure of the set of $\lambda\in \mathbb{C}$ for which approximating the evaluation of independence polynomial at $\lambda$ (in the sense of~\eqref{eq:approx}) is computationally hard (technically \#P-hard) contains the closure of the set $\lambda \in \mathbb{C}$ for which there exists a graph $G\in \mathcal{G}_d$ such that $Z_G(\lambda)=0$.
It would be interesting to see to what extent such a result hold for more restricted families of bounded degree graphs. 
We suspect that, by slightly enlarging the family of highly unbalanced tori, using the techniques of~\cite{de2021zeros,bencs2022approximating}, it can be shown that approximating the evaluation of the independence polynomial at large $\lambda$ for graphs in this family (in the sense of~\eqref{eq:approx}) is as hard as computing the evaluation exactly. 

\subsection{Proof techniques}

The proof of the main theorem relies upon two different techniques: the zero-freeness for balanced tori relies heavily on Pirogov-Sinai theory, while the existence of unbounded zeros for highly unbalanced tori uses the transfer-matrix expression of the independence polynomial on tori.

\subsubsection{Pirogov-Sinai theory} 

Intuitively, for $\lambda \in \mathbb{C}$ with large norm, the value of independence polynomial in $\lambda$ is determined by the large independent sets. 
Pirogov-Sinai theory builds on this intuition \cite{PirogovSinai}. The main idea is to study the independence polynomial as deviations from the maximal independent sets. For even tori, there are two distinct largest independent sets, one containing the even vertices of the torus and the other containing the odd vertices. 
The vertices where an independent set locally differs from one of these maximal independent sets will be part of so-called \emph{contours}. 
The use of contours goes back to Peierls \cite{peierls_1936} and was further developed by Minlos and Sinai in \cite{MinlosSinai1} and \cite{minlos1967phenomenon}, both originally for the Ising model. 
Using ideas from Pirogov-Sinai theory the independence polynomial can be expressed as a partition function of a polymer model, similar to as was done in \cite{PirogovSinaiWillGuusTyler} for padded regions in $\mathbb Z^d$, where the polymers will be certain sets of contours. 
One of the challenges in this rewriting is posed by the geometry of the torus. 
In particular, contours that wrap around the torus pose difficulties.
We deal with this by defining a suitable compatibility relation we call \emph{torus-compatibility} and by splitting the contours up in two `small' and `large' contours that we analyze separately, exploiting the symmetry of the torus.
This is similar in spirit as to how Borgs Chayes and Tetali~\cite{BCTtorpid} use Pirogov-Sinai theory for the ferromagnetic Potts model on the torus, although our notions of what it means for a contour to be `large' differ. See also~\cite{BCHPTalltemperatures} in which Borgs, Chayes, Helmuth, Perkins and Tetali build on~\cite{BCTtorpid} to develop an efficient algorithm for approximating the partition function of the Potts model on the discrete torus for all temperatures.

In our analysis, we apply Zahradn\'ik's truncated-based approach to Pirogov-Sinai theory \cite{Zahradnk1984AnAV}, and take inspiration from its usage by Borgs and Imbrie in \cite{BorgsImbrie}.
The idea of this approach is to first restrict the polymer partition function to well-behaved contours, so-called \emph{stable contours}. Then one applies the theory to the truncated partition function and with the estimates that follow one shows in fact all contours are stable, obtaining bounds for the original polymer partition function.


\subsubsection{Transfer-matrices}The transfer-matrix method, introduced by Kramers and Wannier \cite{kramers1941statistics1,kramers1941statistics2}, can be applied to rewrite the partition function of a one-dimensional lattice. It is heavily used in the literature to obtain both rigorous results and numerical approximations regarding the accumulation of zeros on physical parameters for other models; see for example \cite{Onsager,shrock2001,SalasSokal,shrock2009,shrock2015}.

In our setting we fix even integers $n_1, \dots, n_{d-1}$ and consider the sequence of tori $\torus_n = \mathbb{Z}_{n_1} \times \cdots \times \mathbb{Z}_{n_{d-1}} \times \mathbb{Z}_n$. The transfer-matrix method allows us to write the independence polynomials of these tori as
\[
    Z_{\torus_n}(\lambda) = \mathrm{Tr}\left(M_\lambda^{n}\right).
\]
Here $M_\lambda$ is a matrix whose entries are indexed by independent sets of the fixed torus $\mathbb{Z}_{n_1} \times \cdots \times \mathbb{Z}_{n_{d-1}}$ and contain monomials in $\lambda$; see Section~\ref{sec: transfer-matrix method} details. If we denote the eigenvalues of $M_\lambda$ by $e_1(\lambda), \ldots, e_N(\lambda)$, the above equation translates to
\[
Z_T(\lambda) = e_1(\lambda)^n + \cdots + e_N(\lambda)^n.
\]
For $|\lambda|$ large we will show that there are two simple eigenvalues, which we denote by $q^+$ and $q^-$, of approximately the same norm that dominate the remaining eigenvalues. Normality arguments then give a relatively quick proof that zeros of $\{Z_{\torus_n}\}_{n \geq 1}$ accumulate at $\infty$. We note that this type of argument can be seen as a special case of Theorem 1.5 from Sokal~\cite{Sokaldense}, where the existence of multiple dominant functions are used to deduce the accumulation of zeros; in that case the zeros of the chromatic polynomial. See also \cite{BKW78} for other related results.

The normality argument does not give any bounds on how large $n$ has to be with respect to $(n_1, \dots, n_{d-1})$ to obtain zeros of a certain magnitude. We will more thoroughly investigate the eigenvalues of $M_\lambda$, and in particular $q^{\pm}$, to explicitly describe such bounds. These bounds imply the unboundedness of zeros for highly unbalanced tori.

\subsection{Questions for future work}

When it comes to describing the complex zeros of the independence polynomials for graphs that converge to the integer lattice, the results in this paper barely touch the surface and raise a number of interesting questions. A first issue, already addressed above, is to close the gap between balanced and highly-unbalanced tori. 

Several steps of the proof for boundedness of zeros of balanced tori rely in an essential way on the assumption that the tori are balanced. On the other hand, the highly-unbalanced assumption on the family of tori that guarantees the existence of unbounded zeros seems far from sharp, evidenced for example by the fact that the demonstrated zeros of the tori escape very rapidly in terms of the sizes of the tori. It therefore seems reasonable to expect that the balanced assumption is necessarily, while the highly-unbalanced assumption is not.

\begin{question}
    Let $\mathcal{F}$ be a family of even $d$-dimensional tori for which the zeros of the independence polynomials are uniformly bounded. Is $\mathcal{F}$ necessarily balanced?
\end{question}

As discussed above, there are many other natural families of graphs that converge to the integer lattice, in the sense that the free energy per site converges to the same limit. Knowing that for families of induced subgraphs of $\mathbb Z^d$ with padded boundaries the zeros are automatically uniformly bounded, while for tori an additional assumption is required, it would be interesting to have a more general criterion that guarantees boundedness of the zero sets. 

\begin{question}
    Let $\mathcal{F}$ be a family of graphs for which the free energy per site exists and agrees with the free energy per site of $d$-dimensional balanced even tori. Under which conditions are the zeros of the independence polynomials uniformly bounded? Of particular interest are graphs with boundaries that are not necessarily padded, such as rectangles and cylinders.
\end{question}


Further questions are discussed in the next section, where we discuss the relationship between partition functions on sequences of related graphs, and iteration of rational maps in the complex domain. While this relationship does not seem to play a strong role in the new results discussed in this paper, the relationship with holomorphic dynamics is very prominent in new developments in the field and functioned as an inspiration for the current paper.

\subsection{Organization of the paper}

In Section \ref{sec:dynamics} we give a thorough discussion of the relationship with holomorphic dynamics and discuss further open questions inspired by this relationship. Section \ref{sec:PirogovSinai} provides a self-contained background in Pirogov-Sinai theory used to prove boundedness of zeros for balanced tori in Section \ref{sec:boundedzerosbalancedtori}. 
We prove the unboundedness of zeros for highly unbalanced tori in Section \ref{sec:unboundedzeros}. In Section \ref{sec:algorithms} we finish by proving implications to efficient approximation algorithms for balanced tori.

\medskip

\noindent {\bf Acknowledgment} The authors are grateful to Ferenc Bencs for inspiring discussions related to the topic of this paper. They are moreover thankful to the referees for useful comments and to Aernout van Enter for suggesting useful references.

\section{Relationship with holomorphic dynamics}\label{sec:dynamics}

For a graph $G = (V,E)$ and a vertex $v \in V$, we denote by $G-v$ and $G-N[v]$ the graphs obtained respectively by removing the vertex $v$ or the vertex and all its neighbors. Since each independent subset either contains $v$ or does not, we obtain the following fundamental equation:
\begin{equation}\label{eq:fundamental}
Z_G(\lambda) =  \lambda \cdot Z_{G-N[v]}(\lambda) + Z_{G-v}(\lambda).
\end{equation}
Together with the fact that the independence polynomial of an empty graph is $1$, this recursive formula can be used as the definition of the independence polynomial. However, for large arbitrary graphs the recursive formula does not provide an efficient means of computing the independence polynomial: the number of graphs that can occur by repeatedly removing either a vertex or a vertex-neighborhood can grow exponentially fast in terms of the number of vertices.

However, when a graph exhibits a clear recursive structure, the above formula can provide a simple recursive formula for computing the independence polynomial, and the description of the zeros of the sequence of independence polynomials can often be translated to a question concerning rational dynamical systems. Let us discuss two explicit examples of graphs with recursive structures.

\subsection{Rooted $d$-ary trees}

In what follows fix an integer $d \ge 2$. We define the sequence of rooted $d$-ary trees $T_n$ by letting $T_0$ be a single vertex, which is of course also its root, and recursively defining $T_{n+1}$, by taking $d$ copies of $T_n$, adding a new root vertex $v$, and adding an edge between $v$ and all the root vertices $v_1, \ldots, v_d$ of the $d$ copies. Alternativelty we can characterize the finite tree $T_n$ by asking that each non-leaf of $T_n$ has exactly $d$ children, and that each leaf has distance $n$ to the root vertex.

The recursive definition of $T_n$ immediately gives rise to a recursive formula: let us write 
$$
x_n = Z_{T_n - v}(\lambda) \; , \quad \mathrm{and} \quad y_n = Z_{T_n - N[v]}(\lambda)
$$
Then it follows that
$$
\left(\begin{matrix}
    x_{n+1}\\
    y_{n+1}
\end{matrix}\right) = 
\left(\begin{matrix}
    (x_n  + y_n)^d\\
    \lambda \cdot x_n^d
\end{matrix}\right),
$$
with initial values $x_0 = 1$ and $y_0 = \lambda$. We are interested in the parameters for which $x_n + y_n = 0$. Since this line and the recursive relation are both given by homogeneous equations in $x$ and $y$, we may work in projective space $\mathbb P^1$ instead of $\mathbb C^2$. Thus we can work with the variable $r_n = y_n / x_n$, which induces the rational function
$$
f_{\lambda,d}(r) = f_\lambda(r) = \frac{\lambda}{(1+r)^d}.
$$
The variable $r$ is often referred to as the occupation ratio: it describes the odds of the vertex $v$ lying inside or outside of a randomly chosen independent subset.

The $\lambda$-values for which the independence polynomial is zero correspond to the parameters $\lambda$ for which 
$$
f_\lambda^n(\lambda) = f_\lambda^{n+1}(0) = -1.
$$
Since $f$ has two critical points $0$ and $\infty$, and $0 = f_\lambda(\infty)$, the orbit of $0$ is the only critical orbit. Observing that $f_\lambda^2(-1) = 0$, we note that the zeros of $Z_{T_n}$ correspond to the parameters $\lambda$ for which the rational function $f_\lambda$ has a super-attracting periodic cycle of period $n+3$. Just as for the family of quadratic polynomials $z^2 + c$, the super-attracting parameters are the unique centers of the hyperbolic components, except for the hyperbolic component centered at $0$ which contains no zeros of the independence polynomial.

As the period $n$ of these cycles converges to infinity, the probability measures that place equal masses on each of the super-attracting parameters of period $n$ converge weakly to what is called the bifurcation measure, see~\cite{DeMarco2001}. The support of this measure is the bifurcation locus: it is the maximal set where the family $(f_\lambda)$ does not locally form a normal family. An illustration of the zeros of the independence polynomial and the bifurcation locus of $f_c$ can be found in Figure \ref{figure:degree3}.

The hyperbolic component centered at $0$, which we will call the main hyperbolic component, contains all the parameters for which $f_\lambda$ has an attracting fixed point. By the Fatou-Julia Lemma, the unique critical orbit must converge to this attracting fixed point, hence the orbit of $-1$ converges to the attracting fixed point and cannot pass through $-1$, which implies this component is zero-free. Just as for the Main Cardioid of the Mandelbrot set, which contains the parameters $c$ for which $z^2 + c$ has an attracting fixed point, the boundary of the main hyperbolic component can be explicitly parameterized. The intersection of the main hyperbolic component with the real axis coincides with the open interval bounded by the critical parameters
$$
\lambda_{-} = \frac{-d^d}{(d+1)^{d+1}}\; \quad \mathrm{and} \quad \lambda_{+} = \frac{d^d}{(d-1)^{d+1}}.  
$$
As these two parameters lie on the bifurcation locus, the zeros of the independence polynomial on $d$-ary trees accumulate on $\lambda_-$ and $\lambda_+$. The idea that the $d$-ary trees are in some sense extremal led Sokal to conjecture that the real interval $(\lambda_-, \lambda_+)$ has an open neighborhood in $\mathbb C$ that is zero-free for all graphs of degree at most $d+1$ \cite{sokal2001personal}. This conjecture was confirmed in \cite{PetersRegts2017}.

A natural follow-up question was to describe the complex zeros for the family of bounded degree graphs in relationship with the main hyperbolic component of the map $f_d$. It follows from the results in \cite{BGGS20,de2021zeros} that the zeros are dense outside of the main hyperbolic component. In fact, it was shown in \cite{de2021zeros} that the accumulation set of the complex zeros corresponds to the non-normality locus of the family of occupation ratios, just as for the family of $d$-ary trees. In \cite{buys20} it was however shown that the zeros do intersect the main hyperbolic component. The papers cited here all rely in essential way on using dynamical ideas to control the zero sets of graphs, but not necessarily on iteration of a single function $f_{\lambda,d}$. For example, to show that there are graphs for which the zeros lie in the main hyperbolic component, different elements in the semi-group generated by the functions $f_{\lambda,1}, \ldots, f_{\lambda,d}$ were considered. By results of Ma\~n\'e-Sad-Sullivan \cite{ManeSadSullivan1983}, iterates of a composition
$$
g_\lambda = f_{\lambda, d_n} \circ \cdots \circ f_{\lambda, d_1}
$$
cannot form a normal family near a parameter $\lambda_0$ where $g_\lambda$ has a non-persistent neutral fixed point, see also \cite[Theorem 4.2]{McMullen1994}. It was shown in \cite{buys20} that for suitable compositions the curves in parameter space containing the parameters for which there exists a neutral fixed point may intersect the main hyperbolic component. In particular, zeros do accumulate at some points in the main hyperbolic component.

The exact maximal zero-free domain for the families of graphs of bounded degree $d+1$ has not been characterized, and perhaps it is unrealistic to expect an analytic characterization of these domains. It was shown in \cite{Bencs2021} that as $d\rightarrow \infty$, proper rescalings of the maximal zero-free domains do converge to a limit domain, which can be studied via the dynamics of (a semi-group of) exponential functions $\mathrm{Exp}(c\cdot z)$, and that the limit domain is strictly smaller than the analogous limit of the main hyperbolic components. A significant part of the boundary of the limit domain can be explicitly parameterized, and it seems likely that, away from the intersection with the negative real axis, the boundary of the limit domain is piecewise real analytic.

For the sequence of $d$-ary trees, the induced dynamical systems were successfully exploited \footnote{The results mentioned in this paragraph have been claimed by Rivera-Letelier and Sombra~\cite{RiveraLetelier2019}, but a publication is at the time of this writing not available. Our citation contains a link to the recording of a presentation at the Fields Institute by Rivera-Letelier; a presentation containing main results and outlines of techniques used in the proof. Some of these steps have been worked out in more detail in the master's thesis of Jos van Willigen~\cite{Willigen}.} to describe phase transiions: their locations and their order, i.e. the order $r$ for which the free energy per site $\rho(\lambda)$ is not real-analytic but still $C^r$-smooth. In unpublished work of Rivera-Letlier and Sombra \cite{RiveraLetelier2019} it was claimed that for the sequence of $d$-ary trees, for fixed $d$, the phase transition at $\lambda_+$ is of infinite order: the free energy per site $\rho(\lambda)$ defined in equation~\eqref{eq:free energy} is $C^\infty$ but not real analytic at $\lambda_+$. The result relies upon careful estimates on the number of super-attracting parameters in the proximity of $\lambda_+$.

\begin{figure}
\includegraphics[width=0.45\textwidth]{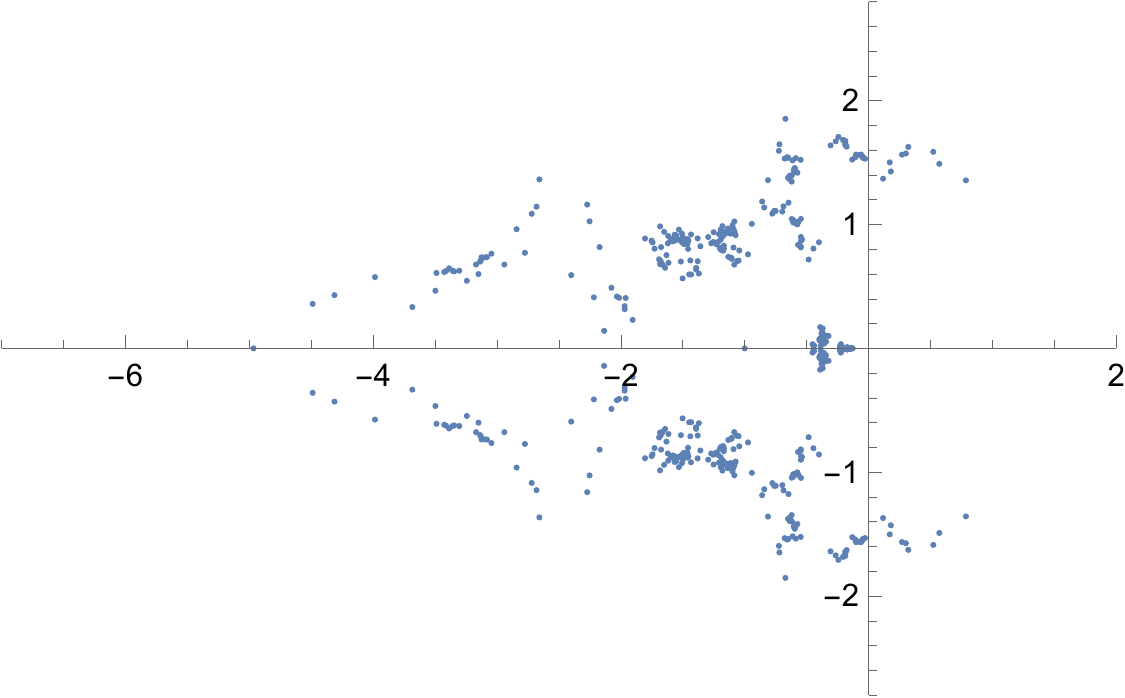}
\includegraphics[width=0.45\textwidth]{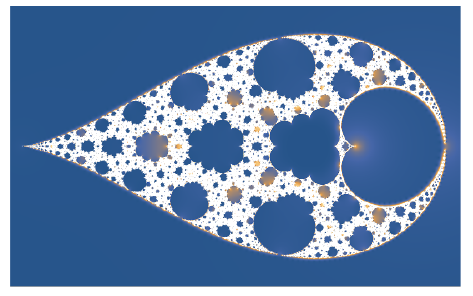}
\caption{On the left the zeros of partition function of the the $d$-ary trees, and on the right the bifurcation diagram of the functions $f_{\lambda,d}$, for $d=3$. The roots were computed by first exploiting the recursive structure to find an exact formula for the independence polynomial $Z_{T_n}$, and then algebraically solving for the roots of the equation $Z_{T_n} = 0$. The degree and especially the size of the coefficients of these polynomials grow respectively exponentially and super-exponentially with the depth $n$: We have plotted the zeros for depth $n=6$, in which case the independence polynomial has degree $820$, and the coefficients add up to roughly $10^{255}$. The number $n=6$ is clearly far too small to obtain a good idea what happens as $n \rightarrow \infty$. In particular the left hand side plot sheds little light on the known phase transition at $\lambda_+=27/16$.\\
The bifurcation diagram on the right is generated by computing the approximate values of $f^n(\lambda)$ for a discrete raster, and using nearby values to find a discrete approximation of the spherical derivative. Computationally this method is far more efficient. The illustration on the right depicts the bifurcation diagram for $n=200$, and the phase transition at $\lambda_+$ is clearly visible as the right-most point where the bifurcation locus intersects the real axis, which is
the horizontal line midway up the figure.}
\label{figure:degree3}
\end{figure}

It is tempting to say that this result describes the phase transition of the infinite $d$-ary tree, but we emphasize that the function $\rho(\lambda)$ is defined as the limit of a sequence of functions, each defined in terms of a single finite graphs. For sequences of graphs that converge in a suitable topology to a specific infinite limit, such as the cubic lattice, it can be shown that the free energy per site only depends on the limit graph, and not on the sequence of finite graphs, and in this case one can reasonably speak about phase transitions of the model on the infinite graph.

However, we note that the $d$-ary trees do not converge to this infinite $d$-ary tree in the sense of van Hove, or even Benjamini-Schramm. The $d$-ary trees can be regarded as a nested sequence of subgraphs whose union equals the infinite $d$-ary tree, but this does not warrant conclusions about the infinite union. For example, one can construct a similar sequence of nested rooted trees, called \emph{Fibonacci trees}, by starting with $G_0$ equal to a single vertex and $G_1$ equal to two vertices joined by an edge, and recursively defining $G_{n+2}$ by joining the two roots of $G_n$ and $G_{n+1}$ to a new root vertex. By writing $r_n$ for the occupation ratio of the root vertex of $G_n$ one obtains the relationship
$$
r_{n+2} = \frac{\lambda}{(1+r_{n+1})\cdot(1+r_n)}.
$$
By studying the iteration of the two-dimensional rational map
$$
(x,y) \mapsto (\frac{\lambda}{(1+x)\cdot(1+y)}, x)
$$
it was shown in~\cite{BGGS20} that for any parameter away from the negative real axis, the orbit of $(\lambda/(1+\lambda), \lambda)$ converges to a fixed point $(x,x)$, where $x$ is a fixed point of $f_\lambda$. As this fixed point varies analytically on $\mathbb R_+$, there is no phase transition.

This sequence of Fibonacci trees was used elegantly in~\cite{BGGS20} as a crucial ingredient in the proof that approximation of the independence polynomial on the family of bounded degree graphs is computationally hard outside of the main hyperbolic component. The idea is that for these trees the occupation ratio converges to a fixed point, which for $\lambda$ outside of the main hyperbolic component is actually a repelling fixed point for the map $f_{\lambda,d}$. By ``implementing'' the Fibonacci trees into the regular $d$-ary trees, one can obtain non-normality and as a result computational hardness.

A description of the phase transition of the infinite $d$-ary tree was made by Sly and Sun in \cite{Slysun}, who considered the infinite $d$-ary tree as the limit of either a sequence of random $d$-regular graphs, or a sequence of random $d$-regular bipartite graphs. 
While these sequences of random graphs portray no apparent recursive structure, the behavior of the functions $f_\lambda$ plays a role here as well. For $\lambda<\lambda_+$ the occupation ratio is to a large extend determined by the local behavior of the graphs, and for random graphs these local neighborhoods stabilize almost certainly to $d$-ary trees. It should therefore come as no surprise that for $\lambda < \lambda_+$ the free energy per site can be expressed in terms of the attracting fixed point of $f_\lambda$. For $\lambda> \lambda_+$ it is not the fixed point, but the attracting periodic $2$-cycle that plays a crucial role in determining the free energy per site, although the situation is different depending on whether the graphs are assumed to be bipartite or not. The location of the phase transition is the same as for the sequence of $d$-ary trees, but the order of the phase transition remains to be determined.

\subsection{Diamond hierarchical graphs}

Zeros of partition functions on the diamond hierarchical lattice have been studied extensively, we refer to the papers \cite{berker1979renormalisation,griffiths1982spin, derrida1983fractal, bleher1989asymptotics,BLRI, BLRII} for the case of the Ising model, and to \cite{chio2021chromatic} for the chromatic polyonomial. In either case the recursive nature of the sequence of graphs allows for a translation into a purely dynamical setting, a translation that is often referred to as Migdal-Kadanoff renormalization \cite{migdal1975recurrence,kadanoff1976notes}. We will discuss this renormalization scheme in the setting of the independence polynomial, where one similarly obtains a rational dynamical systems.

The sequence of graphs is defined as follows: start with a $4$-cycle $G_1$, with vertices labeled $1,2,3$, and $4$, with the pairs $(1,3)$ and $(2,4)$ not being connected by an edge. We choose $1$ and $3$ to be marked points.

Having constructed a graph $G_n$ with two marked points $v_n$ and $w_n$, the graph $G_{n+1}$ is obtained by replacing each edge $(v,w)$ of $G_n$ with the $4$-cycle, where the vertices $v$ and $w$ correspond to the marked points $1$ and $3$. Alternatively, the graph $G_{n+1}$ can be obtained by replacing each edge of $G_1$ with the graph $G_n$, where the two endpoints of the edge correspond to the marked vertices $v_n$ and $w_n$ of the graph $G_n$.

We introduce $3$ variables $x_n,y_n,z_n$ that are defined as follows:
$$
\begin{aligned}
x_n & = \sum_{I \subset V(G_n) \setminus \{v_n, w_n\}}\lambda^{|I|},\\
y_n & = \sum_{I \subset V(G_n) \setminus \{v_n\}: \; w_n \in I}\lambda^{|I|},\\
z_n & = \sum_{I \subset V(G_n): \; \{v_n, w_n\} \subset I}\lambda^{|I|},\\
\end{aligned}
$$
where in each case we sum only over independent subsets of $V(G_n)$. By symmetry it follows that
$$
Z_{G_n} = x_n + 2 y_n + z_n.
$$
On the other hand, it follows from the fundamental equation \eqref{eq:fundamental} that $(x_{n+1}, y_{n+1}, z_{n+1}) = F_\lambda(x_n, y_n, z_n)$, where
$$
\begin{aligned}
F_\lambda:(x,y,z) & \mapsto (x^4 + 2x^2y^2/\lambda + y^4/\lambda^2, x^2y^2 + 2xy^2z/\lambda  + y^2z^2/\lambda^2, y^4 + 2y^2z^2/\lambda  + z^4/\lambda^2)\\
&= ((x^2 + y^2/\lambda)^2, y^2(x+z/\lambda)^2, (y^2 + z^2/\lambda)^2) 
\end{aligned}
$$
We note that 
$$
\left(
\begin{matrix}
    x_1\\
    y_1\\
    z_1
\end{matrix}
\right) = \left(
\begin{matrix}
    \lambda^2 + 2\lambda + 1\\
    \lambda^2 + \lambda\\
    \lambda^2
\end{matrix}
\right).
$$
Hence the independence polynomial is given by
$$
Z_{G_n}(\lambda) = \pi \circ F^{n-1}(\lambda^2 + 2\lambda + 1, \lambda^2 + \lambda, \lambda^2) = \pi \circ F^n (1,\lambda, 0),
$$
where $\pi(x,y,z) = x+2y+z$.

The dynamical system that we have obtained is quite similar to what was obtained for ferromagnetic Ising model on diamond hierarchical lattices in \cite{BLRI, BLRII}, and for the chromatic polynomial in \cite{chio2021chromatic}. The zeros of independence polynomial correspond to the parameters for which the point $(1,\lambda, 0)$ is contained in the inverse image
\begin{equation}\label{eq:pullbacks}
F_\lambda^{-n} \{x+2y+z = 0\}.
\end{equation}

While the formula for $F_\lambda$ is different for the three different partition functions, the independence polynomial, the partition function of the Ising model, and the chromatic polynomial, there are strong similarities between the three settings. For example, for each of the three models one obtains a homogeneous polynomial $F_\lambda$. The map $F_\lambda$ does not act holomorphically on $\mathbb P^2$, for example for the independence polynomial the line through $(1,i\sqrt{\lambda},-\lambda)$ is mapped to $(0,0,0)$, which gives rise to an indeterminancy point on $\mathbb P^2$, significantly complicating the analysis of the dynamical system. We note that the occurance of indeterminancy points meant a significant complication for the analysis in the Ising model as well, see~\cite{BLRI, BLRII}.

Roughly speaking, the method used in the previous studies is to obtain equidistribution of the inverse images in equation~\eqref{eq:pullbacks}, not in terms of manifolds but in terms of $(1,1)$-currents: objects over which one can integrate $1$-forms. The distribution of the zeros is obtained by slicing the equidistribution current with the line given by multiples of $(x_1, y_1, z_1)$.

It would be interesting to see whether the methods developed for the other two models can also be applied in the setting of the independence polynomial, and whether the phase transition(s) of the independence polynomial on the hierarchical lattices can be succesfully described using methods from tools from potential theory and higher dimensional complex dynamical systems.

\subsection{Zeros and normal families in regular lattices}

From a physical perspective, the main interest in $d$-ary trees or hierarchical lattices is as models for actual physical lattices that are much harder to describe, such as the cubic lattice $\mathbb Z^d$. While $\mathbb Z^d$ also seems to have a recurring structure that repeats at different levels of the lattice, there is no clear renormalization operator that relates different scale partition functions. A characterization of the renormalization operation in terms of iteration of a possibly rational map could signify a major breakthrough in our understanding of phase transitions on regular lattices, but it remains unclear whether such characterization actually exists.

In the meantime it is worthwhile to determine whether the conclusions drawn in the literature for the models, i.e. sequences of recursive graphs, also hold for sequences of graphs that converge to $\mathbb Z^d$, in this case in the sense of van Hove. A first natural question is the following:

\begin{figure}
    \centering
    \begin{minipage}{0.48\textwidth}
    \newpage\vspace{-7pt}
        \centering
        \includegraphics[width=\textwidth]{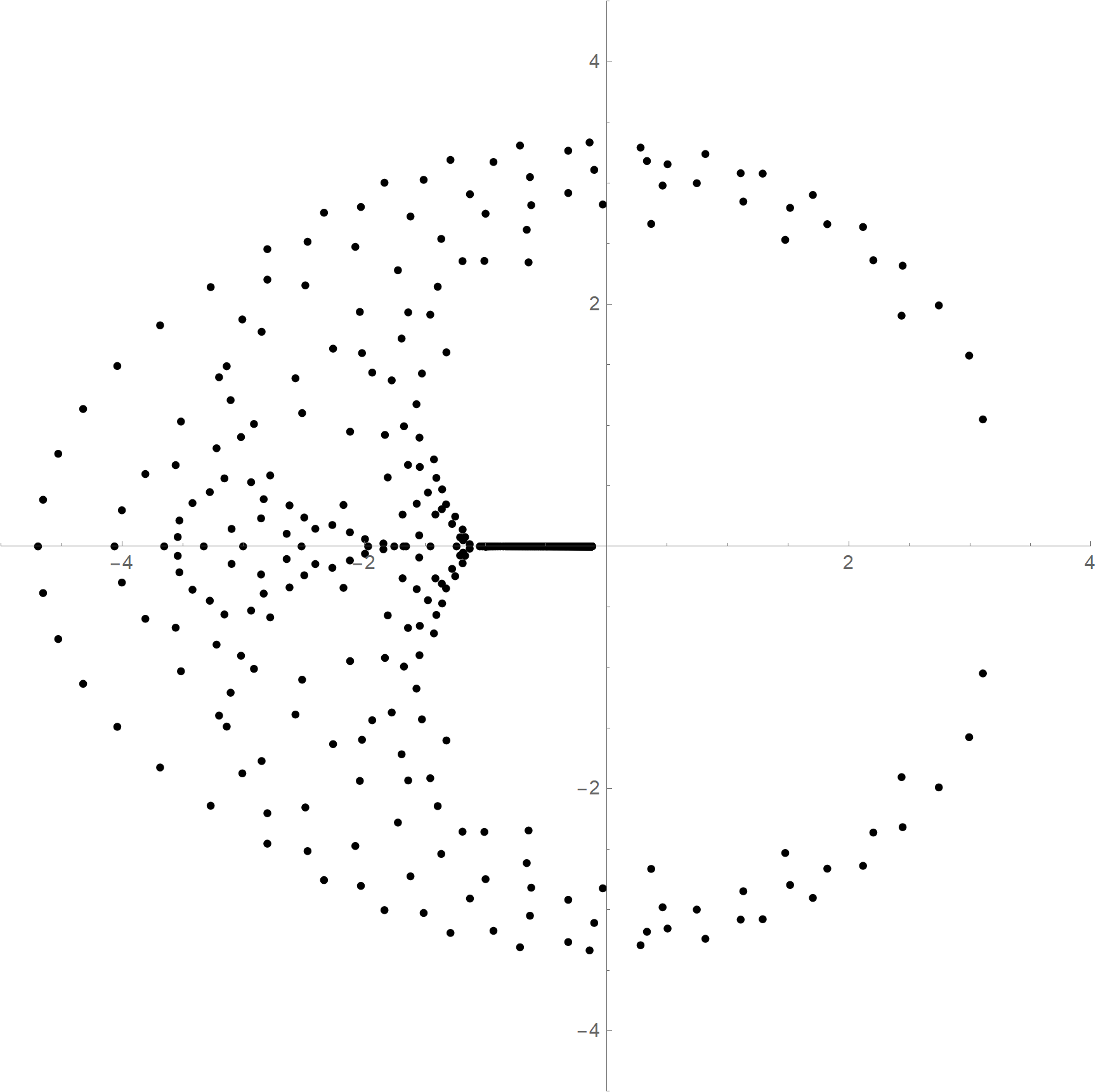}
    \end{minipage}\hfill
    \begin{minipage}{0.48\textwidth}
        \centering
        \includegraphics[width=\textwidth]{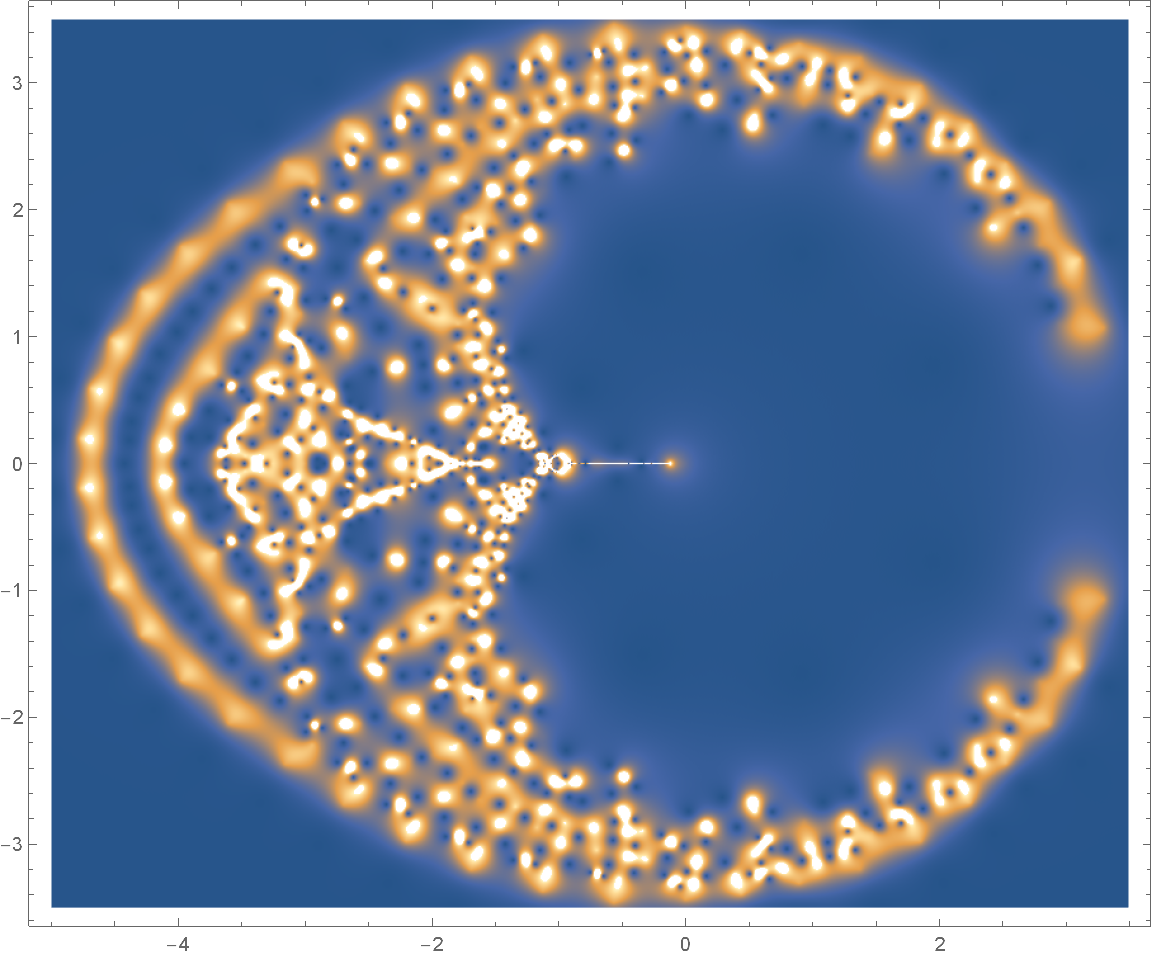}
    \end{minipage}
        \caption{The figure on the left depicts the zeros of the independence polynomial the subgraph $B_{22}(0)$ in $\mathbb Z^2$: all points in the square lattice whose distance in the graph to the origin is at most $22$. The figure on the right depicts the spherical derivative of the occupation ratio of $(B_{22}(0),0)$. }    \label{fig:diamond}
\end{figure}

\begin{question}\label{ques:one}
Let $G_n$ be a sequence of graphs that converges to $\mathbb Z^d$ in the sense of van Hove, and let $v_n \in V(G_n)$ be a sequence of marked vertices. Is it true that the accumulation set of the zeros of the independence polynomials $Z_{G_n}$ coincides with the non-normality locus of the sequence of occupation ratios $R_{G_n, v_n}$?
\end{question}

Figures \ref{fig:diamond} and \ref{fig:torus} illustrate the relationship between the zero sets and non-normality for two subgraphs of the square lattice $\mathbb Z^2$. The illustrations in \ref{fig:diamond} cover the distance-$22$ neighborhood of the point $(0,0) \in \mathbb Z^2$ in terms of the graph-distance. The figure on the right is drawn by plotting values of the (discrete approximation) of the spherical derivative of the occupation ratio, just as was done for the $d$-ary trees in Figure \ref{figure:degree3}. We note that in this setting there is no proven relationship between these two illustrations, and it is even unclear whether the zero sets will converge to the set where the sequence of spherical derivatives of the occupation ratios is unbounded as the the radius of the neighborhood converges to infinity, but to our mind Figure \ref{fig:diamond} clearly suggests a positive answer to Question \ref{ques:one}, at least for well-chosen sequences of graphs and marked vertices. The authors are grateful to Raymond van Veneti\"e for his help in writing efficient code to compute the independence polynomial of the distance-$22$ neighborhood. 

Figure \ref{fig:torus} is generated by considering the $18\times 18$ torus; not formally a subgraph of $\mathbb Z^2$ but note that the $n\times n$ tori do converge to $\mathbb Z^2$ in the sense of van Hove. Since the automorphism group of this graph is transitive it does not matter which marked vertex is chosen. Again we see a clear relationship between the distribution of the zeros on the left, and the approximation of the ``bifurcation locus'' depicted on the right, which gives further evidence for a positive answer to Question \ref{ques:one}. 

Comparing the depictions of the bifurcation diagrams on the $22$-neighborhood and the $18\times 18$-torus to that of the depth $200$ $d$-ary tree, it is clear that the right hand side diagram in figure \ref{figure:degree3} is much sharper than the diagrams in Figures \ref{fig:diamond} and \ref{fig:torus}. This is not due to an increased resolution, the resolutions are similar, but to the fact that the degree of the independence polynomial for the $d$-ary tree is of a much larger order of magnitude. The possibility of computing the independence polynomial for such large degrees is of course due to the description of the renormalization map in terms of a rational dynamical system. Nevertheless, there also is a resemblance between the three pictures: the bifurcation locus always seems to be bounded sets, the complement seems to consists of infinity many simply-connected components, and the bifurcation locus seems to intersect the positive real axis in a single parameter, the conjectured unique phase transition. We recall that all three of these properties are known to hold for the $d$-ary trees, and all three are open questions for graphs converging to lattices $\mathbb Z^d$.

The desire to obtain a rigorous basis for these observed properties has motivated this project. In what follows we focus on the boundedness of the zeros of the independence polynomials for different sequences of tori converging to $\mathbb Z^d$.


\begin{figure}
    \centering
    \begin{minipage}{0.50\textwidth}
    \newpage\vspace{-5pt}
        \centering
        \includegraphics[width=\textwidth]{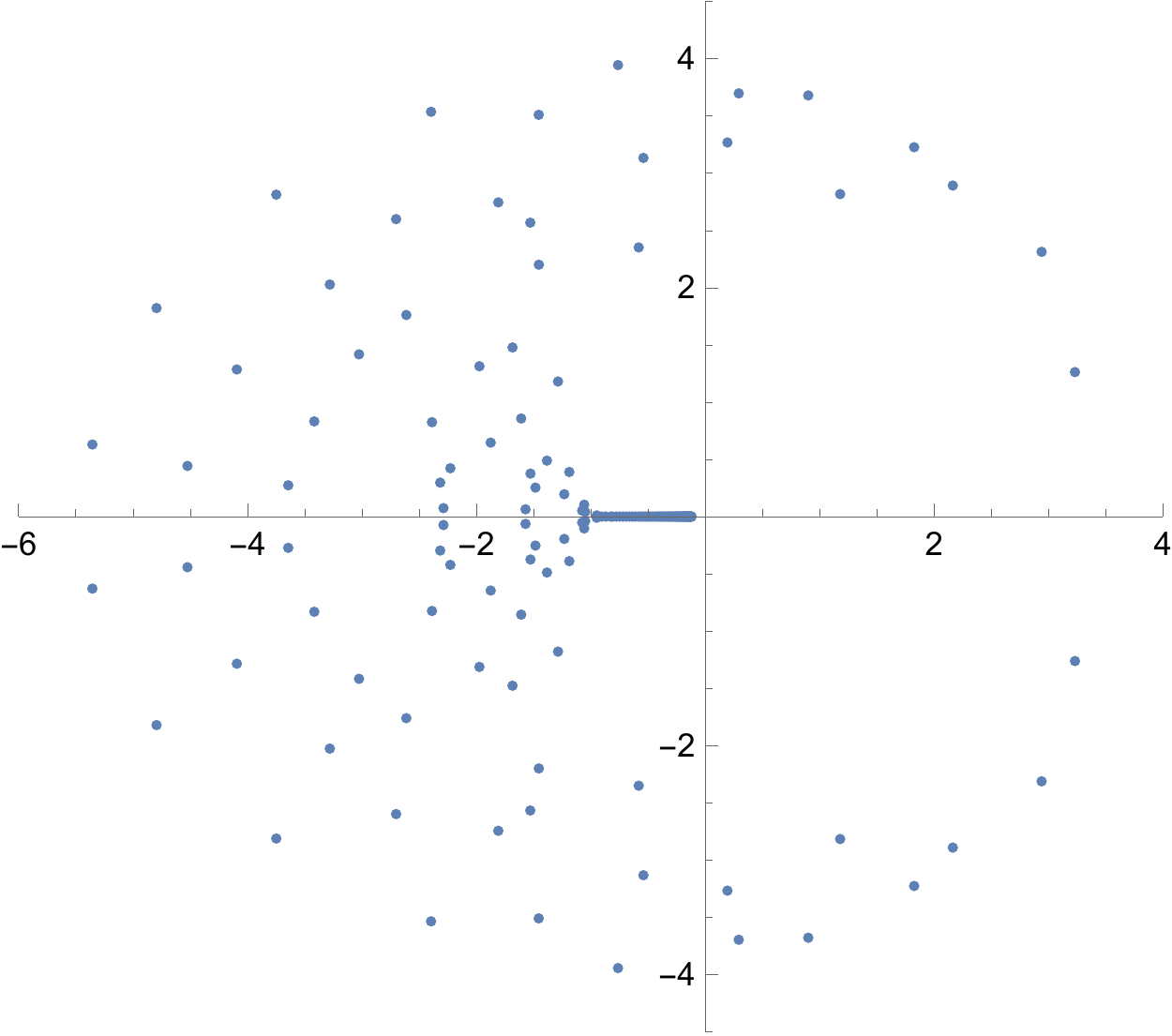}
    \end{minipage}\hfill
    \begin{minipage}{0.48\textwidth}
        \centering
        \includegraphics[width=\textwidth]{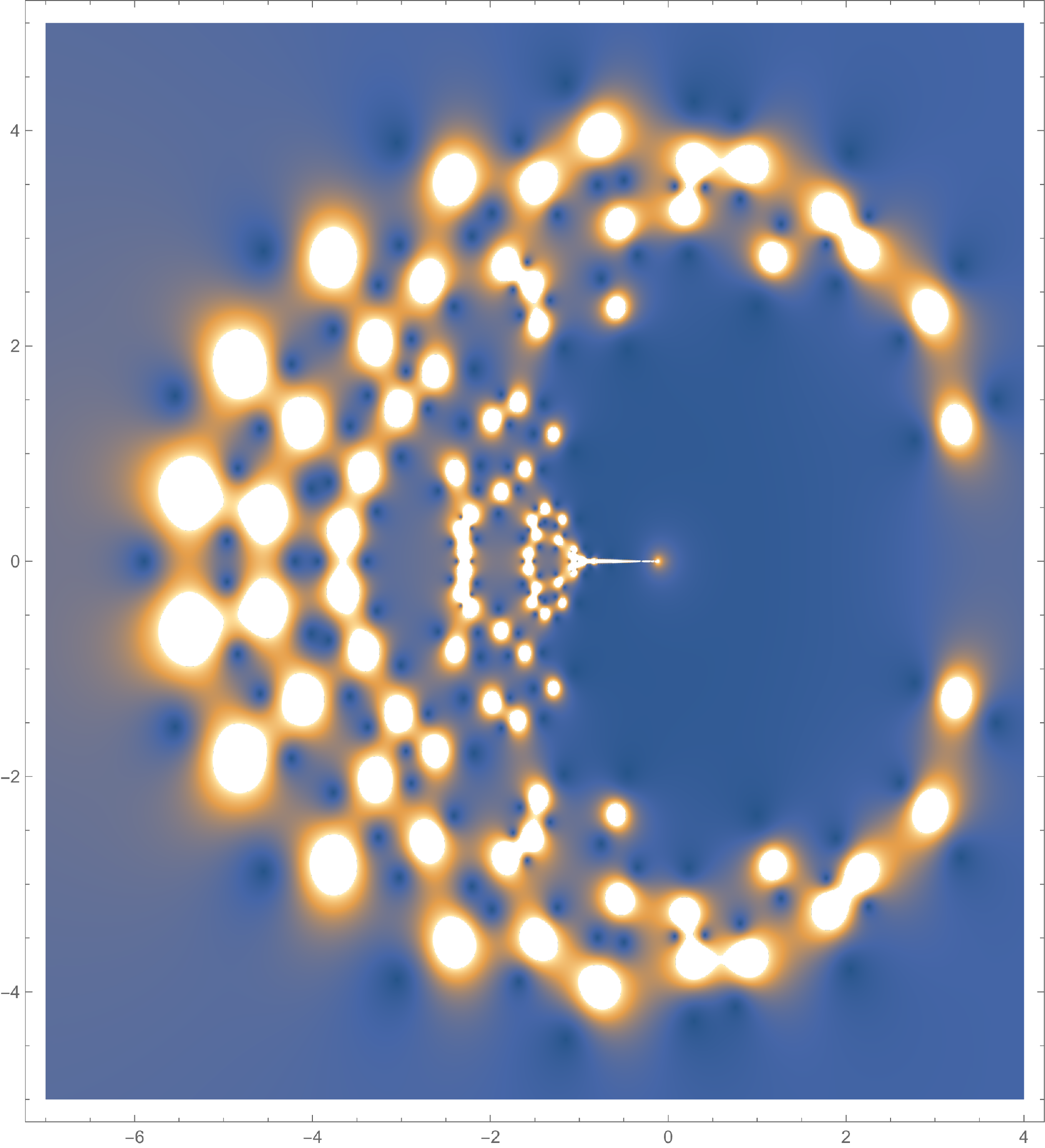}
    \end{minipage}
        \caption{The figure on the left depicts the zeros of the independence polynomial of the $2$-dimensional torus of size $18\times 18$. The figure on the right depicts the spherical derivative of the occupation ratio of this torus.}    \label{fig:torus}
\end{figure}

\section{Pirogov-Sinai theory}\label{sec:PirogovSinai}

This section provides a self-contained background in Pirogov-Sinai theory. We closely follow the framework of \cite{PirogovSinaiWillGuusTyler}, but apply it to the independence polynomial of tori, which requires several adjustments.
While much of this background section is classical, proofs in the literature are often omitted or stated in a different context. For this reason this section contains several results and proofs that are not original but may be difficult to find in the literature.

In subsection \ref{sec: cluster expansion} we discuss contains the required background on polymer partition functions. in what follows we rewrite the independence polynomial of the torus as a suitable polymer partition function.


\subsection{Polymer models and the Koteck\'y-Preiss theorem}\label{sec: cluster expansion}

A polymer model consists of a finite set of polymers $S$, an anti-reflexive and symmetric compatibility relation $\sim$ on $S$ and a weight function $w:S \to \mathbb{C}$.
We define $\SetOfContours$ to be the set of collections of pairwise compatible polymers. The \emph{polymer partition function} is defined as
\[
Z_{\text{pol}} := \sum_{\Gamma \in \SetOfContours} \prod_{\gamma \in \Gamma} w(\gamma).
\]
We note $\emptyset \in \SetOfContours$, hence if $w(\gamma) = 0$ for all $\gamma \in S$ we see $Z_{\text{pol}} = 1$. 

\begin{remark}
    Whenever $f$ is a holomorphic function with $f(0) > 0$, we write $\log f(z)$ for a branch with $\log f(0) \in \mathbb R$. We will use this convention throughout the paper. 
\end{remark}

The \emph{cluster expansion}, see for example \cite{kotecky1986cluster} and Section 5.3 in \cite{friedli_velenik_2017}, states that the polymer partition function can be expressed in terms of the following formal power series in the weights:
\begin{equation}\label{eq:clusterexpansionorderedobjects}
\log Z_{\text{pol}} = \sum_{k\geq 1}   \frac{1}{k!} \sum_{(\gamma_1, \ldots, \gamma_k)} \psi (\gamma_1, \ldots, \gamma_k) \prod_{i=1}^k w(\gamma_i),
\end{equation}
where the sum runs over ordered $k$-tuples of polymers and $\psi$ is the \emph{Ursell function} defined as follows. Let $H$ be the \emph{incompatibility graph} of the polymers $\gamma_1, \ldots, \gamma_k$, i.e. the graph with vertex set $[k]$ and an edge between $i$ and $j$ if $\gamma_i$ is incompatible with $\gamma_j$. Then
\[
\psi(\gamma_1, \ldots, \gamma_k) := \sum_{\substack{E \subseteq E(H)\\ \text{spanning, connected}}} (-1)^{|E|}.
\]
A multiset $\{\gamma_1, \ldots, \gamma_k\}$ of polymers is  a \emph{cluster} if its incompatibility graph is connected. 
For a cluster $X:= \{\gamma_1, \ldots, \gamma_k\}$ of polymers we define
\[
\Phi(X) = \prod_{\gamma \in S} \frac{1}{n_X(\gamma)!} \psi(\gamma_1, \ldots, \gamma_k) \prod_{i=1}^k w(\gamma_i),
\]
with $n_X(\gamma)$ the number of times the polymer $\gamma$ appears in $X$. Then one sees the cluster expansion \eqref{eq:clusterexpansionorderedobjects} can be equivalently written as
\begin{equation}\label{eq:onlyclustersclusterexpansion}
    \log Z_{\text{pol}} = \sum_{X \text{cluster}} \Phi(X).
\end{equation}

The Koteck\'y-Preiss theorem provides a sufficient condition on the weights that guarantees convergence of the cluster expansion, see Theorem 1 in \cite{kotecky1986cluster}:

\begin{theorem}\label{thm:KoteckyPreiss}
Suppose there are functions $a: S \rightarrow [0,\infty)$, $b: S \rightarrow [0,\infty)$ such that for every polymer $\gamma \in S$ we have
\[
\sum_{\gamma'\not \sim \gamma } |w(\gamma')| e^{a(\gamma') + b(\gamma')} \leq a(\gamma),
\]
then $Z_{\text{pol}} \neq 0$, the cluster expansion of the polymer partition function is convergent and for any polymer $\gamma \in S$ we have
\[
\sum_{X\not \sim \gamma } |\Phi(X)| e^{b(X)} \leq a(\gamma),
\]
where for a cluster $X$ we define $b(X) = \sum_{\gamma \in X} b(\gamma)$ and we write $X\not \sim \gamma $ if and only if there is a $\gamma'\in X$ with $\gamma'\not \sim \gamma$.
\end{theorem}

\subsection{Contour representation of the independence polynomial on tori}\label{sec: contour representation}

In this section we express the independence polynomial of a torus as a sum of two polymer partition functions, using contours as polymers. 
This is based on ideas and definitions from Pirogov-Sinai theory \cite{PirogovSinai, Zahradnk1984AnAV}, as applied to the independence polynomial in \cite{PirogovSinaiWillGuusTyler}. 
In \cite{PirogovSinaiWillGuusTyler} the contour models are defined for padded induced subgraphs in $\mathbb{Z}^d$; we will modify the ideas and definitions such that they apply to tori instead. 

\subsubsection{Preliminaries on the topology of tori}

\begin{definition}
We denote a $d$-dimensional torus by $\torus := \mathbb{Z}_{\ell_1} \times \cdots \times \mathbb{Z}_{\ell_d}$ with $ \ell_1 \leq \ell_2 \leq \cdots \leq \ell_d$ and write $|\torus| :=\prod_{i=1}^d \ell_i$.
Let $C>0$. A torus $\torus$ is said to be \emph{$C$-balanced} if $|\torus| \le e^{C \ell_1}$. We denote the set of all even $d$-dimensional $C$-balanced tori by $\Balancedtori_d(C)$.
\end{definition}

Note that a family of $d$-dimensional tori is balanced as defined in the introduction if and only if there exists a uniform $C>0$ such that every torus in the family is $C$-balanced.
In particular the $d$-dimensional torus with equal side lengths, denoted by $\mathbb{Z}_n^{d}$, is in $\Balancedtori_d(1)$ for any $d$ and any even $n \geq e^d$. 

We label the vertices of $\torus$ as $(v_1, \ldots v_d)$ with $v_i\in \{-\ell_i/2, \ldots, 0, \ldots, \ell_i/2 -1\}$ for each $i \in [d]$. 
Denote the $d$-dimensional zero vector by $\vec{0}$.
Throughout this and later sections we let $\torus$ be an even $d$-dimensional torus, for a fixed $d \ge 2$. When $\torus$ is assumed to be balanced or unbalanced we will state so explicitly.

\begin{definition}
     We denote the $\infty$-neighborhood of $v \in \torus$ by 
    \[
        N_\infty[v]:=\{u \in \mathcal{T}: \|v-u\|_\infty \leq 1\}.
     \]
     Note that each neighborhood $N_\infty[v]$ consists of $3^d$ distinct vertices.
     We say an induced subgraph $\Lambda \subset \torus$ is $\infty$-connected if for each $u,v \in \Lambda$ there is a sequence $(v_0, \ldots, v_n)$ of vertices in $\Lambda$ such that $v_0 = v$, $v_n = u$ and for each $i$ we have $v_{i+1} \in N_\infty[v_i]$. Such a sequence is called an $\infty$-path from $v$ to $u$ in $\Lambda$. 
\end{definition}

We denote the set of connected components of a graph $G$ by $\mathcal{C}(G)$.

\begin{definition}
    For subsets $A,B \subseteq V(\torus)$ we define their \emph{distance} and the \emph{box-diameter} as
    \[
    \dist(A,B) := \min_{a\in A,\ b \in B} \dist(a,b) \text{ and } \diam_{\square}(A) := \max_{i=1, \dots, d} |A_i|,
    \]
    where $\dist$ denotes the graph distance on $\torus$ and $A_i$ denotes the $i$th marginal of $A$. We define $\diam_{\square}(\emptyset) = 0$. Define the closure of $A$ as 
    \[
        \textrm{cl}(A) = A_1 \times \cdots \times A_d.
    \]
    When we apply these parameters to induced subgraphs of $\torus$ it should be read as applying it to their vertex sets.
    
    Let $\ell_1$ denote the length of the shortest side of $\torus$. If $\diam_{\square}(A) < \ell_1$ it is not hard to see that $\torus \setminus \textrm{cl}(A)$ is contained in a unique connected component of $\torus \setminus A$, which we will denote by $\ext(A)$. We let $\inte(A) = \mathcal{T} \setminus (A \cup \ext(A))$.
\end{definition}

The following lemma will be used implicitly several times.
\begin{lemma}
    Let $A_1,A_2$ be induced subgraphs of a torus $\torus$ with shortest side $\ell_1$ such that $\diam_{\square}(A_i) < \ell_1$, $\dist(A_1,A_2) \geq 2$ and both $A_1 \subseteq \ext(A_2)$ and $A_2 \subseteq \ext(A_1)$. Then $\inte(A_1) \cap \inte(A_2)=\emptyset$.
\end{lemma}

\begin{proof}
    Suppose for the sake of contradiction that $A_1,A_2$ is a counterexample for which $|\inte(A_1)|$ is minimized. Let $v$ be a vertex of $\inte(A_1)$ that is connected to a vertex in $A_1$, say $u$. Because $\dist(A_1,A_2) \geq 2$ it follows that $u \not \in A_2$. Therefore $u$ and $v$ lie in the same connected component of $\torus\setminus A_2$. Because $u$ lies in $\ext(A_2)$ it follows that $v \in \ext(A_2)$. Note also that $v$ is not connected to an element of $A_2$ because $A_2 \subseteq \ext(A_1)$ and $v \in \inte(A_1)$. Because $\diam_{\square}(A_1 \cup \inte(A_1)) = \diam_{\square}(A_1)$ it follows that
    \begin{itemize}
        \item $\diam_{\square}(A_1 \cup \{v\}) < \ell$ and $\diam_{\square}(A_2) < \ell_1$;
        \item $\dist(A_1 \cup \{v\}, A_2) \geq 2$;
        \item both $A_1 \cup \{v\} \subseteq \ext(A_2)$ and $A_2 \subseteq \ext(A_1 \cup \{v\} )$;
        \item $\inte(A_1 \cup \{v\}) \cap \inte(A_2) = \inte(A_1) \cap \inte(A_2)$, which is non-empty by assumption.
    \end{itemize}
    This is a contradiction because $\inte(A_1 \cup \{v\})$ is strictly smaller than $\inte(A_1)$.
\end{proof}


Let $\Lambda \subseteq \torus$ be an induced subgraph. We denote the boundary of $\Lambda$ by $\partial \Lambda \subseteq \Lambda$, i.e. the subgraph of $\Lambda$ induced by the vertices of $\Lambda$ with at least one neighbor in $\torus \setminus \Lambda$. We define $\partial^c \Lambda := \partial (\torus\setminus \Lambda)$. Denote by $\Lambda^{\circ}=\Lambda \setminus \partial \Lambda$ the interior of $\Lambda$. 
We write $|\Lambda|$ instead of $|V(\Lambda)|$ and we write $v \in \Lambda$ instead of $v \in V(\Lambda)$. 


\begin{remark}\label{rmk:infty connected boundaries of small subgraphs}
     Let $\torus$ be a $d$-dimensional even torus with minimal side length $\ell_1$. For any induced subgraph $\Lambda$ in $\torus$ with $\diam_\square(\Lambda) < \ell_1$, the induced subgraph $\ext(\Lambda) \cap \partial^c \Lambda$ is $\infty$-connected by Proposition B.82 in \cite{friedli_velenik_2017}.
\end{remark}

\begin{lemma}
    \label{lem: inf_con}
 Let $\torus$ be a $d$-dimensional even torus with minimal side length $\ell_1$. Let $\Lambda_1, \ldots, \Lambda_n$ and $A$ be induced subgraphs of $\torus$ satisfying
 \begin{enumerate}
     \item for each $i$ we have $\diam_\square(\Lambda_i) < \ell_1$ and $\Lambda_i^\circ \subseteq A$;
     \item for $i \neq j$ we have $\dist(\Lambda_i, \Lambda_j) \geq 2$;
     \item $A$ is $\infty$-connected,
 \end{enumerate}
then $\cap_{i=1}^ n \ext(\Lambda_i) \cap A$ is $\infty$-connected.
\end{lemma}
\begin{proof}
Take $u, v \in \cap_{i=1}^ n \ext(\Lambda_i) \cap A$. Because $A$ is $\infty$-connected, there is an $\infty$-path from $u$ to $v$ through $A$.
Denote this path by $(a_0, \ldots, a_k)$ for some $k \geq 0$, where $a_0 = u$ and $a_k = v$.
If the path has empty intersection with the sets $\Lambda_i$, we are done. 
Let $l$ denote the minimal index such that $a_l \in \Lambda_i$ for some $i$.
As $u, v \in \ext(\Lambda_i)$, there is a minimal index $m$ with $l<m<k$ such that $a_m \not \in \Lambda_i$.
As $\Lambda_i^\circ \subseteq A$, we see $a_{l-1},a_m \in \ext(\Lambda_i) \cap \partial^c \Lambda_i$.
We now claim there is a $\infty$-path from $a_{l-1}$ to $a_m$ which does not intersect $\Lambda_i$.

To prove this claim, first note for each $i$ we have $\diam_\square(\Lambda_i) < \ell_1$ and thus the induced subgraph $\ext(\Lambda_i) \cap \partial^c \Lambda_i$ is $\infty$-connected by Remark~\ref{rmk:infty connected boundaries of small subgraphs}.
Since $\Lambda_i^\circ \subseteq A$, we see $\ext(\Lambda_i) \cap \partial^c \Lambda_i \subseteq A$.
As for any $j \neq i$ we have $\dist(\Lambda_i, \Lambda_j) \geq 2$ we see 
$\ext(\Lambda_i) \cap \partial^c \Lambda_i$ does not intersect $\Lambda_j$.
Hence there is a $\infty$-path from $a_{l-1}$ to $a_m$ using only vertices from $\ext(\Lambda_i) \cap \partial^c \Lambda_i$, and none of the vertices of this path intersect $\Lambda_j$ for $j \neq i$.
This proves the claim from which the lemma follows. 
\end{proof}

\subsubsection{Contour representation of independent sets}
\begin{definition}
    Let $\Lambda \subseteq \torus$ be an induced subgraph. A map $\sigma:V(\Lambda) \to \{0,1\}$ is called a \emph{feasible configuration on $\Lambda$} if $I_\sigma = \{v \in V(\Lambda): \sigma(v) = 1\}$ is an independent set of $\Lambda$. 
\end{definition}

Given an independent set $I$ we denote the associated feasible configuration by $\sigma_I$.

   \begin{definition} We call a vertex of $\torus$ either \emph{even} or \emph{odd} if the sum of its coordinates is even or odd respectively. For an induced subgraph $\Lambda \subset \torus$ we denote by $\Lambda_\text{even}$ the subgraph induced by the even vertices of $\Lambda$ and $\Lambda_\text{odd}$ the subgraph induced by the odd vertices of $\Lambda$. The feasible configurations corresponding to the two maximal independent subsets of $\torus$, consisting of either all even or all odd vertices, are denoted by $\sigma_{\text{even}}$ and $\sigma_{\text{odd}}$. We refer to $\{\text{even}, \text{odd}\}$ as the set of \emph{ground states}.  Given a ground state $\varphi$, the complementary ground state will be denoted by $\overline{\varphi}$.
   \end{definition}

\begin{definition}\label{def:correct}
     Let $\Lambda \subseteq \torus$ be an induced subgraph. Given any feasible configuration $\sigma:V(\Lambda)\to \{0,1\}$ we say a vertex $v \in V(\Lambda)$ is \emph{correct} if there exists a ground state $\varphi \in \{\text{even}, \text{odd}\}$ such that for all $u \in N_\infty[v]\cap \Lambda$ we have $\sigma(u) = \sigma_\varphi(u)$, otherwise $v$ is defined to be \emph{incorrect}. We write $\overline{\Gamma(\Lambda,\sigma)}$ for the subgraph of $\Lambda$ induced by the set of incorrect vertices in $\Lambda$ with respect to $\sigma$.
\end{definition}

\begin{figure}
  \includegraphics[scale=0.7]{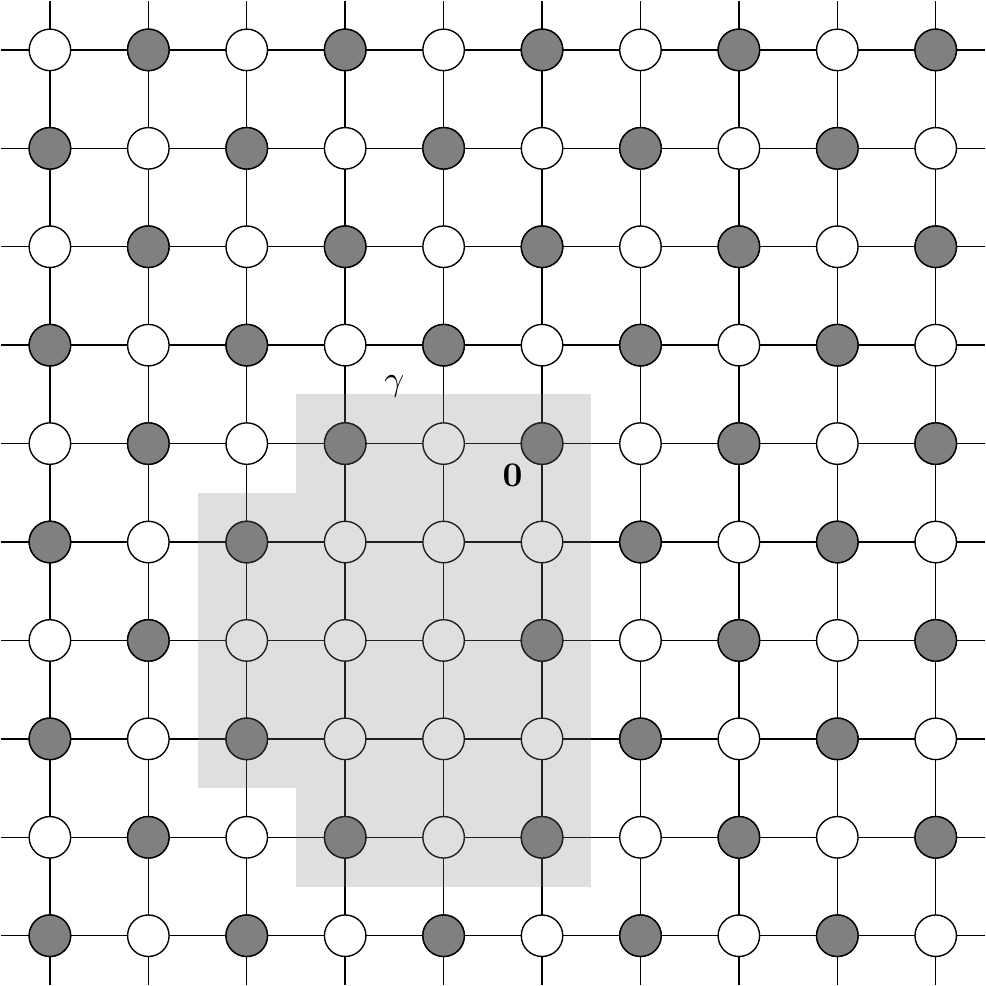}
\caption{A contour $\gamma$ in a $10$ by $10$ torus. Vertices $v$ such that $\sigma(v)=1$ are in dark gray and vertices $v$ such that $\sigma(v)=0$ are in white. The shaded gray region denotes the support of $\gamma$. The label of $\mathbb{Z}_{10}^2 \setminus \overline{\gamma}$ is even.}
    \label{fig:SingleContour}
\end{figure}

\begin{definition}\label{def:contour}
    Let $\gamma$ be a tuple $(\overline{\gamma},\sigma_\gamma)$ with \emph{support} $\overline{\gamma}$ a nonempty induced subgraph of $\torus$ and a feasible configuration $\sigma_\gamma: V(\overline{\gamma}) \to \{0,1\}$ for which there exists a labeling function $\lab_\gamma: \mathcal{C}(\torus\setminus \overline{\gamma}) \to \{ \text{even}, \text{odd}\}$ such 
    that the map $\hat{\sigma}_{\gamma}: V(\torus)\to \{0,1\}$ given by \[\hat{\sigma}_{\gamma}(v) = \begin{cases} 
    \sigma_\gamma(v)    &\text{ if }        v \in V(\overline{\gamma})\\
    \sigma_{\lab_\gamma(A)}(v) &\text{ if }     v \in V(A) \text{ with } A\in \mathcal{C}(\torus\setminus \overline{\gamma})
    \end{cases}
    \]
    is a feasible configuration on $\torus$ and $\overline{\gamma} = \overline{\Gamma(\torus,\hat{\sigma}_{\gamma})}$.
    Let $\ell_1$ denote the minimal side length of $\torus$. We say that $\gamma$ is a \emph{small contour} if $\overline{\gamma}$ is connected and satisfies $\diam_\square(\overline{\gamma})<\ell_1$.
    We say that $\gamma$ is a \emph{large contour} if each connected component of $\overline{\gamma}$ satisfies $\diam_\square(\overline{\gamma}) \geq \ell_1$. A \emph{contour} is either a small or a large contour. Two contours $\gamma_1, \gamma_2$ in $\torus$ have \emph{compatible support} if 
    \[
    \dist(\overline{\gamma_1},\overline{\gamma_2}) \geq 2.
    \] 
\end{definition}

See Figure \ref{fig:SingleContour} for an illustration of a contour $\gamma$ in the torus $\mathbb{Z}_{10}^2$.

\begin{remark}
    A contour $\gamma=(\overline{\gamma},\sigma_\gamma)$ uniquely determines the labeling function, $\lab_\gamma$, and the associated feasible configuration, $\hat{\sigma}_\gamma$.
\end{remark}

\begin{definition}
    We denote the exterior of a small contour $\gamma$ by $\ext(\gamma)$ instead of $\ext(\overline{\gamma})$. The label of $\ext(\gamma)$ is called the \emph{type} of $\gamma$. For a set $\Gamma$ of small contours we define the exterior $\ext(\Gamma) = \cap_{\gamma \in \Gamma} \ext(\gamma)$, with the convention that $\ext(\emptyset) = \torus$. For a large contour we do not define the exterior, but we artificially define the type of a large contour to be even.
    
    For any contour $\gamma$ and any ground state $\xi \in \{\text{even}, \text{odd}\}$ we define the \emph{$\xi$-interior of $\gamma$} as the union over all non-exterior connected components of $\torus \setminus \overline{\gamma}$ with label $\xi$, we denote this induced subgraph of $\torus $ by $\inte_\xi(\gamma)$. Denote the \emph{interior} of a contour $\gamma$ by  $\inte(\gamma) = \inte_{\text{even}}(\gamma)\cup \inte_{\text{odd}}(\gamma)$. 
\end{definition}

We note that the interior of any small contour $\gamma$ cannot contain a connected component of a large contour because its box-diameter is strictly less than $\ell_1$ (where $\ell_1$ denotes the minimum side length of the underlying torus).

\begin{definition}\label{def: new matching contours def}
    Let $\Gamma$ be a set of contours with pairwise compatible supports containing at most one large contour. We say $\Gamma$ is a \emph{matching set of contours} if there is a labeling function
\[
    \text{lab}_\Gamma: \mathcal{C}\big(\mathcal{T} \setminus \bigcup_{\gamma \in \Gamma} \overline{\gamma}\big) \to \{\text{even}, \text{odd}\}
\]
such that for each $A \in \mathcal{C}\big(\mathcal{T} \setminus \bigcup_{\gamma \in \Gamma} \overline{\gamma}\big)$ and $\gamma \in \Gamma$ with $\dist(A,\overline{\gamma})=1$ we have that $\hat{\sigma}_\gamma$ is equal to $\sigma_{\text{lab}(A)}$ when restricted to $A$.
\end{definition}

For any contour $\gamma$ the set $\Gamma = \{\gamma\}$ is a matching set of contours. 

\begin{definition}\label{def: omegamatch and the two empty clusters}
     For non-empty $\Gamma$ the labeling function $\lab_\Gamma$ is unique. If $\Gamma$ is empty there are two possible labeling functions, namely the one that assigns either even or odd to $\torus$. We view these as distinct matching sets of contours and denote them by $\emptyset_{\text{even}}$ and $\emptyset_{\text{odd}}$. Formally we thus define 
\[
    \Omega_{\text{match}}(\torus) = \{\Gamma: \text{$\Gamma$ a non-empty matching set of contours}\} \cup \{\emptyset_{\text{even}}\} \cup \{\emptyset_{\text{odd}}\}
\]
as the set of all matching sets of contours.
\end{definition}

\begin{figure}
  \includegraphics[scale=0.7]{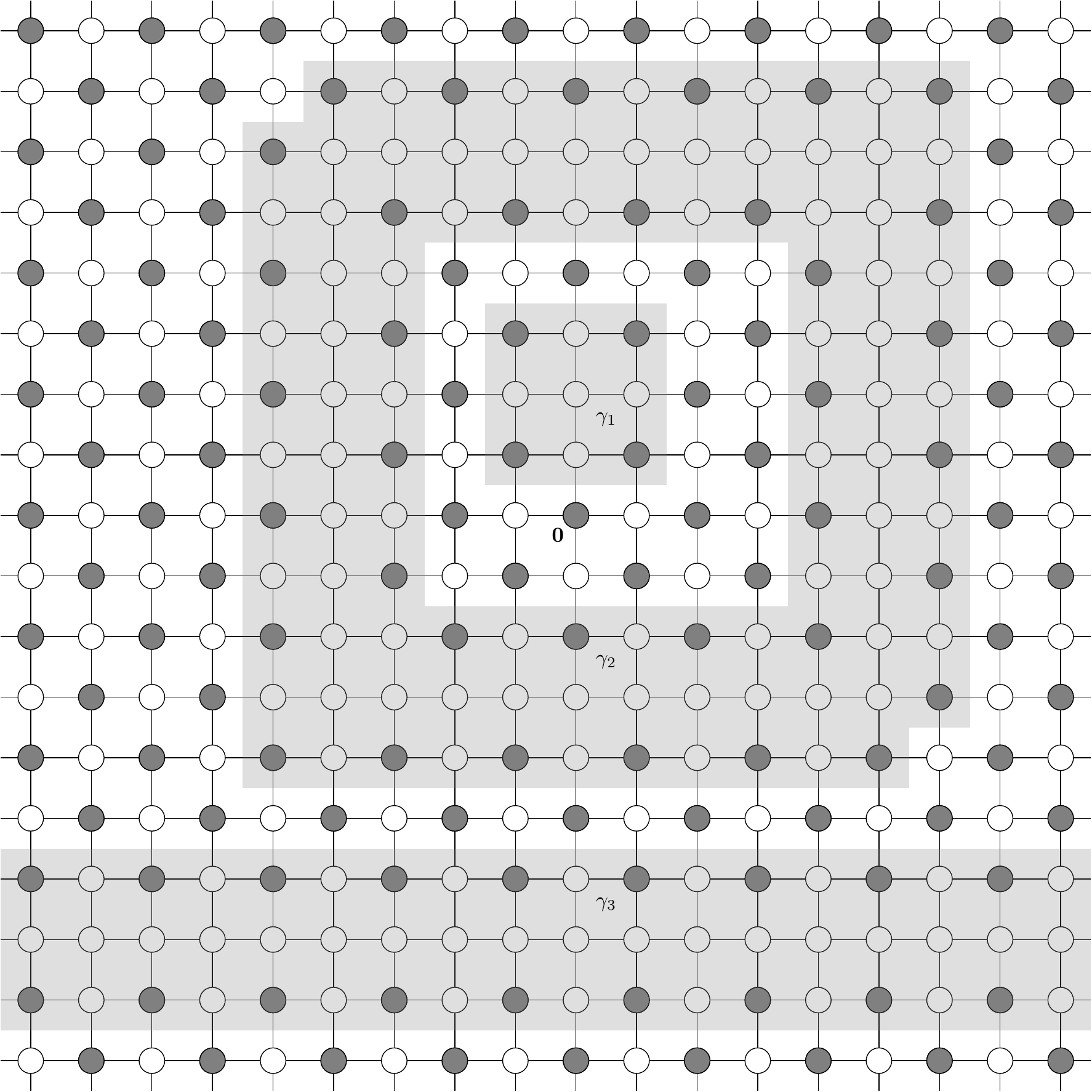}
\caption{A matching set of contours in an $18$ by $18$ torus. The contour $\gamma_1$ is small of type even, $\gamma_2$ is small of type odd and $\gamma_3$ is large. The contours $\gamma_2$ and $\gamma_1$ lie in the odd-interior of $\gamma_3$, the contour $\gamma_1$ lies in the even-interior of $\gamma_2$.}
    \label{fig:ConfigurationContours}
\end{figure}

See Figure \ref{fig:ConfigurationContours} for an illustration of a matching set of contours in an $18$ by $18$  torus.

\begin{definition}
    \label{def: surface energy contour}
     For a contour $\gamma = (\overline{\gamma},\sigma_\gamma)$ in $\torus$ we define the \emph{surface energy} as
    \[
    ||\gamma|| := \frac{1}{4d} \sum_{\substack{v \in V(\overline{\gamma})\\ \sigma_\gamma(v) = 0}} \Big(2d - \sum_{u \in N(v)} \hat{\sigma}_\gamma(u)\Big).
    \]
    For a matching set of contours $\Gamma$ we define $||\Gamma|| = \sum_{\gamma \in \Gamma} ||\gamma||$.
\end{definition}

In Theorem \ref{thm: bijection matching contours and spin configs and formula for surface energy part 2} we show the surface energy is always integer.


\begin{theorem}\label{thm: bijection matching contours and spin configs and formula for surface energy part 2}
      There is a bijection between the set of all sets of matching contours $ \SetOfContours_{\text{match}}(\torus) $ and the set of feasible configurations on an even torus $\torus$ such that for any $\Gamma \in \SetOfContours_{\text{match}}(\torus)$ and its corresponding feasible configuration $\tau:V(\torus) \to \{0,1\}$ we have
    \begin{equation}
    \label{eq: surface_energy_correspondence}
 ||\Gamma|| = \frac{|\torus|}{2} - |I_{\tau}|.
    \end{equation}
\end{theorem}
\begin{proof}


For $\Gamma \in \SetOfContours_{\text{match}}(\torus)$ we define the feasible configuration $\tau_\Gamma$ as 
\[
    \tau_\Gamma(v) = 
\begin{cases}
    \sigma_\gamma(v)  &\text{if $v \in \overline{\gamma}$ for some $\gamma \in \Gamma$} \\
    \sigma_{\lab_\Gamma(A)}(v)  &\text{if $v \in A$ for some $A \in \mathcal{C}\big(\mathcal{T} \setminus \bigcup_{\gamma \in \Gamma} \overline{\gamma}\big)$}. \\
\end{cases}
\]
We recall here that $\Gamma \in \SetOfContours_{\text{match}}(\torus)$ contains two copies of the empty set with either an even or an odd label. These correspond to $\sigma_{\text{even}}$ and $\sigma_{\text{odd}}$ respectively. It follows directly from the definition of $\SetOfContours_{\text{match}}(\torus)$ that $\tau_\Gamma$ is indeed feasible. 


We now define a map from the set of feasible configurations to $\SetOfContours_{\text{match}}(\torus)$. Let $\tau:V(\torus) \to \{0,1\} $ be a feasible configuration. The induced subgraph  $\overline{\Gamma(\torus ,\tau)}$ consists of a union of say $s \geq 0$ connected components with box-diameter strictly less than $\ell_1$ and $m \geq 0$ connected components with box-diameter $\geq \ell_1$. If $s=m=0$ then $\tau$ is equal to either $\sigma_{\text{even}}$ or $\sigma_{\text{odd}}$ which we map to $\emptyset_{\text{even}}$ or $\emptyset_\text{odd}$ respectively.
If $m \geq 1$ then we denote the union of all connected components with box-diameter $\geq \ell_1$ by $\overline{\gamma_{\text{large}}}$. 
Denote the remaining connected components of $\overline{\Gamma(\torus ,\tau)}$ by $\overline{\gamma_i}$ for $i \in \{1, \ldots, s\}$.
By restricting $\tau$, we define a feasible configuration $\sigma_\gamma$ on each support $\overline{\gamma} \in \{\overline{\gamma_{\text{large}}}, \overline{\gamma_1}, \ldots, \overline{\gamma_s} \}$. We have to show that for each such $\overline{\gamma}$ we can define a labeling function $\lab_{\gamma}$ on the connected components of $\torus \setminus \overline{\gamma}$ that makes $\gamma$ into a well-defined contour.

To do this it is sufficient to show that, given $A \in \mathcal{C}(\torus \setminus \overline{\gamma})$, there exists $\varphi\in \{\text{even},\text{odd}\}$ such that $\tau$ restricted to the vertices of $A \cap \partial^c(\overline{\gamma})$ is equal to $\sigma_\varphi$. Indeed, if this is the case then we can define $\lab_{\gamma}(A) = \varphi$. It is not hard to see that the corresponding configuration $\hat{\sigma}_\gamma$ as defined in Definition~\ref{def:contour} is then feasible and satisfies $\overline{\gamma} = \overline{\Gamma(\torus,\hat{\sigma}_\gamma)}$. We distinguish between two cases. 

In the first case $\overline{\gamma}$ is not $\overline{\gamma_{\text{large}}}$ and $A = \ext(\overline{\gamma})$. It then follows from Remark~\ref{rmk:infty connected boundaries of small subgraphs} that $A \cap \partial^c(\overline{\gamma})$ is a $\infty$-connected set of correct vertices with respect to $\tau$. It follows that there is a unique $\varphi$ such that $\tau = \sigma_\phi$ when restricted to $A \cap \partial^c(\overline{\gamma})$. 

In the second case $A$ has empty intersection with $\overline{\gamma_{\text{large}}}$ and thus any $\overline{\gamma}'$ contained in $A$ has box-diameter strictly less than $\ell_1$. Let $\Gamma'$ be the collection of these $\overline{\gamma}'$. Any $\overline{\gamma}' \in \Gamma'$ must be contained in $A^\circ$ because otherwise $\overline{\gamma} \cup \overline{\gamma}'$ would be a single component of $\overline{\Gamma(\torus ,\tau)}$. It now follows from Lemma~\ref{lem: inf_con} that $A' := \cap_{\gamma' \in \Gamma'} \ext(\gamma') \cap A$ is $\infty$-connected. Because $A'$ consists of correct vertices with respect to $\tau$ and $\partial^c(\overline{\gamma}) \cap A \subseteq A'$ it follows that there is a $\varphi$ such that $\tau = \sigma_\varphi$ when restricted to $\partial^c(\overline{\gamma}) \cap A$.

We have shown that $\Gamma_\tau := \{\gamma_{\text{large}}, \gamma_1, \dots, \gamma_s\}$ is a set of contours with pairwise compatible supports. The labeling function $\lab_\Gamma$ that assigns to any component with vertex $v$ the label inherited from $\tau$ shows that indeed $\Gamma_\tau \in \SetOfContours_{\text{match}}(\torus)$. By definition the maps $\tau \mapsto \Gamma_\tau$ and $\Gamma \mapsto \tau_\Gamma$ are each others inverse.


We now prove the equality in equation~{(\ref{eq: surface_energy_correspondence})}. Let $\Gamma$ be a set of matching contours and $\tau$ its corresponding feasible configuration.
We count the number of edges in $\torus$ in two ways. The total number of edges in $\torus$ is $2d \cdot \frac{|\torus|}{2}$, as there are $\frac{|\torus|}{2}$ even vertices in $\torus$ and each even vertex is incident to $2d$ distinct edges. The number of edges in $\torus$ also equals the number of edges between $I_{\tau}$ and $\torus \setminus I_\tau$ plus the number of edges within $\torus \setminus I_\tau$.
The number of edges  between $I_\tau$ and $\torus \setminus I_\tau$ is equal to $2d \cdot |I_\tau|$, as each vertex of $I_\tau$ has degree $2d$ and $I_\tau$ is independent.
For a vertex $v\in \torus \setminus I_\tau$ the number of neighbors of $v$ in $\torus \setminus I_\tau$ is $2d - \sum_{u \in N(v)} \tau(u)$, the degree of $v$ minus the number of neighbors of $v$ in $I_\tau$. 
The number of edges within $\torus \setminus I_\tau$ is the sum of the number of neighbors of $v$ in $\torus \setminus I_\tau$ over all $v\in \torus \setminus I_\tau$ divided by $2$, hence
\[
|E(\torus \setminus I_\tau)| = \frac{1}{2} \sum_{\substack{v \in V(\torus)\\ \tau(v) = 0}} \Big(2d - \sum_{u \in N(v)} \tau(u)\Big) = \frac{1}{2} \sum_{\gamma \in \Gamma} \sum_{\substack{v \in V(\overline{\gamma})\\ \tau(v) = 0}} \Big(2d - \sum_{u \in N(v)} \tau(u)\Big)=2d \cdot ||\Gamma||,
\]
where the second equality follows as each vertex outside of $\cup_{\gamma\in \Gamma} V(\overline{\gamma})$ is correct with respect to $\tau$.
From this we see
\[
2d \cdot ||\Gamma||+ 2d \cdot |I_\tau| = 2d \cdot \frac{|\torus|}{2},
\]
hence $||\Gamma|| = \frac{|\torus|}{2} - |I_\tau|$.
\end{proof}


\begin{definition}
     We define the \emph{matching contour partition function on $\torus$} as
\[
Z_{\text{match}}(\torus;z) := \sum_{\Gamma \in \SetOfContours_{\text{match}}(\torus)} \prod_{\gamma \in \Gamma} z^{||\gamma||}.
\]
\end{definition}

Let $Z_\text{ind}(G;\lambda)$ deote the independence polynomial of a graph $G$ evaluated at $\lambda$.

\begin{corollary}\label{cor:Zmatch = Zind}
    We have
\[
Z_{\text{ind}}( \torus; \lambda) := \lambda^{\frac{|\torus|}{2}} \cdot Z_{\text{match}}(\torus; \frac{1}{\lambda}).
\]
\end{corollary}
\begin{proof}
    This follows from Theorem \ref{thm: bijection matching contours and spin configs and formula for surface energy part 2} and the definition of the matching contour partition function.
\end{proof}

As $Z_{\text{match}}(\torus; 0) = 2 \neq 0$, the first part of the main theorem is equivalent to finding a zero free region around $z=0$ of $Z_{\text{match}}(\torus;z)$ for all $C$-balanced tori.

\subsubsection{Contours as polymers}
We next collect some definitions allowing us to split up $Z_{\text{match}}$ up into two parts which we can then interpret as polymer partition functions.

\begin{definition}
    We partition $\SetOfContours_{\text{match}}(\torus)$ into three subsets. We let $\SetOfContours_{\text{match}}^{\text{large}}(\torus)$ consists of those $\Gamma \in \SetOfContours_{\text{match}}(\torus)$ that contain a large contour. If $\Gamma$ consists of small contours we define the \emph{type} of $\Gamma$ as the label assigned to $\ext(\Gamma)$ by $\lab_{\Gamma}$. We denote those $\Gamma$ with type $\varphi \in \{\text{even}, \text{odd}\}$ by $\SetOfContours_{\text{match}}^\varphi(\torus)$. Note that $\emptyset_{\text{even}} \in \SetOfContours_{\text{match}}^{\text{even}}(\torus)$, $\emptyset_{\text{odd}} \in \SetOfContours_{\text{match}}^{\text{odd}}(\torus)$ and
    \[
        \SetOfContours_{\text{match}}(\torus)=\SetOfContours_{\text{match}}^{\text{even}}(\torus)\cup \SetOfContours_{\text{match}}^{\text{odd}}(\torus) \cup
        \SetOfContours_{\text{match}}^{\text{large}}(\torus).
    \]
    For $\varphi \in \{\text{even}, \text{odd}, \text{large}\}$ we define 
    \[
        Z_{\text{match}}^{\varphi}(\torus;z) = \sum_{\Gamma \in \SetOfContours_{\text{match}}^\varphi(\torus)} \prod_{\gamma \in \Gamma} z^{||\gamma||}.
    \]
\end{definition}

\begin{definition}\label{def:externals}
For any $\varphi \in \{\text{even}, \text{odd}\}$ and $\Gamma \in \SetOfContours_{\text{match}}^\varphi(\torus)$ we define 
\begin{equation}
    \label{eq: externals}
    \Gamma_{\ext} = \{\gamma \in \Gamma: \overline{\gamma} \subseteq \ext(\gamma') \text{ for all }\gamma' \in \Gamma \text{ not equal to $\gamma$}\}.
\end{equation}
We furthermore define 
\[
\SetOfContours_{\text{ext}}^\varphi(\torus) = \{\Gamma \in \SetOfContours_{\text{match}}^\varphi(\torus): \Gamma = \Gamma_\text{ext}\}.
\]
    
\end{definition}

We will further rewrite $Z_{\text{match}}(\torus;z)$ in order to apply the framework outlined in Section \ref{sec: cluster expansion}.
The first step in rewriting is a standard technique from Pirogov-Sinai theory, analogous to what was done for finite induced subgraphs $\Lambda \subseteq \mathbb{Z}^d$ with padded boundary conditions in \cite{PirogovSinaiWillGuusTyler}. 

We define a class of well-behaved induced subgraphs of tori.

\begin{definition}
    Let $\Lambda \subseteq \torus$ be an induced subgraph. If for any small contour $\gamma$ with $\overline{\gamma} \subseteq \Lambda^\circ$ we have $\inte(\gamma) \subseteq \Lambda^\circ$ we say $\Lambda$ is \emph{closed under taking interiors of small contours}, or more succinctly \emph{closed}. 
\end{definition}

Note $\torus$ is closed, and for any contour $\gamma$ the induced subgraphs $\inte_{\text{odd}}(\gamma)$, $\inte_{\text{even}}(\gamma)$ and $\inte(\gamma)$ are also closed. 

\begin{definition}
    Let $\Lambda \subseteq \torus$ be an induced closed subgraph and $\varphi \in \{\text{even,odd}\}$ a ground state. We define 
\[
\SetOfContours^\varphi_{\text{match}}(\Lambda) := \{ \Gamma \in \SetOfContours^\varphi_{\text{match}}(\torus): \text{ for all } \gamma \in \Gamma \text{ we have } \overline{\gamma} \subseteq \Lambda^\circ \}.
\]
and
\[
\SetOfContours^\varphi_{\text{ext}}(\Lambda) := \{ \Gamma \in \SetOfContours^\varphi_{\text{ext}}(\torus): \text{ for all } \gamma \in \Gamma \text{ we have } \overline{\gamma} \subseteq \Lambda^\circ \}
\]
We also define the \emph{matching contour partition function} as
\[
Z_{\text{match}}^\varphi(\Lambda;z) = \sum_{\Gamma \in \SetOfContours^{\varphi}_{\text{match}}(\Lambda)} \prod_{\gamma \in \Gamma} z^{||\gamma||}.
\]
\end{definition}

Note if $\Lambda^\circ = \emptyset$ then $Z_{\text{match}}^\varphi(\Lambda;z) = 1$ as in that case $\SetOfContours^{\varphi}_{\text{match}}(\Lambda)= \{ \emptyset_\varphi \}$. 

\begin{lemma}\label{lem: factorising property Zmatch}
   For induced closed subgraphs  $\Lambda_1, \Lambda_2\subset \torus$ with $\dist(\Lambda_1,\Lambda_2) \geq 2$ and any $\varphi \in \{\text{even, odd}\}$ we have $Z^\varphi_{\text{match}}(\Lambda_1 \cup \Lambda_2;z)= Z^\varphi_{\text{match}}(\Lambda_1;z) \cdot Z^\varphi_{\text{match}}(\Lambda_1;z)$. 
\end{lemma}
\begin{proof}
Note that the induced subgraph $\Lambda_1 \cup \Lambda_2$ is closed.
 The equality follows from the bijection between $\SetOfContours^{\varphi}_{\text{match}}(\Lambda_1) \times \SetOfContours^{\varphi}_{\text{match}}(\Lambda_2)$ and $\SetOfContours^{\varphi}_{\text{match}}(\Lambda_1 \cup \Lambda_2)$ given by $(\Gamma_1, \Gamma_2) \mapsto \Gamma_1 \cup \Gamma_2$.
\end{proof}

\begin{lemma}
    \label{lem: match_to_ext}
    For any induced closed subgraph $\Lambda \subseteq \torus$, any ground state $\varphi \in \{\text{even}, \text{odd}\}$ and any $z \in \mathbb{C}$ we have
    \[
Z_{\text{match}}^{\varphi}(\Lambda;z) = \sum_{\Gamma \in \SetOfContours^{\varphi}_{\text{ext}}(\Lambda)} \prod_{\gamma \in \Gamma} z^{||\gamma||} Z_{\text{match}}^{\text{even}}(\inte_{\text{even}}(\gamma);z) Z_{\text{match}}^{\text{odd}}(\inte_{\text{odd}}(\gamma);z).
    \]
\end{lemma}
\begin{proof}
Given an induced closed subgraph $\Lambda \subseteq \torus$ and a set $\Gamma \in \SetOfContours^{\varphi}_{\text{ext}}(\Lambda)$
we have 
\[
\sum_{\substack{\Gamma' \in \SetOfContours^{\varphi}_{\text{match}}(\Lambda)\\ \Gamma'_{\text{ext}} = \Gamma}} \prod_{\gamma \in \Gamma'\setminus\Gamma}  z^{||\gamma||} =  Z_{\text{match}}^{\text{even}}(\cup_{\gamma \in \Gamma} \inte_{\text{even}}(\gamma);z) Z_{\text{match}}^{\text{odd}}(\cup_{\gamma \in \Gamma}\inte_{\text{odd}}(\gamma);z),
\]
because any $\Gamma' \in \SetOfContours^{\varphi}_{\text{match}}(\Lambda)$ with $\Gamma'_{\text{ext}} = \Gamma$ gives an associated set of matching contours in  
\[
\SetOfContours^{\text{even}}_{\text{match}}(\cup_{\gamma \in \Gamma} \inte_{\text{even}}(\gamma)) \times \SetOfContours^{\text{odd}}_{\text{match}}(\cup_{\gamma \in \Gamma} \inte_{\text{odd}}(\gamma)),
\]
as any non external contour $\gamma'\in \Gamma'$ lies in the interior of a unique contour $\gamma \in \Gamma$.
By Lemma \ref{lem: factorising property Zmatch} we see for any induced closed subgraph $\Lambda \subseteq \torus$ and any $\Gamma \in \SetOfContours^{\varphi}_{\text{ext}}(\Lambda)$ that
\[
\prod_{\gamma \in \Gamma} Z_{\text{match}}^{\text{even}}(\inte_{\text{even}}(\gamma);z) Z_{\text{match}}^{\text{odd}}(\inte_{\text{odd}}(\gamma);z) = Z_{\text{match}}^{\text{even}}(\cup_{\gamma \in \Gamma} \inte_{\text{even}}(\gamma);z) Z_{\text{match}}^{\text{odd}}(\cup_{\gamma \in \Gamma}\inte_{\text{odd}}(\gamma);z).
\]
Combined, these two facts yield
 \begin{align*}
      &Z_{\text{match}}^{\varphi}(\Lambda;z) = \sum_{\Gamma \in \SetOfContours^{\varphi}_{\text{ext}}(\Lambda)} \sum_{\substack{\Gamma' \in \SetOfContours^{\varphi}_{\text{match}}(\Lambda)\\ \Gamma'_{\text{ext}} = \Gamma}} \prod_{\gamma \in \Gamma'}  z^{||\gamma||} =   \\
      &\sum_{\Gamma \in \SetOfContours^{\varphi}_{\text{ext}}} Z_{\text{match}}^{\text{even}}(\cup_{\gamma \in \Gamma} \inte_{\text{even}}(\gamma);z) Z_{\text{match}}^{\text{odd}}(\cup_{\gamma \in \Gamma}\inte_{\text{odd}}(\gamma);z) \prod_{\gamma \in \Gamma} z^{||\gamma||} = \\
      &\sum_{\Gamma \in \SetOfContours^{\varphi}_{\text{ext}}(\Lambda)} \prod_{\gamma \in \Gamma} z^{||\gamma||} Z_{\text{match}}^{\text{even}}(\inte_{\text{even}}(\gamma);z) Z_{\text{match}}^{\text{odd}}(\inte_{\text{odd}}(\gamma);z).
 \end{align*}
\end{proof}

\begin{definition}
\label{def: weights}
We define for a contour $\gamma$ in $\torus$ of type $\varphi$ the weight to be the following rational function in $z$
\[
w(\gamma;z) := z^{||\gamma||} \frac{Z_{\text{match}}^{\overline{\varphi}} (\inte_{\overline{\varphi}}(\gamma);z)}{Z_{\text{match}}^{\varphi} (\inte_{\overline{\varphi}}(\gamma);z)}.
\]
\end{definition}

Recall that the type of a large contour is defined to be even. Note that for any induced closed subgraph $\Lambda \subset \torus$ and any contour $\gamma$ in $\Lambda$ the contour $\gamma$ is also a contour in $\torus$ and hence $w(\gamma;z)$ is defined. We also note the denominator of $w(\gamma;z)$ has constant term $1$ for any contour $\gamma$.

The definition of these weights is a standard trick in Pirogov-Sinai theory used to rewrite the independence polynomial as a polymer partition function; see for example \cite{PirogovSinaiWillGuusTyler}. 
To also do this for tori, we define a suitable compatibility relation, which is a modification of the compatibility relation used in \cite{PirogovSinaiWillGuusTyler} to accommodate for the large contours. 

\begin{definition}
    We define two contours $\gamma_1,\gamma_2$ in $\torus$ to be \emph{torus-compatible} if they have compatible supports and if either $(1)$ $\gamma_1$ and $\gamma_2$ are both small and of the same type or if $(2)$ one contour is large and the other is small and of type even. 
 Denote by $\TorusCompatContours^{\varphi}_{\text{small}}(\torus)$ the collection of sets containing small pairwise torus-compatible contours in $\torus$ of type $\varphi$ and by $\TorusCompatContours^{\text{even}}(\torus)$ the collection of sets of torus-compatible contours in $\torus$ of type even in which we allow both large and small contours.
\end{definition}

Note that torus-compatibility is an anti-reflexive and symmetric relation on the set of contours. 

\begin{definition}
    Let $\Lambda \subseteq \torus$ be an induced closed subgraph and let $\varphi \in \{\text{even,odd}\}$ be a ground state. We define 
\[
\TorusCompatContours^\varphi_{\text{small}}(\Lambda) := \{ \Gamma \in \TorusCompatContours^\varphi_{\text{small}}(\torus): \text{ for all } \gamma \in \Gamma \text{ we have } \overline{\gamma} \subseteq\Lambda^\circ \}.
\]
\end{definition}

For any $\Gamma \in \TorusCompatContours^\varphi_{\text{small}}(\Lambda)$ we can define $\Gamma_{\ext}$ exactly how is done in \eqref{eq: externals} in Definition~\ref{def:externals}. 
It is not difficult to see that then $\Gamma_{\ext} \in \SetOfContours^{\varphi}_{\text{ext}}(\Lambda)$. This observation, together with Lemma~\ref{lem: match_to_ext} and the choice of weights in Definition~\ref{def: weights}, allows us to rewrite the matching contour partition function, which is a sum over matching sets of contours, as a sum over sets that only require pairwise compatibility. 

\begin{lemma}\label{lem:contourpartitionfunction_usingtoruscomptaibility}
    Let $\Lambda \subseteq \torus$ be an induced closed subgraph and let $\varphi \in \{\text{even,odd}\}$ be a ground state. We have for any $z \in \mathbb{C}$
    \[
Z_{\text{match}}^{\varphi}(\Lambda;z) = \sum_{\Gamma\in \TorusCompatContours^\varphi_{\text{small}}(\Lambda)} \prod_{\gamma \in \Gamma} w(\gamma;z).
\]
Furthermore, we have
\[
Z_{\text{match}}^{\text{even}}(\torus;z) + Z_{\text{match}}^{\text{large}}(\torus;z) = \sum_{\Gamma\in \TorusCompatContours^{\text{even}}(\torus)} \prod_{\gamma \in \Gamma} w(\gamma;z).
\]
\end{lemma}
\begin{proof}
We prove the first claim by induction on $|\Lambda|$. The base case is trivial. Suppose the claim holds for $|\Lambda'| \leq k$. Next suppose that $|\Lambda| = k+1$. 
By Lemma~\ref{lem: match_to_ext} we have
\begin{align*}
Z_{\text{match}}^{\varphi}(\Lambda;z) &= \sum_{\Gamma \in \SetOfContours^{\varphi}_{\text{ext}}(\Lambda)} \prod_{\gamma \in \Gamma} z^{||\gamma||} Z_{\text{match}}^{\varphi}(\inte_{\varphi}(\gamma);z) Z_{\text{match}}^{\overline{\varphi}}(\inte_{\overline{\varphi}}(\gamma);z),
\end{align*}
which by definition of the weights is equal to
\begin{align*}
&\sum_{\Gamma \in \SetOfContours^{\varphi}_{\text{ext}}(\Lambda)} \prod_{\gamma \in \Gamma} w(\gamma;z)  Z_{\text{match}}^{\varphi}(\inte(\gamma);z) = \sum_{\Gamma \in \SetOfContours^{\varphi}_{\text{ext}}(\Lambda)} \prod_{\gamma \in \Gamma} w(\gamma;z)  \sum_{\Gamma'\in \TorusCompatContours^\varphi_{\text{small}}(\inte(\gamma))} \prod_{\gamma' \in \Gamma'}w(\gamma';z)\\
=&  \sum_{\Gamma\in \TorusCompatContours^\varphi_{\text{small}}(\Lambda)} \prod_{\gamma \in \Gamma} w(\gamma;z),
\end{align*}
where the first equality uses the induction hypothesis on the induced closed subgraph $\inte(\gamma)$ and the last equality follows from the definition of torus-compatibility.
This proves the first part.

Note that $Z_{\text{match}}^{\text{large}}$ is a sum over matching set of contours that contain a large contour. Therefore we can instead write $Z_{\text{match}}^{\text{large}}$ as a sum over all large contours. 
Reasoning as above we obtain
\begin{align*}
    Z_{\text{match}}^{\text{large}}(\torus;z) &=  \sum_{\substack{\gamma} }
    z^{\|\gamma\|}\cdot Z_{\text{match}}^{\text{even}}(\inte_{\text{even}}(\gamma);z) \cdot Z_{\text{match}}^{\text{odd}}(\inte_{\text{odd}}(\gamma_{\text{large}});z) 
    \\
     &= \sum_{\substack{\gamma}} w(\gamma;z)\cdot Z_{\text{match}}^{\text{even}}(\inte(\gamma);z) 
    = \sum_{\substack{\gamma}} w(\gamma;z) \sum_{\Gamma\in \mathcal{T}_{\text{small}}^{\text{even}}(\inte(\gamma_{\text{large}}))} \prod_{\tau \in \Gamma} w(\tau;z),
\end{align*}
where each sum is over large contours $\gamma$ and where the last inequality follows as $\inte(\gamma)$ is an induced closed subgraph of $\torus$.
By the definition of torus-compatibility of contours we obtain
\[
Z_{\text{match}}^{\text{even}}(\torus;z) + Z_{\text{match}}^{\text{large}}(\torus;z) = \sum_{\Gamma\in\TorusCompatContours^{\text{even}}(\torus)} \prod_{\gamma \in \Gamma} w(\gamma;z),
\]
as desired.
\end{proof}

\begin{remark}\label{rmk: translating contours for tori}
    An automorphism $t:\torus\to \torus$ acts on contours by pushing forward their support and pulling back their configurations and associated labels and type. We note that labels are preserved when $t(\vec{0})$ is even, and switched when $t(\vec{0})$ is odd. The surface energy is always preserved.
\end{remark}

\begin{lemma}\label{lem: symmetry even odd tori}
    For all $z\in \mathbb{C}$ we have $Z^{\text{even}}_\text{match}(\torus;z) = Z^{\text{odd}}_\text{match}(\torus;z)$.
\end{lemma}
\begin{proof}
Let $t:\torus \to \torus$ be the translation by $(0,\ldots, 0, 1)$. By Remark \ref{rmk: translating contours for tori} any even contour $\gamma$ corresponds to an odd contour $t(\gamma)$ with the same weight.
\end{proof}


Denote the set of small contours of type $\varphi$ with support contained in an induced closed subgraph $\Lambda \subseteq\torus$ by $S^{\varphi}_{\text{small}}(\Lambda)$. Denote the set of all small and large contours of even type with support contained in $\torus$ by $S^{\text{even}}(\torus)$.
Using Lemma \ref{lem:contourpartitionfunction_usingtoruscomptaibility} and the definition of torus-compatibility it follows that for a type $\varphi \in \{ \text{even, odd}\}$ and any induced closed subgraph $\Lambda \subseteq\torus$ the function $Z_{\text{match}}^{\varphi}(\Lambda;z)$ equals a polymer partition function as defined in Section \ref{sec: cluster expansion} with set of polymers $S^{\varphi}_{\text{small}}(\Lambda)$ and torus-compatibility as compatibility relation on $S^{\varphi}_{\text{small}}(\Lambda)$.
Similarly, we see that $Z_{\text{match}}^{\text{even}}(\torus;z) + Z_{\text{match}}^{\text{large}}(\torus;z)$ equals a polymer partition function with set of polymers $S^{\text{even}}(\torus)$.
We observe that by Lemma~\ref{lem: symmetry even odd tori}
\begin{align*}
Z_{\text{match}}(\torus;z) &= Z_{\text{match}}^{\text{odd}}(\torus;z) + Z_{\text{match}}^{\text{even}}(\torus;z) + Z_{\text{match}}^{\text{large}}(\torus;z)
\\
&=2Z_{\text{match}}^{\text{even}}(\torus;z) + Z_{\text{match}}^{\text{large}}(\torus;z),
\end{align*}
giving us the promised way of writing $Z_{\text{match}}(\torus;z)$ as the sum of two polymer partition functions.
Note that we cannot view $Z_{\text{match}}(\torus;z)$ as a single polymer partition function since it contains the occurrence of two `distinct' empty sets of matching contours.


\subsection{Applying the Koteck\'y-Preiss theorem}
To apply the Koteck\'y-Preiss theorem to $Z_{\text{match}}^{\text{odd}}(\torus;z)=Z_{\text{match}}^{\text{even}}(\torus;z)$ and $Z_{\text{match}}^{\text{even}}(\torus;z) + Z_{\text{match}}^{\text{large}}(\torus;z)$ we apply Zahradn\'ik's truncated-based approach to Pirogov-Sinai theory \cite{Zahradnk1984AnAV}, which is also used in \cite{BorgsImbrie}.
The idea is to first restrict to contours for which the weights respect a proper bound which helps us to check the condition of the Koteck\'y-Preiss theorem. This process `truncates' the partition function. 
We then prove, using bounds we obtain from the Koteck\'y-Preiss theorem on the truncated partition function, that in fact all contours satisfy this bound. To define the bound om the weights of the contours, we need the following lemmas and definition.

\begin{lemma}\label{lem:contours bound connected}
 Let $S_m$ denote the set of small contours $\gamma$ in $\torus$ with support of size $m$ containing $\vec{0}$. Then there is a constant $C_d$ depending only on $d$ such that $|S_m| \leq C_d^m$.
\end{lemma}
\begin{proof}
The size of $S_m$ is bounded by the number of connected subsets in $\mathbb{Z}^d$ of size $m$ containing $\vec{0}$ times $2^m$, as a contour is uniquely determined by its support and its configuration. In  \cite{Barequet2009FormulaeAG} connected subsets of size $m$ containing $\vec{0}$ are called \emph{strongly-embedded lattice site animals} and in \cite{Madras1999APT} just \emph{site animals}. The number of strongly embedded lattice site animals of size $m$ grows as $\lambda_d^m$ for a constant $\lambda_d$ depending on $d$, see \cite{Madras1999APT} and \cite{Barequet2009FormulaeAG}, which implies that there exists a constant $C_d$ such that $|S_m| \leq C_d^m$.
\end{proof}

We also need a lower bound on the surface energy of contours in terms of the number of vertices in the support.

\begin{lemma}[Peierls condition]\label{lem:Peierls condition}
    Let $\gamma$ be a contour in $\torus$.
For $\rho=\rho(d) := \frac{1}{2d\cdot 3^d}$ the surface energy of a contour  satisfies \[
\rho |\overline{\gamma}| \leq ||\gamma|| \leq |\overline{\gamma}|.
\]
\end{lemma}
\begin{proof}
The inequality $||\gamma|| \leq |\overline{\gamma}|$ is trivial. For the other inequality, note for each incorrect vertex $v\in \overline{\gamma}$ there is at least one vertex $u\in N_\infty[v]$ and a neighbor $w\in N_\infty[v]$ of $u$ such that $\sigma_\gamma(u)= \sigma_\gamma(w) = 0$. Hence we see that at least two of the $3^d$ vertices in  $N_\infty[v]$ have a contribution of at least $\frac{1}{4d}$ each to $||\gamma||$. 
We double count this contribution at most $3^d$ times, as $|N_\infty[v]| = 3^d$.
This yields $||\gamma|| \geq \rho |\overline{\gamma}|$ for $\rho = \rho(d)=  2\cdot \frac{1}{4d} \cdot \frac{1}{3^d} = \frac{1}{2d\cdot 3^d}$.
\end{proof}

\begin{definition}\label{def:deltad}
Define for any $d \in \mathbb{Z}_{\geq 2}$ and any $x >0$ the real number 
\[
\delta_1(d,x):= e^{-(\log(2C_d)+4d+5\cdot e^{-x3^d}+x)/\rho(d)},
\]
where $C_d$ is the constant from Lemma \ref{lem:contours bound connected} and $\rho(d) = \frac{1}{2d\cdot 3^d}$ is the constant from lemma \ref{lem:Peierls condition}.
\end{definition}

We can now define stability of contours.

\begin{definition}
  Let $C>0$ and $\torus\in \Balancedtori_d(C)$. We define a small contour $\gamma$ in $\torus$ to be \emph{$C$-stable} if for all $|z| < \delta_1(d,C)$
\[
|w(\gamma;z)| \leq |z|^{||\gamma||}e^{5 e^{-C 3^d} \cdot  |\overline{\gamma}|}.
\]
We define a large contour in $\torus$ to be  \emph{$C$-stable} if for all $|z|< \delta_1(d,C)$ 
\[
|w(\gamma;z)| \leq |z|^{||\gamma||}e^{5 e^{-C 3^d} \cdot  |\overline{\gamma}|} \cdot e^{4}.
\]
\end{definition}

For an induced closed subgraph $\Lambda \subseteq \torus$ denote by $\mathcal{C}^\varphi_{\Lambda}(\torus, C)$ the set of clusters $X$ consisting of contours $\gamma$ in $\torus$ that are small and of type $\varphi$, $C$-stable and satisfy $\overline{\gamma} \subseteq \Lambda^\circ$.
Recall that the condition of being a small contour depends on the shortest side length $\ell_1$ of $\torus$.
When $\Lambda = \torus$ we write $\mathcal{C}^\varphi(\torus, C)$ instead of $\mathcal{C}^\varphi_{\torus}(\torus, C)$.

\begin{definition}\label{def:truncated partition functions are still contour partitio functions}
    Let $\torus \in \Balancedtori_d(C)$. For any induced closed subgraph $\Lambda \subseteq\torus$ and ground state $\varphi \in \{\text{even, odd}\}$ we define 
\[
Z_{\text{trunc}}^\varphi(\Lambda;z):= \sum_{\Gamma\in \mathcal{C}^\varphi_{\Lambda}(\torus, C)} \prod_{\gamma \in \Gamma} w(\gamma;z).
\]
   We also define 
    \[
Z_{\text{trunc}}^\text{large}(\torus;z):= \sum_{\substack{\Gamma\in \TorusCompatContours^{\text{large}}(\torus)\\ \text{all $\gamma \in \Gamma$ $C$-stable}}} \prod_{\gamma \in \Gamma} w(\gamma;z).
    \]
Note each of these partition functions is a polymer partition function.
\end{definition}

Analogous to Lemma \ref{lem: symmetry even odd tori} we also see
\[
Z_{\text{trunc}}^\text{even}(\torus;z)=Z_{\text{trunc}}^\text{odd}(\torus;z).
\]

\subsubsection{\texorpdfstring{Convergence of $\log Z^\varphi_{\text{trunc}}$}{Convergence of log(Z_trunc)}}
We apply the Koteck\'y-Preiss theorem to $Z^\varphi_{\text{trunc}}(\Lambda;z)$ for induced closed subgraphs $\Lambda \subseteq \torus$ and $\varphi \in \{\text{even, odd}\}$. 
The set of polymers is the set of small $C$-stable contours of type $\varphi$ in $\Lambda \subseteq \torus$, the weights of a polymer $\gamma$ is defined as $w(\gamma; z)$ and the compatibility relation is torus-compatibility. 
The cluster expansion takes the form
\begin{equation}\label{eq:cluster expansion trunc}
\log Z^\varphi_{\text{trunc}}(\Lambda;z) = \sum_{X \in \mathcal{C}^\varphi_{\Lambda}(\torus, C)} \Phi(X;z),
\end{equation}
where $\Phi(X;z) = \prod_{\gamma \in \Gamma} \frac{1}{n_X(\gamma)!} \psi(\gamma_1, \ldots, \gamma_n) \prod_{i=1}^n w(\gamma_i;z)$ is defined as in Section \ref{sec: cluster expansion}. 
 We define the \emph{support} of a cluster $X = \{\gamma_1, \ldots, \gamma_k\}$ to be $\overline{X} = \cup_{i=1}^k \overline{\gamma_i}$ and we denote by $|\overline{X}|$ the size of the vertex set of $\overline{X}$. Because $X$ is a cluster the incompatibility graph induced by $\gamma_1, \dots, \gamma_k$ is connected, which by definition of torus-compatibility implies that $\overline{X}$ is connected, because the $\gamma_i$ are small contours and thus themselves connected. 

\begin{theorem}\label{thm:logtrunc converges_tori}
Let $C>0$, $d\in \mathbb{Z}_{\geq 2}$ and $\varphi \in \{\text{even, odd}\}$. Let $\torus\in \Balancedtori_d(C)$ and let $\Lambda \subseteq \torus$ be any induced closed subgraph. For all $z \in \mathbb{C}$ with $|z| < \delta_1(d,C)$ the cluster expansion for $\log Z^\varphi_{\text{trunc}}(\Lambda;z)$ is convergent, where $\delta_1(d,C)$ defined in Definition \ref{def:deltad}. Furthermore for any $v \in \Lambda$ and any $|z| < \delta_1(d,C)$ we have
\[
\sum_{\substack{X \in \mathcal{C}^\varphi_{\Lambda}(\torus, C) \\ v \in \overline{X}}} |\Phi(X;z)| e^{ \sum_{\gamma\in X}C|\overline{\gamma}|} \leq 2.
\]
\end{theorem}
\begin{proof}
Fix $v \in \Lambda$. Define the artificial contour $v_\gamma$ with support $v$, weight $0$, and which is torus incompatible with each small contour $\gamma$ for which $v \in V(\overline{\gamma})$.
Add $v_\gamma$ to the set of small $C$-stable contours of type $\varphi$ in $\Lambda$. With the artificial contour added, $Z^{\varphi}_{\text{trunc}}(\Lambda;z)$ is still equal to the sum over torus-compatible collections of small contours of type $\varphi$, as the weight of $v_\gamma$ is zero.
For $|z| <\delta_1(d,C)$ we verify the condition of Theorem \ref{thm:KoteckyPreiss} with $a(\gamma) = 4d|\overline{\gamma}| $ and $b(\gamma) = C|\overline{\gamma}|$.  For any contour $\gamma$
\[
\sum_{\gamma'\not \sim \gamma } |w(\gamma';z)| e^{a(\gamma')+b(\gamma')} \leq \sum_{\gamma'\not \sim \gamma} |z|^{||\gamma'||} e^{5\cdot e^{-C 3^d} \cdot|\gamma'|} e^{a(\gamma')+b(\gamma')} \leq \sum_{\gamma'\not \sim \gamma} |z|^{\rho|\overline{\gamma'}|} e^{(4d+5\cdot e^{-C 3^d} +C)|\overline{\gamma'}|},
\]
where the sums run over non-artificial contours $\gamma'$. In the final inequality we used $||\gamma'|| \geq \rho |\overline{\gamma'}|$. 
Since $|z| <\delta_1(d,C)$ we have
\[
\sum_{\gamma'\not \sim \gamma} |z|^{\rho|\overline{\gamma'}|} e^{(4d+5\cdot e^{-C 3^d}+C) |\overline{\gamma'}|} < \sum_{\gamma'\not \sim \gamma} e^{-\log(2C_d) |\overline{\gamma'}|}.
\]
There are at most $(|\overline{\gamma}| + |\partial^c \overline{\gamma}|)C_d^m$ small contours $\gamma'\not \sim \gamma$ with $|\overline{\gamma'}| = m$, where $C_d$ is the constant from Lemma \ref{lem:contours bound connected}. 
This can be seen by upper bounding the number of small contours that is torus incompatible with a single vertex and applying this bound for each vertex of $\overline{\gamma}\cup\partial^c \overline{\gamma}$. 
Note that $|\partial^c \overline{\gamma}|<(2d-1)|\overline{\gamma}|$.
Hence 
\[
 \sum_{\gamma'\not \sim \gamma} e^{-\log(2C_d)  |\overline{\gamma'}|} < (|\overline{\gamma}| + |\partial^c \overline{\gamma}|) \cdot \sum_{m \geq 0} (C_d)^m e^{-\log(2C_d)  m} \leq 2 (|\overline{\gamma}| + |\partial^c \overline{\gamma}|)\leq 4d |\overline{\gamma}| = a(\gamma).
\]

This shows the condition of Theorem \ref{thm:KoteckyPreiss} holds, which implies the cluster expansion is convergent for $|z| < \delta_1(d,C)$. By Theorem \ref{thm:KoteckyPreiss} and the definition of $v_\gamma$ we have for any $v \in \Lambda$ and any $|z| < \delta_1(d,C)$ we obtain 
\[
\sum_{\substack{X \in \mathcal{C}^\varphi_{\Lambda}(\torus, C) \\ v \in \overline{X}}} |\Phi(X;z)| e^{ \sum_{\gamma\in X}C|\overline{\gamma}|} = \sum_{\substack{X \in \mathcal{C}^\varphi_{\Lambda}(\torus, C) \\X\not \sim v_\gamma} } |\Phi(X;z)| e^{b(X)} \leq a(v_\gamma)= 2.
\] 
\end{proof}

\subsubsection{All contours are stable}
To prove that all contours are stable we need some estimates on certain subseries of the cluster expansion.

\begin{lemma}\label{lemma:def half free enrgy}
    Let $C>0$ and let $\varphi\in \{\text{even},\text{odd}\}$. Then for any $z \in \mathbb{C}$ with $|z| < \delta_1(d,C)$ the limit
    \[
        \lim_{n \to \infty }
        \sum_{\substack{X \in \mathcal{C}^\varphi(\mathbb{Z}_{n}^{d}, C) \\ \vec{0} \in \overline{X},\ |\overline{X}|<n}} \frac{\Phi(X;z)}{|\overline{X}|}
    \]
    exists and is an analytic function of $z$.
\end{lemma}

\begin{proof}
First note for any $C>0$, there is an $N=N(d)>0$ such that for all $n \geq N$ we have $\mathbb{Z}_{n}^{d} \in \Balancedtori_d(C)$, as for large enough $n$ we have $e^{C \cdot n} \geq n^d$.
For each $n\geq N$ and each $z \in \mathbb{C}$ with $|z| < \delta_1(d,C)$ define the series
\[
S_{n}(z) :=  \sum_{\substack{X \in \mathcal{C}^\varphi(\mathbb{Z}_{n}^{d}, C) \\ \vec{0} \in \overline{X},\ |\overline{X}|<n}} \frac{\Phi(X;z)}{|\overline{X}|}.
\] 
By Theorem \ref{thm:logtrunc converges_tori} we see $|S_{n}(z)|\leq 2$ for all $n \geq N$ and all $|z| < \delta_1(d,C)$. Thus the family of maps $\{S_n\}_{n \geq N}$ is normal on $B_{\delta_1(d,C)}$, where $B_r$ denotes the open disk centered at $0$ with radius $r$. 
For $n_2> n_1$ any cluster $X$ of small contours in $\mathbb{Z}^d_{n_1}$ with $|\overline{X}| < n_1$ and $\vec{0} \in \overline{X}$ can be unambiguously viewed as a cluster in $\mathbb{Z}^d_{n_2}$ with $|\overline{X}| < n_1$ and $\vec{0} \in \overline{X}$.
From this and the fact that for any contour $\gamma$ we have $||\gamma|| \geq \rho |\overline{\gamma}|$, where $\rho = \rho(d)$ denotes the constant from Lemma \ref{lem:Peierls condition}, we see for $n_2> n_1$ that the first $\rho n_1$ coefficients of the power series expansions of $S_{n_1}(z)$ and $S_{n_2}(z)$ are the same.
Hence the coefficients of $S_{n}(z)$ are stabilizing, which implies that every convergent subsequence of $S_n$ converges to the same limit. Normality implies that the entire sequence converges to this limit.
\end{proof}

\begin{definition}\label{def:half free enrgy}
    We denote the the limit function in the lemma above by $f_{\varphi,C}(z)$.
\end{definition}
In fact, for the definition of $ f_{\varphi,C}(z)$ one can take any sequence of tori $\torus \in \Balancedtori_d(C)$ with increasing minimal side length $\ell_1$, as is implied by the following lemma.

\begin{lemma}\label{lem:equality for balanced tori and equal side length tori for small clusters}
    Let $C>0$ and let $\torus \in \Balancedtori_d(C)$. Denote the smallest side length of $\torus$ by $\ell_1$ and let $\varphi \in \{ \text{even, odd}\}$. For any $|z|<\delta_1(d,C)$ we have
 \begin{align*}
    \sum_{\substack{X\in \mathcal{C}^\varphi(\torus, C)\\  \vec{0} \in \overline{X},\ |\overline{X}|<\ell_1}} \frac{\Phi(X;z)}{|\overline{X}|}  =  \sum_{\substack{X\in \mathcal{C}^\varphi(\mathbb{Z}^{d}_{\ell_1}, C)\\  \vec{0} \in \overline{X},\ |\overline{X}|<\ell_1}} \frac{\Phi(X;z)}{|\overline{X}|}.
\end{align*}
\end{lemma}
\begin{proof}
As $\torus \in \Balancedtori_d(C)$ with minimal side length $\ell_1$, we have $\mathbb{Z}_{\ell_1}^{d} \in \Balancedtori_d(C)$.
Hence Theorem \ref{thm:logtrunc converges_tori} implies that both series are convergent for $|z| < \delta_1(d,C)$.
 Any cluster $X$ in either $\mathcal{C}^\varphi(\mathbb{Z}^{d}_{\ell_1}, C)$ or $\mathcal{C}^\varphi(\mathcal{T}, C)$ with $\vec{0} \in \overline{X}$ and $ |\overline{X}|<\ell_1$ can unambiguously be viewed as being supported on $\{-(\ell_1-1), \dots, \ell_1-1\}^{d}$ because $\overline{X}$ is connected. This yields a weight preserving bijection between the two sets, which implies the equality holds for all $|z| < \delta_1(d,C)$.
\end{proof}

The following estimate is well-known in the statistical physics literature. It is for example used in the proof of Lemma 5.3 of \cite{BorgsImbrie}, though no formal proof is given there. 
The proof we provide here is based on Section 5.7.1 in \cite{friedli_velenik_2017}, adapted to our setting.  

\begin{theorem}\label{thm: extracting volume boundary for tori}
Let $C>0$ and let $\torus \in \Balancedtori_d(C)$. Denote the smallest side length of $\torus$ by $\ell_1$ let $\varphi \in \{ \text{even, odd}\}$. Let $\Lambda \subseteq \torus$ be an induced closed subgraph. For any $|z| < \delta_1(d,C)$ we have 
\[
\left| \log Z^\varphi_\text{trunc}(\Lambda;z) - |\Lambda_{\text{even}}^{\circ}| f_{\varphi,C}(z)-|\Lambda_{\text{odd}}^{\circ}| f_{\overline{\varphi},C}(z)\right| \leq |\partial \Lambda| \cdot 2\cdot e^{-C 3^d}  + |\Lambda^\circ| \frac{4}{\ell_1 e^{C \ell_1}},
\]
where $f_{\varphi}(z)$ and $f_{\overline{\varphi}}(z)$ are the functions defined in Definition~\ref{def:half free enrgy}. 
\end{theorem}
\begin{proof}
For $|z| < \delta_1(d,C)$ we have the following equalities of convergent power series
\begin{align*}
    &\log Z^\varphi_{\text{trunc}}(\Lambda;z) = \sum_{X \in \mathcal{C}^\varphi_{\Lambda}(\torus, C)} \Phi(X;z) = \sum_{v \in \Lambda^\circ} \sum_{\substack{X \in \mathcal{C}^\varphi_{\Lambda}(\torus, C)\\v \in \overline{X} }} \frac{\Phi(X;z)}{|\overline{X}|}=\\
    & \sum_{v \in \Lambda^\circ} \Big(\sum_{\substack{X \in \mathcal{C}^\varphi(\torus, C)\\v \in \overline{X} }} \frac{\Phi(X;z)}{|\overline{X}|} - \sum_{\substack{X\in \mathcal{C}^\varphi(\torus, C)\\v \in \overline{X} \not \subset \Lambda^\circ}} \frac{\Phi(X;z)}{|\overline{X}|}\Big)=\\
    & |\Lambda_{\text{even}}^{\circ}| \sum_{\substack{X\in \mathcal{C}^\varphi(\torus, C)\\  \vec{0} \in \overline{X}}} \frac{\Phi(X;z)}{|\overline{X}|}+|\Lambda_{\text{odd}}^{\circ}| \sum_{\substack{X\in \mathcal{C}^{\overline{\varphi}}(\torus, C)\\  \vec{0} \in \overline{X}}} \frac{\Phi(X;z)}{|\overline{X}|}- \sum_{v \in \Lambda^\circ} \sum_{\substack{X\in \mathcal{C}^\varphi(\torus, C)\\v \in \overline{X} \not \subset \Lambda^\circ}} \frac{\Phi(X;z)}{|\overline{X}|},
\end{align*}
where in the final equality we use that a cluster of contours containing $v \in \Lambda^\circ$ can be translated to a cluster of contours containing $\vec{0}$; see Remark \ref{rmk: translating contours for tori}.

We prove the following bounds:
\begin{equation}\label{eq:bound 1}
|\sum_{v \in \Lambda^\circ} \sum_{\substack{X\in \mathcal{C}^\varphi(\torus, C)\\v \in \overline{X} \not \subset \Lambda^\circ}} \frac{\Phi(X;z)}{|\overline{X}|}|\leq  |\partial \Lambda|\cdot 2\cdot e^{-C 3^d},
\end{equation}
 and for $\xi\in \{\varphi,\overline{\varphi}\}$,
\begin{equation}\label{eq:bound 2}
\Bigg|\sum_{\substack{X\in \mathcal{C}^\xi(\torus, C)\\  \vec{0} \in \overline{X}}} \frac{\Phi(X;z)}{|\overline{X}|} - f_{\xi,C}(z)\Bigg|\leq \frac{4}{\ell_1 e^{C \ell_1}}.
\end{equation}
Since $|\Lambda^\circ| = |\Lambda_{\text{even}}^{\circ}| + |\Lambda_{\text{odd}}^{\circ}|$ these bounds together complete the proof.

To prove \eqref{eq:bound 1} we bound
\begin{align*}
     \Bigg|\sum_{v \in \Lambda^\circ} \sum_{\substack{X\in \mathcal{C}^\varphi(\torus, C)\\v \in \overline{X} \not \subset \Lambda^\circ}} \frac{\Phi(X;z)}{|\overline{X}|}\Bigg| &=  \Bigg|\sum_{\substack{X\in \mathcal{C}^\varphi(\torus, C)\\ \overline{X} \not \subset \Lambda^\circ}}  \sum_{v \in \Lambda^\circ}  \frac{\Phi(X;z) \mathbf{1}_{\overline{X}}(v)}{|\overline{X}|}\Bigg| = \Bigg|\sum_{\substack{X\in \mathcal{C}^\varphi(\torus, C)\\ \overline{X} \not \subset \Lambda^\circ,\ \overline{X} \cap \Lambda^\circ \neq \emptyset}}   \frac{\Phi(X;z) |\overline{X} \cap \Lambda^\circ|}{|\overline{X}|}\Bigg|\\
     &\leq  \sum_{w \in \partial \Lambda} \sum_{\substack{X\in \mathcal{C}^\varphi(\torus, C)\\ w\in \overline{X} \not \subset \Lambda^\circ}} \Big|\frac{\Phi(X;z)|\overline{X} \cap \Lambda^\circ|}{|\overline{X}|}\Big|\leq  \sum_{w \in \partial \Lambda} \sum_{\substack{X\in \mathcal{C}^\varphi(\torus, C)\\ w\in \overline{X} \not \subset \Lambda^\circ}}|\Phi(X;z)| \\
     &\leq |\partial \Lambda| \max_{w \in \partial \Lambda} \sum_{\substack{X\in \mathcal{C}^\varphi(\torus, C)\\ w\in \overline{X} \not \subset \Lambda^\circ}} |\Phi(X;z)| 
    \leq |\partial \Lambda|\cdot 2\cdot e^{-C 3^d},
\end{align*}
where the last inequality follows from Theorem \ref{thm:logtrunc converges_tori} using that any cluster $X$ with $w \in \overline{X}$ satisfies $\sum_{\gamma \in X} |\overline{\gamma}| \geq 3^d$.

Next we show~\eqref{eq:bound 2}. We split the clusters in $\torus$ based on size and use the triangle inequality
\begin{align*}
 \Bigg|\sum_{\substack{X\in \mathcal{C}^\xi(\torus, C)\\  \vec{0} \in \overline{X}}} \frac{\Phi(X;z)}{|\overline{X}|} - f_{\xi,C}(z)\Bigg|&\leq 
 \Bigg|\sum_{\substack{X\in \mathcal{C}^\xi(\torus, C)\\  \vec{0} \in \overline{X},\ |\overline{X}|<\ell_1}} \frac{\Phi(X;z)}{|\overline{X}|} - f_{\xi,C}(z)\Bigg|+  \Bigg| \sum_{\substack{X\in \mathcal{C}^\xi(\torus, C)\\  \vec{0} \in \overline{X},\ |\overline{X}|\geq\ell_1}} \frac{\Phi(X;z)}{|\overline{X}|}\Bigg|\\
 &\leq  \Bigg|\sum_{\substack{X\in \mathcal{C}^\xi(\torus, C)\\  \vec{0} \in \overline{X},\ |\overline{X}|<\ell_1}} \frac{\Phi(X;z)}{|\overline{X}|} - f_{\xi,C}(z)\Bigg|+ \frac{2}{\ell_1 \cdot e^{C \ell_1}}\\
 &=\Bigg|\sum_{\substack{X\in \mathcal{C}^\xi(\mathbb{Z}^{d}_{\ell_1}, C)\\  \vec{0} \in \overline{X},\ |\overline{X}|<\ell_1}} \frac{\Phi(X;z)}{|\overline{X}|} - f_{\xi,C}(z)\Bigg|+ \frac{2}{\ell_1 \cdot e^{C \ell_1}},
\end{align*}
where the last inequality follows from Theorem \ref{thm:logtrunc converges_tori} and the last equality follows from Lemma \ref{lem:equality for balanced tori and equal side length tori for small clusters}.

For any $\epsilon>0$ there is an $\ell^*$ large enough such that for any $|z| \leq \delta_1(d,C)$ we have 
\[
\Bigg|\sum_{\substack{X\in \mathcal{C}^\xi(\mathbb{Z}^{d}_{\ell^*}, C)\\  \vec{0} \in \overline{X},\ |\overline{X}|<\ell^*}} \frac{\Phi(X;z)}{|\overline{X}|} - f_{\xi,C}(z)\Bigg| \leq \epsilon,
\]
by Lemma \ref{lemma:def half free enrgy}. By increasing $\ell^*$ if necessary we may assume $\ell^* > 2\ell_1$.
We have
\begin{align*}
    \Bigg|\sum_{\substack{X\in \mathcal{C}^\xi(\mathbb{Z}^{d}_{\ell_1}, C)\\  \vec{0} \in \overline{X},\ |\overline{X}|<\ell_1}} \frac{\Phi(X;z)}{|\overline{X}|} - f_{\xi,C}(z)\Bigg| &\leq  \Bigg|\sum_{\substack{X\in \mathcal{C}^\xi(\mathbb{Z}^{d}_{\ell_1}, C)\\  \vec{0} \in \overline{X},\ |\overline{X}|<\ell_1}} \frac{\Phi(X;z)}{|\overline{X}|} -\sum_{\substack{X\in \mathcal{C}^\xi(\mathbb{Z}^{d}_{\ell^*}, C)\\  \vec{0} \in \overline{X},\ |\overline{X}|<\ell^*}} \frac{\Phi(X;z)}{|\overline{X}|}  \Bigg| + \epsilon \\
    &= \Bigg| \sum_{\substack{X\in \mathcal{C}^\xi(\mathbb{Z}^{d}_{\ell^*}, C)\\  \vec{0} \in \overline{X},\ \ell_1 \leq |\overline{X}|<\ell^*}} \frac{\Phi(X;z)}{|\overline{X}|}  \Bigg| + \epsilon \leq \frac{2}{\ell_1 e^{C \ell_1}} + \epsilon,
\end{align*}
where we used Theorem \ref{thm:logtrunc converges_tori} in the last inequality. As this holds for any $\epsilon>0$, we see 
\[
\Bigg|\sum_{\substack{X\in \mathcal{C}^\xi(\torus, C)\\  \vec{0} \in \overline{X},\ |\overline{X}|<\ell_1}} \frac{\Phi(X;z)}{|\overline{X}|} - f_{\xi,C}(z)\Bigg|\leq \frac{2}{\ell_1 e^{C \ell_1}}.
\]
This finishes the proof of~\eqref{eq:bound 2}. 
\end{proof}

With this bound, we can now finally show that all contours are $C$-stable.

\begin{theorem}\label{thm:stable contours are all contours}
Let $C>0$ and let $\torus \in \Balancedtori_d(C)$. Let $\Lambda \subseteq \torus$ be an induced closed subgraph of $\torus$ and let $\varphi \in \{\text{even}, \text{odd}\}$ be a ground state. For all $|z| < \delta_1(d,C)$, we have
\[
Z^\varphi_\text{match}(\Lambda;z) = Z^\varphi_{\text{trunc}}(\Lambda;z)
\]
and
\[
Z^{\text{large}}_\text{match}(\torus;z) = Z^{\text{large}}_{\text{trunc}}(\torus;z).
\]
\end{theorem}
\begin{proof}
We first prove by induction on $|\Lambda|$ that all small contours of type $\varphi$ are $C$-stable. The base case follows as the weight of an empty contour is $1$. Suppose the claim holds for all $\Lambda$ with $|\Lambda| \leq k$ for some $k \geq 0$. Let $\Lambda$ be such that $|\Lambda| = k+1$ and take any small contour $\gamma$ of type $\varphi$ in $\Lambda$. We aim to bound $|w(\gamma;z)|$ by $|z|^{||\gamma||} e^{5\cdot e^{-C3^d}|\overline{\gamma}|}$, which shows $\gamma$ is $C$-stable. By the induction hypothesis we see
\[
\left|\frac{w(\gamma;z)}{z^{||\gamma||}}\right| = \bigg |\frac{Z^{\overline{\varphi}}_\text{match} (\inte_{\overline{\varphi}}(\gamma);z)}{Z^{\varphi}_\text{match} (\inte_{\overline{\varphi}}(\gamma);z)} \bigg | = \bigg |\frac{Z^{\overline{\varphi}}_{\text{trunc}} (\inte_{\overline{\varphi}}(\gamma);z)}{Z^{\varphi}_{\text{trunc}} (\inte_{\overline{\varphi}}(\gamma);z)} \bigg|.
\]
Write $V = \inte_{\overline{\varphi}}(\gamma)$. 
For any $\epsilon>0$ there exists $\ell$ large enough such that $|V^\circ| \frac{4}{\ell e^{C \ell}} < \epsilon$ and such that $V$ is isomorphic to an induced closed subgraph of $\mathbb{Z}^{d}_{\ell}$. Fix such an $\ell.$


Define
\[
h_{\varphi,C}(V;z) = |V_{\text{even}}^{\circ}| f_{\varphi,C}(z)+|V_{\text{odd}}^{\circ}| f_{\overline{\varphi},C}(z),
\]
where $f_{\varphi,C}(z)$ and $f_{\overline{\varphi},C}(z)$ denote the functions defined in Definition~\ref{def:half free enrgy}. 
We also write $g_{C}(z) = f_{\varphi,C}(z) - f_{\overline{\varphi},C}(z).$ By Theorem \ref{thm:logtrunc converges_tori} and Lemma \ref{lemma:def half free enrgy} we see for any $z \in \mathbb{C}$ with  $|z| \leq \delta_1(d,C)$ we have
\[
|g_{C}(z)| \leq |f_{\varphi,C}(z)|+|f_{\overline{\varphi},C}(z)| \leq \frac{4}{3^d} \cdot e^{-C3^d} \leq  e^{-C3^d},
\]
using the fact that any cluster $X$ with $\vec{0} \in \overline{X}$ satisfies $|\overline{X}| \geq 3^d$.

Theorem \ref{thm: extracting volume boundary for tori} applied to $V$ as a induced closed subgraph of $\mathbb{Z}_\ell^d$ now gives 
\begin{align}
   \left |\frac{Z^{\overline{\varphi}}_{\text{trunc}} (V;z)}{Z^{\varphi}_{\text{trunc}} (V;z)} \right| &
   = \left |\frac{e^{\log Z^{\overline{\varphi}}_{\text{trunc}} (V;z) - h_{\overline{\varphi},C}(V;z)}}{e^{\log Z^{\varphi}_{\text{trunc}} (V;z)-h_{\varphi,C}(V;z)}} \right| \cdot \left|\frac{e^{h_{\overline{\varphi},C}(V;z)}}{e^{h_{\varphi,C}(V;z)}}\right| 
   \leq \frac{e^{|\log Z^{\overline{\varphi}}_{\text{trunc}} (V;z) - h_{\overline{\varphi},C}(V;z)|}}{e^{-|\log Z^{\varphi}_{\text{trunc}} (V;z)-h_{\varphi,C}(V;z)|}} \cdot \left|\frac{e^{h_{\overline{\varphi},C}(V;z)}}{e^{h_{\varphi,C}(V;z)}}\right| \nonumber
   \\ 
   &<\frac{e^{2  e^{-C 3^d} |\partial V|+|V^\circ|\tfrac{4}{\ell e^{C\ell}}}}{e^{-2  e^{-C 3^d} |\partial V|-|V^\circ|\tfrac{4}{\ell e^{C\ell}}}} \left|\frac{e^{h_{\overline{\varphi},C}(V;z)}}{e^{h_{\varphi,C}(V;z)}}\right|
   \leq 
   e^{4e^{-C3^d}\cdot |\partial V|+|V^\circ \tfrac{8}{\ell e^{C\ell}}} \cdot \left|\frac{e^{h_{\overline{\varphi},C}(V;z)}}{e^{h_{\varphi,C}(V;z)}}\right|.
   \label{eq:bound ratio trunc general}
 \end{align} 
 
Using that $|V^\circ|\tfrac{4}{\ell e^{C\ell}}<\varepsilon$ and the definitions of $h_{\varphi,C}(V;z)$  and $g_{C}(z)$ we can further bound this by
   \begin{align}
  e^{4 \cdot e^{-C 3^d}|\partial V|} \cdot e^{2\epsilon} \cdot \left|e^{(|V_{\text{even}}^{\circ}|-|V_{\text{odd}}^{\circ}|)g_{C}(z)}\right| 
   \leq e^{4 \cdot e^{-C 3^d} |\partial V|} \cdot e^{2\epsilon} \cdot e^{  e^{-C3^d} \left| |V_{\text{even}}^{\circ}|-|V_{\text{odd}}^{\circ}|\right| }.\label{eq:bound ratio trunc small}
\end{align}

 We next claim that for any induced closed subgraph $V \subseteq \torus$ it holds that 
 \begin{equation}\label{eq:V even odd}
 \left||V_{\text{even}}^{\circ}|-|V_{\text{odd}}^{\circ}|\right| < |\partial V|.
 \end{equation}
 
 Indeed, define $e(A)$ for $A \subseteq V$ to be the set of edges of $\torus$  with at least one endpoint in $A$. 
 We have $|e(V)| = |e(V_{\text{even}}^{\circ})| + |e((\partial V)_{\text{even}})| = |e(V_{\text{odd}}^{\circ})| + |e((\partial V)_{\text{odd}})|$.
 As each vertex in $V^\circ$ has degree $2d$ and each vertex in $\partial V$ has degree strictly less than $2d$ we see that
 \[
\left||V_{\text{even}}^{\circ}|-|V_{\text{odd}}^{\circ}|\right| = \frac{1}{2d} \left||e(V_{\text{even}}^{\circ})| -|e(V_{\text{odd}}^{\circ})|\right|  = \frac{1}{2d} \left||e((\partial V)_{\text{odd}})|-|e((\partial V)_{\text{even}})|\right| < |\partial V|,
 \]
 proving~\eqref{eq:V even odd}.
Substituting~\eqref{eq:V even odd} into~\eqref{eq:bound ratio trunc small} we get for any $\epsilon>0$, 
\[
\left |\frac{Z^{\overline{\varphi}}_{\text{trunc}} (V;z)}{Z^{\varphi}_{\text{trunc}} (V;z)} \right| < e^{2\epsilon} \cdot e^{5\cdot e^{-C 3^d} \cdot  |\partial V|} \leq e^{2\epsilon} \cdot e^{5\cdot e^{-C 3^d} \cdot  |\overline{\gamma}|}.
\]
As $\epsilon \to 0$ we get 
\[
\left |\frac{Z^{\overline{\varphi}}_{\text{trunc}} (V;z)}{Z^{\varphi}_{\text{trunc}} (V;z)} \right|  \leq  e^{5\cdot e^{-C 3^d} \cdot  |\overline{\gamma}|}
\]
and hence we see that the small contour $\gamma$ is $C$-stable. 

Now let $\gamma$ be a large contour. As we already proved that for any induced closed subgraph $\Lambda \subseteq \torus$ all small contours are $C$-stable, we obtain
\[
\left|\frac{w(\gamma;z)}{z^{||\gamma||}}\right| = \left| \frac{Z^{\text{odd}}_\text{match} (\inte_{\text{odd}}(\gamma);z)}{Z^{\text{even}}_\text{match} (\inte_{\text{odd}}(\gamma);z)}\right| = \left| \frac{Z^{\text{odd}}_{\text{trunc}} (\inte_{\text{odd}}(\gamma);z)}{Z^{\text{even}}_{\text{trunc}} (\inte_{\text{odd}}(\gamma);z)}\right|.
\]
Again write $V = \inte_{\text{odd}}(\gamma)$, and write
\[
h_{\varphi,C}(V;z) = |V_{\text{even}}^{\circ}| f_{\varphi,C}(z)+|V_{\text{odd}}^{\circ}| f_{\overline{\varphi},C}(z),
\]
and $g_{C}(z) = f_{\varphi,C}(z) - f_{\overline{\varphi},C}(z)$. 
As above we have $|g_{C}(z)| \leq e^{-C3^d}$ for any $z\in \mathbb{C}$ with $|z| \leq \delta_1(d,C)$.

By Theorem \ref{thm: extracting volume boundary for tori}, now applied to $V$ as an induced closed subgraph of $\torus$, and thus replacing $\ell$ by $\ell_1$ in~\eqref{eq:bound ratio trunc general} we obtain,
\begin{align*}
   \left |\frac{Z^{\overline{\varphi}}_{\text{trunc}} (V;z)}{Z^{\varphi}_{\text{trunc}} (V;z)} \right| &= 
   e^{4 \cdot e^{-C 3^d} |\partial V|+|V^\circ| \frac{8}{\ell_1 e^{C \ell_1}}} \cdot \left|e^{(|V_{\text{even}}^{\circ}|-|V_{\text{odd}}^{\circ}|)g_{C}(z)}\right|\\
  & \leq e^{4 \cdot e^{-C 3^d} |\partial V|+|V^\circ| \frac{8}{\ell_1 e^{C \ell_1}}} \cdot e^{ e^{-C 3^d} |\partial V| }
   \\
   &\leq e^{5\cdot e^{-C 3^d} \cdot  |\partial V|} \cdot e^{|V^\circ| \frac{8}{\ell_1 e^{C \ell_1}}}\leq e^{5\cdot e^{-C 3^d} \cdot  |\overline{\gamma}|} \cdot e^{4},
\end{align*}
using~\eqref{eq:V even odd} and the bound on $g_{C}(z)$ for the second to last inequality and $e^{C \ell_1} \geq |\torus| \geq |V^\circ|$ and $\ell_1 \geq 2$ for the final inequality.
Therefore each large contour $\gamma$ is $C$-stable.
\end{proof}

\section{Bounded zeros for balanced tori}\label{sec:boundedzerosbalancedtori}

In this section we prove the zeros of families of balanced tori are bounded, building on the framework and results of the previous section.

Recall by Lemma \ref{lem: symmetry even odd tori} and Theorem \ref{thm:stable contours are all contours} that we have
\[
Z_{\text{match}}(\torus;z) = 2Z_{\text{trunc}}^{\text{even}}(\torus;z) + Z_{\text{trunc}}^{\text{large}}(\torus;z).
\]
Our first aim this section is to bound $|Z^{\text{large}}_\text{match}(\torus;z))|$ away from $2|Z_{\text{trunc}}^{\text{even}}(\torus;z)|$, these bounds are the final ingredient we need to prove the zeros of families of balanced tori are bounded.

To obtain bounds on $|Z^{\text{large}}_\text{match}(\torus;z))|$ we apply the Koteck\'y-Preiss
theorem to $\log(Z^{\text{even}}_\text{match}(\torus;z)+Z^{\text{large}}_\text{match}(\torus;z))$, which in turn means we need to bound the number of relevant contours.

\begin{lemma}\label{lem:large contours bound}
Let $\torus$ be an even $d$-dimensional torus. Let $L_m$ denote the set of contours $\gamma$ in $\torus$ containing $\vec{0}$ with support of size $m$ that are either large or small and of even type. Then we have 
$$
|L_m| \leq (4C_d|\torus|^{1/\ell_1})^m,
$$
where $C_d$ is the constant from Lemma \ref{lem:contours bound connected}.
\end{lemma}
\begin{proof}
Let $k$ denote the number of connected components of a large contour $\gamma_{\text{large}}$ with $ |\overline{\gamma_{\text{large}}}| = m$.
Each connected component of $\gamma_\text{large}$ has size at least $\ell_1$, hence $k \leq \lfloor m/\ell_1 \rfloor$. Denote by $m_i$ the size of the $i$-th connected component for $i \in \{1, \ldots, k\}$. For each connected component of the large contour that does not contain $\vec{0}$ choose a vertex $v_i$ of $\torus$ in the component for $i \in \{1, \ldots, k-1\}$, this can be done in $|\torus|^{k-1}$ many ways.

Denote by $P_{l}$ the set of connected large contours of size $l$ incompatible with a specified vertex $v$.
The number of connected sets in $\torus$ of size $l$ containing $v$ is bounded by the number of connected sets of size $l$ containing $\vec{0}$ in $\mathbb{Z}^d$. 
As there are at most $2^{l}$ possible feasible configurations on a set of size $l$, we obtain with the same argument as in Lemma \ref{lem:contours bound connected} that $|P_l| \leq C_d^l$.
We apply this bound to each connected component and see the total number of large contours $\gamma$ in $\torus$ with support of size $m$ containing $\vec{0}$ is bounded by 
\begin{equation}\label{eq:bound large contours}
  \sum_{\substack{m_1, \ldots, m_k \\ \sum m_i = m \text{ and } m_i \geq \ell_1}}\prod_{i=1}^k C_{d}^{m_i} (|\torus|)^{k-1}  \leq  \big( \sum_{\substack{m_1, \ldots, m_k \\ \sum m_i = m }} 1\big) C_d^m |\torus|^{k-1} \leq 4^m C_d^m |\torus|^{k-1}.
\end{equation}
Accounting also for the small even contours of size $m$, we get 
\[
|L_m| \leq 4^m C_d^m (|\torus|)^{k} \leq (4C_d|\torus|^{1/\ell_1})^m.
\]
\end{proof}

We also need a tighter bound on the absolute value of $|z|$. 

\begin{definition}\label{def:delta2}
   We define for any $x>0$ the number
\[
\delta_2(d,x) = e^{-(\log(8e^x C_d)+4de^{4}+5\cdot e^{-x 3^d})/\rho(d)},
\]
where $C_d$ is the constant from Lemma \ref{lem:contours bound connected} and $\rho(d) = \frac{1}{2d\cdot 3^d}$ is the constant from lemma \ref{lem:Peierls condition}.
\end{definition}
 
Note that $\delta_1(d,x)>\delta_2(d,x)$ for all $d\in \mathbb{Z}_{\geq 2}$ and $x>0$. We apply the framework of Section \ref{sec: cluster expansion} to $\log(Z^{\text{even}}_\text{match}(\torus;z)+Z^{\text{large}}_\text{match}(\torus;z))$. In this case our polymers are contours of type even, i.e. both large and small. The weights of a contour $\gamma$ equals $w(\gamma; z)$ and the compatibility relation is torus-compatibility. Denote by $\mathcal{C}^{\text{even}}_{\text{large}}(\torus)$ the set of clusters of even and large contours.
The cluster expansion takes the form
\begin{equation}\label{eq:cluster expansion large even}
\log(Z^{\text{even}}_\text{match}(\torus;z)+Z^{\text{large}}_\text{match}(\torus;z)) = \sum_{X \in \mathcal{C}^{\text{even}}_{\text{large}}(\torus)} \Phi(X;z),
\end{equation}
where $\Phi(X;z) = \prod_{\gamma \in \Gamma} \frac{1}{n_X(\gamma)!} \psi(\gamma_1, \ldots, \gamma_n) \prod_{i=1}^n w(\gamma_i;z)$ is defined as in Section \ref{sec: cluster expansion}.

\begin{theorem}\label{thm:logtrunc converges for even plus big}
Let $C>0$ and let $\torus \in \Balancedtori_d(C)$. For any $|z| < \delta_2(d,C)$ the cluster expansion for $\log(Z^{\text{even}}_\text{match}(\torus;z)+Z^{\text{large}}_\text{match}(\torus;z))$ is convergent, where $\delta_2(d,C)$ is defined in Definition \ref{def:delta2}. Furthermore for any $v \in V(\torus)$
\[
\sum_{\substack{X \in \mathcal{C}^{\text{even}}_{\text{large}}(\torus) \\ v \in \overline{X}}} |\Phi(X;z)| \leq 4de^4.
\]
\end{theorem}
\begin{proof}
Fix $v \in V(\torus)$. Define the artificial contour $v_\gamma$ with support $v$, weight $0$ and which is torus incompatible with each contour $\gamma$ such that $v \in V(\overline{\gamma})$.
Add $v_\gamma$ to the set of contours. With the artificial contour added, $Z^{\text{even}}_\text{match}(\torus;z)+Z^{\text{large}}_\text{match}(\torus;z)$ is still equal to the sum over torus-compatible collections of large and even contours, as the weight of $v_\gamma$ is zero.
 Throughout this proof $\sim$ denotes the relation of torus-compatibility.
We verify the condition of Theorem \ref{thm:KoteckyPreiss} with $a(\gamma) = 4de^4|\overline{\gamma}|$ and $b(\gamma) = 0$.

Theorem \ref{thm:stable contours are all contours} applies as $\delta_2(d,C) < \delta_1(d,C)$. 
Hence for any contour $\gamma$ 
\[
\sum_{\gamma'\not \sim \gamma } |w(\gamma';z)| e^{a(\gamma)+b(\gamma)} \leq \sum_{\gamma'\not \sim \gamma } |z|^{\rho |\overline{\gamma'}|}e^{(4de^{4}+5\cdot e^{-C 3^d} )|\overline{\gamma'}|} \cdot e^{4}\leq \sum_{\gamma'\not \sim \gamma } |z|^{\rho |\overline{\gamma'}|}e^{(4de^{4}+5\cdot e^{-C 3^d})|\overline{\gamma'}|} \cdot e^{4},
\]
where without loss of generality we may assume the first sum is over non-artificial contours $\gamma'$, as $w(v_\gamma;z) = 0$.
As $|z| <  \delta_2(d,C) = e^{-(\log(8e^C C_d)+4de^{4}+5\cdot e^{-C 3^d})/\rho}$, we have 
\[
\sum_{\gamma'\not \sim \gamma} |z|^{\rho|\overline{\gamma'}|} e^{(4de^{4}+5\cdot e^{-C 3^d}) |\overline{\gamma'}|} e^{4}< e^{4} \sum_{\gamma'\not \sim \gamma} (8e^CC_d)^{- |\overline{\gamma'}|}.
\]

There are at most $ (|\overline{\gamma}| + |\partial^c \overline{\gamma}|)(4|\torus|^{1/\ell_1} C_d)^m$ contours $\gamma'\not \sim \gamma$ with $|\overline{\gamma'}| = m$, where $C_d$ is the constant from Lemma \ref{lem:contours bound connected}, this can be seen by upper bounding how many ways a contour can be torus incompatible with a single vertex using Lemma \ref{lem:large contours bound} and applying this bound for each vertex of $\overline{\gamma}$. 
We also note that as $\torus$ is $C$-balanced, we have $|\torus|^{1/\ell_1} \leq e^C$.
Hence
\[
    \sum_{\gamma'\not \sim \gamma} (8e^CC_d)^{- |\overline{\gamma'}|} < e^{4}(|\overline{\gamma}| + |\partial^c \overline{\gamma}|) \cdot \sum_{m \geq 0} (4|\torus|^{1/\ell_1} C_d)^m (8e^CC_d)^{- m} \leq 2de^{4}|\overline{\gamma}| \cdot \sum_{m \geq 0} \left(\tfrac{1}{2}\right)^m = a(\gamma),
\]
 where we used $|\overline{\gamma}| + |\partial^c \overline{\gamma}| \leq 2d|\overline{\gamma}|$. 
 
This shows the condition of Theorem \ref{thm:KoteckyPreiss} holds, which implies the cluster expansion is convergent for $|z| < \delta_2(d,C)$. By Theorem \ref{thm:KoteckyPreiss} and the definition of $v_\gamma$ we have for any $v \in \torus$ we have
\[
\sum_{\substack{X \in \mathcal{C}^{\text{even}}_{\text{large}}(\torus) \\ v \in \overline{X}}} |\Phi(X;z)| e^{ \sum_{\gamma\in X}b(\gamma)} = \sum_{\substack{X \in \mathcal{C}^{\text{even}}_{\text{large}}(\torus) \\X\not \sim v_\gamma} } |\Phi(X;z)| \leq a(v_\gamma) = 4de^4,
\] 
where we can assume the clusters $X$ do not contain $v_\gamma$, as for any cluster $X$ containing $v_\gamma$ we have $\Phi(X;z) = 0$ as $w(v_\gamma;z) = 0$.
\end{proof}


\begin{lemma}\label{lem:normal}
    Let $C>0$. The family 
    \begin{equation}\label{eq:family of log}
    \Big\{ \frac{\log(Z^{\text{even}}_\text{match}(\torus;z)+Z^{\text{large}}_\text{match}(\torus;z))}{|\torus|}\Big\}_{\torus \in \Balancedtori_d(C)}
    \end{equation}
    is normal on $|z| < \delta_2(d,C)$.
\end{lemma}
\begin{proof}
   For any $\torus \in \Balancedtori_d(C)$ and any $z$ such that $|z| < \delta_2(d,C)$ we have
   \begin{align*}
    \Big|\frac{\log(Z^{\text{even}}_\text{match}(\torus;z)+Z^{\text{large}}_\text{match}(\torus;z))}{|\torus|}\Big| & = \Bigg| \frac{1}{|\torus|} \sum_{v \in V(\torus)} \sum_{\substack{X \in \mathcal{C}^{\text{even}}_{\text{large}}(\torus) \\ v \in \overline{X}}} \frac{\Phi(X;z)}{|\overline{X}|}\Bigg|\\
    &\leq \max_{v \in V(\torus)} \sum_{\substack{X \in \mathcal{C}^{\text{even}}_{\text{large}}(\torus) \\ v \in \overline{X}}} \frac{|\Phi(X;z)|}{|\overline{X}|} \leq 4de^4,
   \end{align*}
   where the last inequality follows from Theorem \ref{thm:stable contours are all contours}, Theorem \ref{thm:logtrunc converges for even plus big}. 
   Therefore the family defined in~\eqref{eq:family of log} is normal by Montel's theorem.

\end{proof}

To bound $|Z^{\text{large}}(\torus;z))|$, we show the influence of adding large contours to the even contours is negligible, for small enough $z$ as the sizes of the tori tend to infinity.

\begin{lemma}\label{lem:limit free energy same with or without large contours}
    For any $C>0$ and any $|z| < \delta_1(d,C)$ the function
    \[
f_C(z) = \lim_{\substack{|\torus| \to \infty\\ \torus \in \Balancedtori_d(C)}} \frac{\log (Z^{\text{even}}_\text{match}(\torus;z))}{|\torus|}
    \]
   is well-defined and $f_C(z) = \frac{1}{2} f_{\text{even},C}(z) + \frac{1}{2} f_{\text{odd},C}(z)$. For any $|z| < \delta_2(d,C)$ the function
     \[
g_{C}(z) = \lim_{\substack{|\torus| \to \infty\\ \torus \in \Balancedtori_d(C)}} \frac{\log(Z^{\text{even}}_\text{match}(\torus;z)+Z^{\text{large}}_\text{match}(\torus;z))}{|\torus|}
    \]
   is well-defined and $g_{C}(z) = f_C(z)$.
\end{lemma}
\begin{proof}
Take any torus $\torus \in \Balancedtori_d(C)$ and let $\ell_1$ denote the minimal side length of $\torus$. 
From Theorem \ref{thm:stable contours are all contours} and Theorem \ref{thm: extracting volume boundary for tori} we obtain for all $|z| < \delta_1(d,C)$ 
\[
\left| \frac{\log Z^{\text{even}}_\text{match}(\torus;z)}{|\torus|} - \frac{1}{2} f_{\text{even},C}(z) - \frac{1}{2} f_{\text{odd},C}(z) \right| < \frac{2}{|\torus|},
\]
where in the last equality we used $\ell_1 e^{C \ell_1} \geq 2 |\torus|$, as $\ell_1 \geq 2$ and $e^{C \ell_1} \geq  |\torus|$.
This implies $f_C(z)$ exists and $f_C(z) =  \frac{1}{2} f_{\text{even},C}(z) + \frac{1}{2} f_{\text{odd},C}(z)$.

The first $\rho \ell_1$ terms of the power series $\log Z^{\text{even}}_\text{match}(\torus;z)$ and $\log(Z^{\text{even}}_\text{match}(\torus;z)+Z^{\text{large}}_\text{match}(\torus;z))$ are equal, where $\rho$ is the constant from Lemma \ref{lem:Peierls condition} as each large contour contributes at least $z^{\rho \ell_1}$ in each cluster $X$ containing the large contour. Therefore, as $|\torus| \to \infty$, and hence $\ell_1 \to \infty$, the first $\rho \ell_1$ coefficients of \[\frac{\log(Z^{\text{even}}_\text{match}(\torus;z)+Z^{\text{large}}_\text{match}(\torus;z))}{|\torus|}\] converge to the first $\rho \ell_1$ coefficients of $f_C(z)$.
Lemma~\ref{lem:normal} now implies that $g_{C}(z)$ is well-defined on $B_{\delta_2(d,C)}(0)$ and moreover satisfies $g_{C}(z) = f_C(z)$.
\end{proof}

\begin{remark}
The function $f_C(z)$ is the free energy per site for the polymer model with polymers the small even contours in $\torus$. It is related to the free energy per site for the independence polynomial defined in the introduction, which we denote by $\rho(\lambda)$.
For $\lambda \in \mathbb{R}_{\geq 0}$ and $|\lambda| > 1/\delta_1(d,C)$
     both functions are well-defined and satisfy
     $\rho(\lambda) = \frac{\lambda}{2} + f_C(\frac{1}{\lambda})$.
\end{remark}

The following lemma provides sufficient conditions to bound $|Z^{\text{large}}_\text{match}(\torus;z)|$ away from $2|Z^{\text{even}}_\text{match}(\torus;z)|$, which is the final ingredient to prove zeros are bounded for $C$ balanced tori.

\begin{lemma}\label{lem: useful Han's lemma}
    Suppose there exists $\delta>0$ and for each $n\in \mathbb{N}$ there are holomorphic functions  $f_n,g_n: B_\delta(0) \rightarrow \mathbb{C}$ and functions $a,b: \mathbb{N} \to \mathbb{N}$ such that for all $z \in B_\delta(0) $ we have
    \begin{enumerate}
        \item the functions $f_n(z)+z^{a(n)} g_n(z)$ and $f_n(z)$ are nonzero for each $n$,
        \item $\lim_{n \to \infty} \frac{1}{b(n)} \cdot (\log(f_n(z)+z^{a(n)} g_n(z)) - \log f_n(z)) =0$,
        \item there is a constant $\kappa>0$ such that for all $n$ we have $a(n) > \kappa \log b(n)$,
    \end{enumerate}
    then there is a constant $N>0$ such that for $|z|< \frac{\delta}{e^{1/\kappa}}$ and $n \geq N$ we have 
    $$
    |z^{a(n)} g_n(z)| < |f_n(z)|.
    $$
    Furthermore, the inequality in (3) is necessary. 
\end{lemma}

\begin{proof}
 We have 
 \[
 \log(f_n(z) + z^{a(n)} g_n(z)) - \log f_n(z) = \log \Big(1+ \frac{z^{a(n)} g_n(z)}{f_n(z)}\Big) = z^{a(n)} h_n(z)
 \]
 for some convergent power series $h_n(z)$, using item (1). By item (2) we see  $\lim_{n \to \infty} \frac{z^{a(n)} h_n(z)}{b(n)} = 0$ for $|z| < \delta$. Hence for any $\epsilon>0$ and large enough $n$ we have $\left| \frac{z^{a(n)} h_n(z)}{b(n)} \right| < \epsilon$. By the maximum principle we see $\left| \frac{ h_n(z)}{b(n)} \right| < \frac{\epsilon}{\delta^{a(n)}}$. Now take $|z| < \frac{\delta}{e^{1/\kappa}}$ where $\kappa>0$ is the constant from item (3), then $|z|^{a(n)}\left| \frac{ h_n(z)}{b(n)} \right| < \frac{\epsilon}{e^{a(n)/\kappa}} $ therefore
$|z^{a(n)}| |h_n(z)| < \frac{b(n)}{e^{a(n)/\kappa}} \epsilon < \epsilon$, by item (3) of the assumptions. Therefore 
\[
\left|\log \Big(1+ \frac{z^{a(n)} g_n(z)}{f_n(z)}\Big)\right| = |z^{a(n)}| |h_n(z)| <\epsilon,
\]
for large $n$ and $|z| < \frac{\delta}{e^{1/\kappa}}$. From this we conclude $|z^{a(n)} g_n(z)| < \epsilon |f_n(z)| < |f_n(z)| $, which finishes the first part of the proof. 

To prove the estimate in (3) is sharp, let $\delta=1$ and $a(n) = n$, choose any holomorphic map $h: \mathbb{D} \to \mathbb{C}$ and define $f_n(z) = e^{b(n) h(z)}$ and $z^n g_n(z) := f_n(z) (e^{z^n b(n)}-1)$ so that items (1) and (2) hold. Now suppose for any $\kappa > 0$ there is an $n\geq 1$ such that $\kappa \log b(n) > n$, i.e. $b(n) > (e^{1/\kappa})^n$. We have
\[
\frac{z^n g_n(z)}{f_n(z)} = e^{z^n b(n)}-1,
\]
which does not converge to $0$ on any disc $B_r(0)$. In fact, by the assumption on $b(n)$, we see that for any $r>0$ there exist infinitely many $n \geq 1$ such that $b(n)> (1/r)^n$. It follows for $z=r$ that $z^n b(n) > 1$, from which we see $e^{r^n b(n)}-1 > e-1>1$. Hence we do not have $|z^{a(n)} g_n(z)| < |f_n(z)|$ for all $z$ small enough and $n$ large enough.
\end{proof}

\begin{restatable}{thm}{zerofreecontourpartitionfunction}\label{thm:main zero-freeness result contour version}
 Let $C>0$. There exists $\delta = \delta(d,C) >0$ such that for all $z \in \mathbb{C}$ with $|z|< \delta$ and for all tori $\torus \in \Balancedtori_d(C)$ we have
\[
Z_{\text{match}}(\torus;z) \neq 0.
\]   
\end{restatable}
\begin{proof}
We claim there exists an $N>0$ such that for any $z \in \mathbb{C}$ with $|z|<\frac{\delta_2(d,C)}{e^{C}}$ and any torus $\torus \in \Balancedtori_d(C)$ with $|\torus| \geq N$ we have $Z_\text{match}(\torus;z) \neq 0$.
Since there are finitely many tori $\torus \in \Balancedtori_d(C)$ with $|\torus|<N$, by choosing $M$ larger if necessary the theorem follows.

To prove the claim, note $Z_\text{match}(\torus;z) = 2Z^{\text{even}}_\text{match}(\torus;z)+Z^{\text{large}}_\text{match}(\torus;z)$, by Lemma  \ref{lem:contourpartitionfunction_usingtoruscomptaibility} and Lemma \ref{lem: symmetry even odd tori}.
Given $d \in \mathbb{N}_{\geq 2}$ and $C>0$ the set $\Balancedtori_d(C)$ is countable and we can choose a bijection $h:\mathbb{N} \to \Balancedtori_d(C)$ such that $n > m$ implies  $|h(n)| \geq |h(m)|$.

We define maps $\tilde{a},\tilde{b}:\Balancedtori_d(C) \to \mathbb{N}$ as follows.
For $\torus \in \Balancedtori_d(C)$ with shortest side length $\ell_1$ we define $\tilde{a}(\torus) = \lfloor \rho \ell_1 \rfloor$, where $\rho$ is the constant from Lemma \ref{lem:Peierls condition}.
Furthermore we define $\tilde{b}(\torus) = |\torus|$. Define maps $a,b:\mathbb{N} \to \mathbb{N}$ as $a = \tilde{a} \circ h$ and $b = \tilde{b} \circ h$.

For $n \in \mathbb{N}$ the function $Z^{\text{large}}_\text{match}(h(n);z)/(z^{a(n)})$ is a polynomial in $z$ which we denote by $g_{n}(z)$. Write $f_{n}(z) = Z^{\text{even}}_\text{match}(h(n);z)$, thus $Z(h(n);z) = 2 f_{n}(z)+ z^{a(n)} g_{n}(z)$. We check the conditions of Lemma \ref{lem: useful Han's lemma} for functions $f_{n},g_{n}, a(n)$ and $b(n)$ as above with $\delta=\delta_2(d,C)$. Assumption (1) of Lemma \ref{lem: useful Han's lemma} holds by Theorems \ref{thm:stable contours are all contours}, \ref{thm:logtrunc converges_tori} and \ref{thm:logtrunc converges for even plus big}. Assumption (2) of Lemma \ref{lem: useful Han's lemma} holds by choice of the bijection $h$ and Lemma \ref{lem:limit free energy same with or without large contours}.  Assumption (3) of Lemma \ref{lem: useful Han's lemma} also holds with $\kappa = 1/C$ by definition of $\Balancedtori_d(C)$ and the functions $a$ and $b$. It follows from Lemma \ref{lem: useful Han's lemma} there is a constant $M>0$ such that for $|z|< \frac{\delta}{e^{C}}$ and $n \geq M$ we have $|z^{a(n)} g_{n}(z)| < |f_{n}(z)| < 2|f_{n}(z)|$.
Hence $|Z(h(n);z)| = |2 f_{n}(z)+ z^{a(n)} g_{n}(z)| \geq |2 |f_{n}(z)| - |z^{a(n)} g_{n}(z)| | > 0$. As for $n>m$ we have $|h(n)| \geq |h(m)|$, it follows for $N=|h(M)|$, any $z \in \mathbb{C}$ with $|z|<\frac{\delta_2(d,C)}{e^{C}}$ and any torus $\torus \in \Balancedtori_d(C)$ with $|\torus| \geq N$ we have $Z_\text{match}(\torus;z) \neq 0$. This proves the claim, completing the proof of the theorem.
\end{proof}

\begin{remark}
Let $C>0$. For $|z| < \frac{\delta_2(d,C)}{e^{C}}$ the limit exists 
\[
\lim_{\substack{|\torus| \to \infty\\ \torus \in \Balancedtori_d(C)}} \frac{\log Z_\text{match}(\torus;z)}{|\torus|}
\]
and converges to the function $f_C(z)$ defined in Lemma \ref{lem:limit free energy same with or without large contours}. 
For any two constants $C_1>C_2>0$ and any $z \in \mathbb{C}$ with $|z| < \min(\frac{\delta_2(d,C_1)}{e^{C_1}}, \frac{\delta_2(d,C_2)}{e^{C_2}})$ we have $f_{C_1}(z) = f_{C_2}(z)$, hence the function $f$ does not depend on $C$.
This justifies referring to $f$ as the limit free energy of balanced tori around infinity.
\end{remark}

From Theorem \ref{thm:main zero-freeness result contour version}, we immediately obtain the first part of the main theorem. 

\begin{theorem*}[First part of Main Theorem]\label{thm:main_balanced tori}
Let $\mathcal{F}$ be a family of even $d$-dimensional tori. If $\mathcal{F}$ is balanced, then the zeros of the independence polynomials $\{Z_{\mathcal{T}}:\mathcal{T} \in \mathcal{F}\}$ are uniformly bounded.
\end{theorem*}
\begin{proof}
   The family $\mathcal{F}$ is balanced if and only if there is a $C>0$ such that $\mathcal{F} \subset \Balancedtori_d(C)$.
    By Corollary \ref{cor:Zmatch = Zind} and Theorem \ref{thm:main zero-freeness result contour version}, we see there exists a uniform $\Lambda(d,C) = 1/\delta(d,C)$ such that for any $\lambda \in \mathbb{C}$ with $|\lambda| > \Lambda(d,C)$ and any $\torus \in \Balancedtori_d(C)$  we have $Z_{\text{ind}} (\mathcal{T}; \lambda) \neq 0$.
\end{proof}


\section{Unbounded zeros of highly unbalanced tori}\label{sec:unboundedzeros}

In this section we will prove that the independence polynomials of highly unbalanced tori have unbounded zeros. First we will consider tori for which all dimensions except one are constant. The fact that zeros are unbounded when the last dimension diverges will immediately imply that for sufficiently unbalanced sequences of tori the zeros are unbounded. A more careful analysis then provides explicit bounds on the required relative dimensions of the tori. The proofs in this section rely on an analysis of the corresponding transfer-matrices.

For positive integers $n$ we will let $C_n$ denote the cycle graph on $n$ vertices. We let $G_1\square G_2$ denote the cartesian product of two graphs $G_1,G_2$, i.e. the graph with vertex set $V(G_1) \times V(G_2)$ and $(v_1,u_1) \sim (v_2, u_2)$ iff either $v_1 = v_2$ and $u_1 \sim u_2$ in $G_2$ or $u_1 = u_2$ and $v_1 \sim v_2$ in $G_1$. What was previously denoted by $\mathbb{Z}_{n_1} \times \cdots \times \mathbb{Z}_{n_d}$ shall in this section be denoted by $C_{n_1} \square \cdots \square C_{n_d}$.

There will be no other partition function than the independence polynomial which, for a graph $G$ and parameter $\lambda$, we denote by $Z(G;\lambda)$.

\subsection{Transfer-matrix method}
\label{sec: transfer-matrix method}
Fix $G$ to be a finite graph and let $\mathcal{I}$ denote the set of its independent sets. Two independent sets $S,T \in \mathcal{I}$ are said to be \emph{compatible} if $S \cap T = \emptyset$ and we write $S \sim T$. We let $A$ denote the adjacency matrix of the compatibility graph, i.e. the rows and columns of $A$ are indexed by elements of $\mathcal{I}$ and $A_{S,T} = 1$ if $S \sim T$ and $A_{S,T} = 0$ otherwise. Furthermore, for a variable $\lambda$, we let $D_\lambda$ denote the diagonal matrix with $(D_\lambda)_{S,S} = \lambda^{|S|}$. 

\begin{theorem}[Transfer-matrix method]
    For any $n \in \mathbb{Z}_{\geq 1}$ 
    \[
        Z(C_n \square G;\lambda) = \Tr\left[(D_\lambda A)^n\right].
    \]
\end{theorem}
\begin{proof}
Let $\mathcal{P} \subseteq \mathcal{I}^n$ denote those tuples $(S_1, \dots, S_n)$ for which $S_i \sim S_{i+1}$ for all $i = 1, \dots, n$, reducing the index modulo $n$. The independent sets of $C_n \square G$ correspond one to one with the elements of $\mathcal{P}$. We therefore find that 
\[
    \Tr\left[(D_\lambda A)^n\right] = \sum_{(S_1, \dots, S_n) \in \mathcal{I}^n} \prod_{i=1}^n (D_\lambda A)_{S_i S_{i+1}} = \sum_{(S_1, \dots, S_n) \in \mathcal{P}}\lambda^{\sum_{i=1}^n |S_i|}
    = Z(C_n \square G;\lambda).
\]
\end{proof}

Throughout this section we will frequently use that for any complex valued square matrix $M$ and integer $n \geq 1$
\[
    \Tr(M^n) = \hspace{-0.75 cm} \sum_{s \text{ eigenvalue of } M} \hspace{-0.75 cm}s^n. 
\]
This observation reveals the strength of the transfer-matrix method. It shows that $Z(C_n \square G;\lambda)$ can be written as a simple expression in $n$ and a fixed set of values. This motivates the study of the eigenvalues of the transfer-matrix. 

\begin{lemma}
\label{lem: real_large_eigenvalue}
Let $\lambda \in \mathbb{R}_{\geq 0}$. The eigenvalues of $D_{\lambda} A$ are real and there is a simple positive eigenvalue $r$ such that $r > |s|$ for all other eigenvalues $s$.
\end{lemma}

\begin{proof}
    We first consider $\lambda = 0$. The only non-zero entry of the diagonal matrix $D_0$ is $(D_0)_{\emptyset,\emptyset}$. Therefore the matrix $D_0A$ has rank at most $1$ and thus the eigenvalue $0$ appears with multiplicity at least $|\mathcal{I}|-1$. Observe that $D_0 A e_{\emptyset} = e_{\emptyset}$ and thus $1$ is an eigenvalue, which must necessarily be simple.

    Now assume $\lambda > 0$. The matrix $D_\lambda A$ is conjugate to the real symmetric matrix $D_{\lambda^{1/2}} A D_{\lambda^{1/2}}$ and thus all its eigenvalues are real. The entries of $D_{\lambda} A$ are all non-negative and its support matrix is $A$. The matrix $A$ is the adjacency matrix of a connected graph because $S \sim \emptyset$ for every $S \in \mathcal{I}$. The diagonal entry $(D_\lambda A)_{\emptyset,\emptyset} = 1$ is non-zero. These facts allow us to conclude that $D_\lambda A$ is an aperiodic irreducible matrix. The Perron–Frobenius theorem states that we can conclude that the eigenvalue of maximal norm of $D_{\lambda}A$ is simple and positive real. 
\end{proof}

\begin{corollary}
    \label{cor: no_acc_on_R}
    Let $\lambda_0 \in \mathbb{R}_{\geq 0}$. The zeros of the polynomials $\{Z(C_n \square G;\lambda)\}_{n \geq 1}$ do not accumulate on $\lambda_0$.
\end{corollary}

\begin{proof}
    According to Lemma~\ref{lem: real_large_eigenvalue} the matrix $D_{\lambda_0}A$ has a unique eigenvalue of maximal norm, which we denote by $r(\lambda_0)$. Because $r(\lambda_0)$ is simple there exists a neighborhood $U \subseteq \mathbb{C}$ of $\lambda_0$ such that $r: U \to \mathbb{C}$ is the analytic continuation of this eigenvalue, i.e. $r$ holomorphic and $r(\lambda)$ is an eigenvalue of $D_\lambda A$ for all $\lambda \in U$. 

    Because the set of eigenvalues of $D_{\lambda} A$ moves continuously with $\lambda$ there is a radius $R > 0$ and a constant $\zeta < 1$ such that $\zeta \cdot |r(\lambda)| > |s|$ for all other eigenvalues $s$ of $D_{\lambda} A$ for all $\lambda$ with $|\lambda-\lambda_0| \leq R$. For these $\lambda$ we have 
    \[
        \left|\frac{Z(C_n \square G;\lambda)}{r(\lambda)^n} - 1 \right| = \left|\frac{Z(C_n \square G;\lambda) - r(\lambda)^n}{r(\lambda)^n} \right| = \sum_{s \neq r(\lambda)} \left(\frac{s}{r(\lambda)}\right)^n < \zeta^n \cdot \left(|\mathcal{I}| - 1\right),
    \]
    where the sum runs over eigenvalues of $D_\lambda A$ not equal to $r(\lambda)$.
    For $n$ sufficiently large the quantity on the right-hand side is strictly less than $1$, which implies that $Z(C_n \square G;\lambda)$ cannot be zero. The disk of radius $R$ around $\lambda_0$ can therefore only contain finitely many zeros.
\end{proof}

We can deduce that the sequence $\{C_n \square G\}_{n \geq 1}$ undergoes no phase-transition. Indeed, the free energy per site converges:
\[
    \lim_{n \to \infty} \frac{\log(Z(C_n \square G;\lambda))}{n |V(G)|} = \frac{\log(r(\lambda))}{|V(G)|},
\]
where $r(\lambda)$ is the largest eigenvalue of $D_\lambda A$. This is an analytic function of $\lambda$ on $[0,\infty)$.

\subsection{Constant width tori}
\label{sec: cons_width_tori}
We now move from general graphs to tori. Let $\mathcal{T}$ be a fixed even torus (we allow $\mathcal{T}$ to be an even cycle). We again let $\mathcal{I}$ denote the collection of independent sets of $\mathcal{T}$. We will show that the zeros of the tori $C_n \square \mathcal{T}$ are unbounded or, in other words, accumulate at $\infty$.

Define $\alpha = \frac{1}{2} |V(\mathcal{T})|$.  There are two maximum independent sets, namely
\[
    S_{\text{even}} = \{v \in \mathcal{T}: \text{$v$ is even}\}
    \quad
    \text{ and }
    \quad
    S_{\text{odd}} = \{v \in \mathcal{T}: \text{$v$ is odd}\}.
\]
For any $S \in \mathcal{I}$ define 
\[
    \|S\| = \alpha - |S|.
\]
Although related, this definition should not be confused with the surface energy of a contour $\|\gamma\|$. We observe that $\|\Seven\| = \|\Sodd\| = 0$ and $\|S\| > 0$ for all other $S \in \mathcal{I}$. 

We write $z = 1/\lambda$. Define the diagonal matrix $\hat{D}_z$ by $(\hat{D}_z)_{S,S} = z^{\|S\|}$ and recall that $A$ denotes the compatibility matrix of the independent sets. We observe that 
\[
    z^\alpha D_{1/z} = \hat{D}_z 
    \quad
    \text{ and }
    \quad 
    \Tr\left[(\hat{D}_z A)^n\right] = z^{n\alpha} \cdot Z(C_n \square \mathcal{T}; 1/z).
\]

From now on we let $M_z = \hat{D}_z A$. For any $S\in\mathcal{I}$ we let $e_S$ denote the $|\mathcal{I}|$-dimensional unit vector belonging to index $S$. We turn our attention to the eigenvalues of $M_z$ in a neighbourhood of $z=0$.

\begin{lemma}
    \label{lem: q definition}
    There is a neighbourhood $U$ of $0$ and holomorphic functions $q^+$ and $q^{-}$ defined on $U$ such that
    \begin{itemize}
        \item 
        $q^{+}(z)$ and $q^{-}(z)$ are eigenvalues of $M_z$ for all $z \in U$ and
        \item 
        $q^{+}(0) = 1$ and $q^{-}(0)=-1$ are the only non-zero eigenvalues of $M_0$. 
    \end{itemize}
\end{lemma}

\begin{proof}
    We can write 
    \[
        M_0 = e_{\Seven} \sum_{\substack{S \in \mathcal{I}\\ S \sim \Seven}} e_{S}^T + e_{\Sodd} \sum_{\substack{S \in \mathcal{I}\\ S \sim \Sodd}} e_{S}^T.
    \]
    We see that $M_0$ has rank two, $M(e_{\Seven} + e_{\Sodd}) = e_{\Seven} + e_{\Sodd}$ and $M(e_{\Seven} - e_{\Sodd}) = -(e_{\Seven} - e_{\Sodd})$. Therefore $q^+(0) =1$ and $q^-(0)=-1$ are the only two non-zero eigenvalues of $M_0$ and they are both simple. By the implicit function theorem these can be analytically extended to eigenvalues of $M_z$ on a neighborhood of $z=0$. 
\end{proof}

We will keep referring to $q^+$ and $q^-$ as they are defined in Lemma~\ref{lem: q definition}. We can now give a reasonably short proof that the zeros of $C_n \square \mathcal{T}$ accumulate at $\infty$ using Montel's theorem as a black box.
\begin{lemma}
    \label{lem: zer_acc_infty_basic}
    Let $R>0$. There are only finitely many $n$ such that all zeros of $Z(C_n \square \mathcal{T};\lambda)$ are less than $R$ in norm.
\end{lemma}

\begin{proof}
    Let $U$ be a connected neighborhood of $z=0$ such that there is a $\zeta < 1$ for which $|s| < \zeta \cdot \min\{|q^+(z)|,|q^-(z)|\}$ for all other eigenvalues $s$ of $M_z$ for every $z \in U$. We can assume that $q^+$ and $q^-$ are defined on $U$ and that $U$ is contained in a ball of radius $1/R$. Let $N_0$ be such that $\zeta^{N_0}(|\mathcal{I}|-2) \leq 1/2$.  Let $I \subseteq \mathbb{Z}_{\geq N_0}$ be the set of indices such that for $n \in I$ the polynomial $z^{n\alpha} \cdot Z(C_n \square \mathcal{T}; 1/z)$ has no zeros in $U\setminus \{0\}$. We will show that the family of functions
    \[
      \mathcal{F} = \left\{\frac{z^{n\alpha} \cdot Z(C_n \square \mathcal{T}; 1/z)}{q^+(z)^n}\right\}_{n \in I}
    \]
    is a normal family on $U\setminus\{0\}$. We will do this by applying the strong version of Montel's theorem, i.e. we show that $\mathcal{F}$ avoids three values in the Riemann-sphere. 

    Because $z^{n\alpha} \cdot Z(C_n \square \mathcal{T}; 1/z)$ is a polynomial $f(z) \neq \infty$ for every $f\in \mathcal{F}$ and $z \in U$. By definition of $I$ we see that $f(z) \neq 0$ for every $f\in \mathcal{F}$ and nonzero $z \in U$. We also claim that $\mathcal{F}$ avoids $1$. To prove it we assume that there is a $z \in U$ and an index $n \in I$ that show otherwise. Then 
    \begin{align*}
        0 = \left|\frac{z^{n\alpha} \cdot Z(C_n \square \mathcal{T}; 1/z) - q^+(z)^n}{q^-(z)^n}\right| = \left|\frac{\Tr\left[(\hat{D}_z A)^n\right] - q^+(z)^n}{q^-(z)^n}\right|
        &=\left|1 + \sum_{s \neq q^{\pm}(z)} \left(\frac{s}{q^-(z)}\right)^n \right|\\ &\geq 1 - \zeta^n(|\mathcal{I}|-2) \geq 1/2,
    \end{align*}
    where the sum runs over the eigenvalues of $M_z$ not equal to $q^{\pm}(z)$. This is a contradiction and we can thus conclude that $\mathcal{F}$ is a normal family. We will now show that this implies that $\mathcal{F}$ is finite.

    Define $\beta(z) = q^+(z)/q^-(z)$. We observe that $\beta(0)= -1$ and, by Lemma~\ref{lem: real_large_eigenvalue}, $|\beta(z)| > 1$ for $z > 0$. The map $\beta$ is holomorphic and non-constant and thus an open map. Let $U^+ = \{z \in U: |\beta(z)| >1\}$ and $U^- = \{z \in U: |\beta(z)| <1\}$. These are both open non-empty subsets of $U\setminus\{0\}$. 
    For $z \in U^+$ we have that 
    \[
        \lim_{n \to \infty} \frac{z^{n\alpha} \cdot Z(C_n \square \mathcal{T}; 1/z)}{q^+(z)^n} = \lim_{n \to \infty} \left[\beta(z)^n + 1 + \sum_{s \neq q^{\pm}(z)} \left(\frac{s}{q^+(z)}\right)^n\right] = \infty,  
    \]
    while for $z \in U^-$ this limit is equal to $1$. If $\mathcal{F}$ were to have a sequence of elements whose indices converge to $\infty$, it should have a subsequence that converges to a holomorphic function that is constant $\infty$ on $U^+$ and constant $1$ on $U^-$. Because $U\setminus\{0\}$ is connected, such a function does not exist. 
    
    This shows that the index-set $I$ is finite. It follows that there is an $N_1$ such that for all $n\geq N_1$ the polynomial $z^{n\alpha} \cdot Z(C_n \square \mathcal{T}; 1/z)$ has a zero $z_0 \neq 0$ in $U$. Therefore $\lambda_0 = 1/z_0$ is a zero of $Z(C_n \square \mathcal{T}; \lambda)$ with $|\lambda_0| > R$.    
\end{proof}

\begin{remark}
    The proof of Corollary~\ref{cor: no_acc_on_R} works just as well to show that zeros of $Z(C_n \square G; \lambda)$ cannot accumulate on any $\lambda_0$ for which $D_{\lambda_0}A$ has a unique largest (in norm) eigenvalue. Similarly, the proof of Lemma~\ref{lem: zer_acc_infty_basic} works to show that zeros accumulate on any parameter $\lambda_0$ for which $D_{\lambda_0}A$ has two or more simple eigenvalues $\{r_1(\lambda_0), \dots, r_k(\lambda_0)\}$ of the same norm that are larger than all the eigenvalues if no pair of such eigenvalues persistently has the same norm. That is, if there is no distinct pair $i,j$ and neighborhood $U$ of $\lambda_0$ for which the analytic continuations $r_i,r_j$ satisfy $|r_i(\lambda)| = |r_{j}(\lambda)|$ for all $\lambda \in U$.

    This shows that, in the case that there are no eigenvalues that persistently have the same norm, the accumulation points of the zeros of $Z(C_n \square G; \lambda)$ are exactly those parameters $\lambda_0$ for which $D_{\lambda_0}A$ has two or more maximal eigenvalues of the same norm; a special case of \cite[Theorem 1.5]{Sokaldense}. It then follows that the set of accumulation points is a union of real algebraic curves; see Figure~\ref{fig: acc_par} for two examples.
\end{remark}

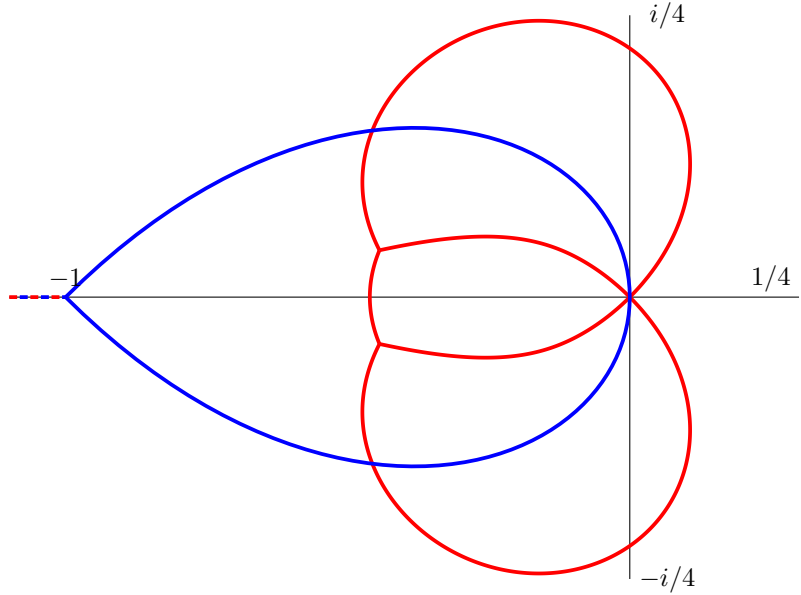
\begin{figure}
\center
\begin{tikzpicture}
\draw[] (-8.25,0) -- (2.25,0);
\draw[] (0,-3.75) -- (0,3.75);

\draw[c0,line width=0.5mm] (-3.32761,0.62329) -- (-3.33614,0.64) -- (-3.34833,0.66454) -- (-3.36034,0.68948) -- (-3.37215,0.71482) -- (-3.38375,0.74057) -- (-3.39512,0.76672) -- (-3.40625,0.79329) -- (-3.41713,0.82028) -- (-3.42773,0.84768) -- (-3.43804,0.8755) -- (-3.44805,0.90373) -- (-3.45773,0.93239) -- (-3.46708,0.96146) -- (-3.47606,0.99095) -- (-3.48466,1.02085) -- (-3.49287,1.05116) -- (-3.50066,1.08187) -- (-3.50801,1.11298) -- (-3.5149,1.14448) -- (-3.52131,1.17636) -- (-3.52723,1.2086) -- (-3.53263,1.24121) -- (-3.53749,1.27415) -- (-3.54179,1.30743) -- (-3.54552,1.34102) -- (-3.54866,1.37491) -- (-3.55119,1.40907) -- (-3.55309,1.44349) -- (-3.55435,1.47815) -- (-3.55496,1.51302) -- (-3.5549,1.54808) -- (-3.55416,1.58331) -- (-3.55274,1.61868) -- (-3.55063,1.65417) -- (-3.54781,1.68974) -- (-3.5443,1.72538) -- (-3.54007,1.76106) -- (-3.53514,1.79675) -- (-3.5295,1.83243) -- (-3.52315,1.86807) -- (-3.5161,1.90364) -- (-3.50836,1.93912) -- (-3.49993,1.97448) -- (-3.49081,2.00971) -- (-3.48102,2.04478) -- (-3.47056,2.07966) -- (-3.45946,2.11434) -- (-3.44771,2.1488) -- (-3.43533,2.18302) -- (-3.42233,2.21698) -- (-3.40874,2.25067) -- (-3.39455,2.28406) -- (-3.37979,2.31715) -- (-3.36447,2.34993) -- (-3.34861,2.38237) -- (-3.33222,2.41447) -- (-3.31531,2.44622) -- (-3.29791,2.4776) -- (-3.28002,2.50862) -- (-3.26166,2.53926) -- (-3.24284,2.56951) -- (-3.22359,2.59936) -- (-3.2039,2.62883) -- (-3.18381,2.65788) -- (-3.16332,2.68653) -- (-3.14244,2.71477) -- (-3.1212,2.74259) -- (-3.09959,2.77) -- (-3.07764,2.79698) -- (-3.05535,2.82355) -- (-3.03275,2.84968) -- (-3.00983,2.87539) -- (-2.98662,2.90068) -- (-2.96312,2.92553) -- (-2.93934,2.94996) -- (-2.9153,2.97396) -- (-2.891,2.99753) -- (-2.86646,3.02067) -- (-2.84168,3.04339) -- (-2.81668,3.06567) -- (-2.79146,3.08753) -- (-2.76603,3.10896) -- (-2.7404,3.12997) -- (-2.71459,3.15055) -- (-2.68859,3.17071) -- (-2.66242,3.19044) -- (-2.63608,3.20976) -- (-2.60958,3.22865) -- (-2.58293,3.24713) -- (-2.55614,3.26519) -- (-2.52921,3.28283) -- (-2.50215,3.30007) -- (-2.47497,3.31689) -- (-2.44767,3.3333) -- (-2.42026,3.3493) -- (-2.39275,3.36489) -- (-2.36514,3.38008) -- (-2.33743,3.39487) -- (-2.30964,3.40925) -- (-2.28177,3.42324) -- (-2.25383,3.43683) -- (-2.22582,3.45002) -- (-2.19774,3.46282) -- (-2.1696,3.47522) -- (-2.14141,3.48724) -- (-2.11317,3.49886) -- (-2.08489,3.5101) -- (-2.05657,3.52096) -- (-2.02821,3.53143) -- (-1.99983,3.54152) -- (-1.97142,3.55123) -- (-1.94299,3.56056) -- (-1.91454,3.56952) -- (-1.88608,3.57811) -- (-1.85762,3.58632) -- (-1.82915,3.59416) -- (-1.80068,3.60164) -- (-1.77222,3.60875) -- (-1.74376,3.61549) -- (-1.71532,3.62187) -- (-1.68689,3.62789) -- (-1.65849,3.63356) -- (-1.63011,3.63886) -- (-1.60176,3.64381) -- (-1.57344,3.64841) -- (-1.54515,3.65265) -- (-1.5169,3.65655) -- (-1.4887,3.66009) -- (-1.46054,3.66329) -- (-1.43243,3.66615) -- (-1.40437,3.66866) -- (-1.37637,3.67083) -- (-1.34843,3.67266) -- (-1.32055,3.67416) -- (-1.29273,3.67532) -- (-1.26499,3.67614) -- (-1.23731,3.67663) -- (-1.20972,3.6768) -- (-1.1822,3.67663) -- (-1.15476,3.67614) -- (-1.1274,3.67532) -- (-1.10013,3.67417) -- (-1.07296,3.67271) -- (-1.04587,3.67092) -- (-1.01888,3.66882) -- (-0.99199,3.6664) -- (-0.96521,3.66366) -- (-0.93852,3.66061) -- (-0.91195,3.65725) -- (-0.88548,3.65358) -- (-0.85913,3.6496) -- (-0.83289,3.64532) -- (-0.80677,3.64073) -- (-0.78077,3.63583) -- (-0.7549,3.63064) -- (-0.72915,3.62514) -- (-0.70353,3.61934) -- (-0.67804,3.61325) -- (-0.65269,3.60687) -- (-0.62747,3.60019) -- (-0.60239,3.59322) -- (-0.57746,3.58596) -- (-0.55266,3.57841) -- (-0.52802,3.57057) -- (-0.50352,3.56245) -- (-0.47918,3.55404) -- (-0.45499,3.54536) -- (-0.43095,3.53639) -- (-0.40708,3.52714) -- (-0.38336,3.51762) -- (-0.35981,3.50782) -- (-0.33643,3.49775) -- (-0.31321,3.48741) -- (-0.29017,3.47679) -- (-0.26729,3.46591) -- (-0.2446,3.45476) -- (-0.22208,3.44334) -- (-0.19974,3.43166) -- (-0.17758,3.41972) -- (-0.15561,3.40751) -- (-0.13383,3.39505) -- (-0.11223,3.38233) -- (-0.09083,3.36935) -- (-0.06961,3.35612) -- (-0.0486,3.34263) -- (-0.02778,3.3289) -- (-0.00717,3.31491) -- (0.01325,3.30067) -- (0.03346,3.28619) -- (0.05346,3.27147) -- (0.07325,3.25649) -- (0.09284,3.24128) -- (0.1122,3.22583) -- (0.13136,3.21014) -- (0.15029,3.1942) -- (0.16901,3.17804) -- (0.1875,3.16164) -- (0.20577,3.145) -- (0.22381,3.12814) -- (0.24162,3.11104) -- (0.2592,3.09371) -- (0.27655,3.07616) -- (0.29366,3.05838) -- (0.31054,3.04038) -- (0.32717,3.02215) -- (0.34357,3.0037) -- (0.35971,2.98504) -- (0.37562,2.96615) -- (0.39127,2.94704) -- (0.40668,2.92772) -- (0.42183,2.90819) -- (0.43672,2.88844) -- (0.45136,2.86848) -- (0.46574,2.84831) -- (0.47985,2.82792) -- (0.49371,2.80733) -- (0.50729,2.78654) -- (0.52061,2.76553) -- (0.53365,2.74432) -- (0.54643,2.72291) -- (0.55892,2.7013) -- (0.57114,2.67948) -- (0.58308,2.65746) -- (0.59473,2.63525) -- (0.6061,2.61283) -- (0.61718,2.59022) -- (0.62796,2.56742) -- (0.63846,2.54441) -- (0.64866,2.52122) -- (0.65856,2.49782) -- (0.66816,2.47424) -- (0.67745,2.45046) -- (0.68644,2.4265) -- (0.69512,2.40234) -- (0.70349,2.37799) -- (0.71154,2.35345) -- (0.71927,2.32872) -- (0.72668,2.30381) -- (0.73377,2.2787) -- (0.74052,2.25341) -- (0.74695,2.22793) -- (0.75304,2.20226) -- (0.7588,2.1764) -- (0.76421,2.15035) -- (0.76927,2.12412) -- (0.77399,2.0977) -- (0.77836,2.07109) -- (0.78236,2.04429) -- (0.78601,2.0173) -- (0.78929,1.99012) -- (0.79219,1.96275) -- (0.79473,1.93519) -- (0.79688,1.90743) -- (0.79865,1.87947) -- (0.80003,1.85132) -- (0.80101,1.82298) -- (0.80159,1.79442) -- (0.80176,1.76567) -- (0.80151,1.73671) -- (0.80085,1.70754) -- (0.79975,1.67816) -- (0.79823,1.64857) -- (0.79625,1.61875) -- (0.79383,1.58871) -- (0.79094,1.55844) -- (0.78758,1.52793) -- (0.78374,1.49719) -- (0.7794,1.46619) -- (0.77457,1.43494) -- (0.76921,1.40343) -- (0.76333,1.37164) -- (0.7569,1.33958) -- (0.7499,1.30721) -- (0.74233,1.27455) -- (0.73416,1.24155) -- (0.72537,1.20823) -- (0.71594,1.17455) -- (0.70584,1.14049) -- (0.69505,1.10604) -- (0.68353,1.07116) -- (0.67124,1.03583) -- (0.65815,1.00001) -- (0.64421,0.96367) -- (0.62936,0.92675) -- (0.61355,0.88921) -- (0.59671,0.85099) -- (0.57875,0.812) -- (0.55957,0.77216) -- (0.53905,0.73136) -- (0.51705,0.68946) -- (0.49336,0.64628) -- (0.46776,0.60161) -- (0.43993,0.55514) -- (0.40944,0.50646) -- (0.37565,0.45497) -- (0.33761,0.39975) -- (0.29372,0.33923) -- (0.24086,0.27037) -- (0.171,0.18523) -- (0.171,0.18523) -- (0.16229,0.17504) -- (0.15307,0.16435) -- (0.14323,0.15306) -- (0.13266,0.14103) -- (0.12114,0.12808) -- (0.10839,0.11391) -- (0.0939,0.09801) -- (0.0767,0.07941) -- (0.05425,0.0556) -- (0.,0.) -- (-0.05425,-0.05298) -- (-0.07671,-0.07418) -- (-0.09392,-0.09016) -- (-0.10842,-0.10344) -- (-0.12118,-0.11499) -- (-0.13272,-0.12533) -- (-0.14331,-0.13473) -- (-0.15317,-0.1434) -- (-0.16242,-0.15148) -- (-0.17117,-0.15905) -- (-0.17117,-0.15905) -- (-0.24152,-0.21801) -- (-0.29521,-0.2607) -- (-0.34025,-0.29505) -- (-0.37977,-0.32411) -- (-0.41538,-0.34945) -- (-0.44802,-0.37199) -- (-0.47833,-0.39234) -- (-0.50673,-0.4109) -- (-0.53355,-0.42798) -- (-0.55903,-0.44381) -- (-0.58334,-0.45856) -- (-0.60664,-0.47237) -- (-0.62906,-0.48535) -- (-0.65069,-0.4976) -- (-0.67161,-0.50919) -- (-0.6919,-0.52019) -- (-0.71161,-0.53065) -- (-0.73081,-0.54062) -- (-0.74952,-0.55013) -- (-0.76781,-0.55924) -- (-0.78569,-0.56796) -- (-0.80321,-0.57632) -- (-0.82038,-0.58435) -- (-0.83725,-0.59208) -- (-0.85382,-0.59952) -- (-0.87012,-0.60668) -- (-0.88617,-0.61359) -- (-0.90199,-0.62026) -- (-0.91759,-0.62671) -- (-0.93298,-0.63293) -- (-0.94819,-0.63896) -- (-0.96322,-0.64479) -- (-0.97808,-0.65043) -- (-0.99279,-0.6559) -- (-1.00735,-0.6612) -- (-1.02178,-0.66634) -- (-1.03608,-0.67133) -- (-1.05026,-0.67617) -- (-1.06433,-0.68086) -- (-1.0783,-0.68542) -- (-1.09217,-0.68985) -- (-1.10595,-0.69416) -- (-1.11965,-0.69834) -- (-1.13327,-0.7024) -- (-1.14682,-0.70635) -- (-1.1603,-0.71019) -- (-1.17372,-0.71393) -- (-1.18709,-0.71756) -- (-1.2004,-0.72109) -- (-1.21366,-0.72452) -- (-1.22688,-0.72786) -- (-1.24006,-0.7311) -- (-1.2532,-0.73426) -- (-1.26632,-0.73733) -- (-1.2794,-0.74031) -- (-1.29246,-0.74321) -- (-1.3055,-0.74603) -- (-1.31852,-0.74877) -- (-1.33153,-0.75143) -- (-1.34453,-0.75402) -- (-1.35751,-0.75653) -- (-1.37049,-0.75897) -- (-1.38347,-0.76133) -- (-1.39645,-0.76363) -- (-1.40943,-0.76585) -- (-1.42242,-0.76801) -- (-1.43541,-0.7701) -- (-1.44842,-0.77212) -- (-1.46143,-0.77408) -- (-1.47447,-0.77597) -- (-1.48752,-0.7778) -- (-1.50059,-0.77957) -- (-1.51368,-0.78127) -- (-1.5268,-0.78292) -- (-1.53995,-0.7845) -- (-1.55312,-0.78602) -- (-1.56632,-0.78748) -- (-1.57956,-0.78889) -- (-1.59284,-0.79023) -- (-1.60615,-0.79152) -- (-1.6195,-0.79275) -- (-1.63289,-0.79392) -- (-1.64633,-0.79503) -- (-1.65982,-0.79609) -- (-1.67335,-0.79709) -- (-1.68693,-0.79804) -- (-1.70056,-0.79893) -- (-1.71425,-0.79976) -- (-1.72799,-0.80053) -- (-1.74179,-0.80126) -- (-1.75565,-0.80192) -- (-1.76957,-0.80253) -- (-1.78356,-0.80308) -- (-1.79761,-0.80358) -- (-1.81173,-0.80403) -- (-1.82592,-0.80441) -- (-1.84018,-0.80474) -- (-1.85452,-0.80502) -- (-1.86893,-0.80524) -- (-1.88342,-0.8054) -- (-1.89799,-0.80551) -- (-1.91264,-0.80556) -- (-1.92737,-0.80555) -- (-1.94219,-0.80549) -- (-1.9571,-0.80537) -- (-1.9721,-0.80519) -- (-1.98719,-0.80495) -- (-2.00238,-0.80465) -- (-2.01766,-0.8043) -- (-2.03304,-0.80389) -- (-2.04852,-0.80341) -- (-2.06411,-0.80288) -- (-2.0798,-0.80228) -- (-2.0956,-0.80163) -- (-2.1115,-0.80091) -- (-2.12752,-0.80013) -- (-2.14366,-0.79929) -- (-2.15991,-0.79839) -- (-2.17628,-0.79742) -- (-2.19278,-0.79639) -- (-2.2094,-0.79529) -- (-2.22614,-0.79412) -- (-2.24302,-0.79289) -- (-2.26002,-0.79159) -- (-2.27717,-0.79023) -- (-2.29445,-0.78879) -- (-2.31187,-0.78729) -- (-2.32943,-0.78571) -- (-2.34714,-0.78407) -- (-2.365,-0.78235) -- (-2.38301,-0.78055) -- (-2.40117,-0.77869) -- (-2.4195,-0.77675) -- (-2.43798,-0.77473) -- (-2.45663,-0.77263) -- (-2.47544,-0.77046) -- (-2.49443,-0.76821) -- (-2.51359,-0.76588) -- (-2.53293,-0.76346) -- (-2.55245,-0.76097) -- (-2.57215,-0.75839) -- (-2.59205,-0.75572) -- (-2.61213,-0.75297) -- (-2.63241,-0.75013) -- (-2.65289,-0.7472) -- (-2.67358,-0.74419) -- (-2.69447,-0.74108) -- (-2.71558,-0.73788) -- (-2.7369,-0.73458) -- (-2.75844,-0.73119) -- (-2.78021,-0.72771) -- (-2.80221,-0.72412) -- (-2.82444,-0.72044) -- (-2.84692,-0.71665) -- (-2.86964,-0.71276) -- (-2.8926,-0.70877) -- (-2.91583,-0.70467) -- (-2.93931,-0.70047) -- (-2.96306,-0.69615) -- (-2.98708,-0.69173) -- (-3.01138,-0.68719) -- (-3.03597,-0.68254) -- (-3.06084,-0.67778) -- (-3.08601,-0.6729) -- (-3.11148,-0.6679) -- (-3.13726,-0.66277) -- (-3.16335,-0.65753) -- (-3.18977,-0.65217) -- (-3.21652,-0.64668) -- (-3.2436,-0.64106) -- (-3.27104,-0.63531) -- (-3.29882,-0.62944) -- (-3.32696,-0.62343) -- (-3.32761,-0.62329);

\draw[c0,line width=0.5mm] (-3.32761,-0.62329) -- (-3.33614,-0.64) -- (-3.34833,-0.66454) -- (-3.36034,-0.68948) -- (-3.37215,-0.71482) -- (-3.38375,-0.74057) -- (-3.39512,-0.76672) -- (-3.40625,-0.79329) -- (-3.41713,-0.82028) -- (-3.42773,-0.84768) -- (-3.43804,-0.8755) -- (-3.44805,-0.90373) -- (-3.45773,-0.93239) -- (-3.46708,-0.96146) -- (-3.47606,-0.99095) -- (-3.48466,-1.02085) -- (-3.49287,-1.05116) -- (-3.50066,-1.08187) -- (-3.50801,-1.11298) -- (-3.5149,-1.14448) -- (-3.52131,-1.17636) -- (-3.52723,-1.2086) -- (-3.53263,-1.24121) -- (-3.53749,-1.27415) -- (-3.54179,-1.30743) -- (-3.54552,-1.34102) -- (-3.54866,-1.37491) -- (-3.55119,-1.40907) -- (-3.55309,-1.44349) -- (-3.55435,-1.47815) -- (-3.55496,-1.51302) -- (-3.5549,-1.54808) -- (-3.55416,-1.58331) -- (-3.55274,-1.61868) -- (-3.55063,-1.65417) -- (-3.54781,-1.68974) -- (-3.5443,-1.72538) -- (-3.54007,-1.76106) -- (-3.53514,-1.79675) -- (-3.5295,-1.83243) -- (-3.52315,-1.86807) -- (-3.5161,-1.90364) -- (-3.50836,-1.93912) -- (-3.49993,-1.97448) -- (-3.49081,-2.00971) -- (-3.48102,-2.04478) -- (-3.47056,-2.07966) -- (-3.45946,-2.11434) -- (-3.44771,-2.1488) -- (-3.43533,-2.18302) -- (-3.42233,-2.21698) -- (-3.40874,-2.25067) -- (-3.39455,-2.28406) -- (-3.37979,-2.31715) -- (-3.36447,-2.34993) -- (-3.34861,-2.38237) -- (-3.33222,-2.41447) -- (-3.31531,-2.44622) -- (-3.29791,-2.4776) -- (-3.28002,-2.50862) -- (-3.26166,-2.53926) -- (-3.24284,-2.56951) -- (-3.22359,-2.59936) -- (-3.2039,-2.62883) -- (-3.18381,-2.65788) -- (-3.16332,-2.68653) -- (-3.14244,-2.71477) -- (-3.1212,-2.74259) -- (-3.09959,-2.77) -- (-3.07764,-2.79698) -- (-3.05535,-2.82355) -- (-3.03275,-2.84968) -- (-3.00983,-2.87539) -- (-2.98662,-2.90068) -- (-2.96312,-2.92553) -- (-2.93934,-2.94996) -- (-2.9153,-2.97396) -- (-2.891,-2.99753) -- (-2.86646,-3.02067) -- (-2.84168,-3.04339) -- (-2.81668,-3.06567) -- (-2.79146,-3.08753) -- (-2.76603,-3.10896) -- (-2.7404,-3.12997) -- (-2.71459,-3.15055) -- (-2.68859,-3.17071) -- (-2.66242,-3.19044) -- (-2.63608,-3.20976) -- (-2.60958,-3.22865) -- (-2.58293,-3.24713) -- (-2.55614,-3.26519) -- (-2.52921,-3.28283) -- (-2.50215,-3.30007) -- (-2.47497,-3.31689) -- (-2.44767,-3.3333) -- (-2.42026,-3.3493) -- (-2.39275,-3.36489) -- (-2.36514,-3.38008) -- (-2.33743,-3.39487) -- (-2.30964,-3.40925) -- (-2.28177,-3.42324) -- (-2.25383,-3.43683) -- (-2.22582,-3.45002) -- (-2.19774,-3.46282) -- (-2.1696,-3.47522) -- (-2.14141,-3.48724) -- (-2.11317,-3.49886) -- (-2.08489,-3.5101) -- (-2.05657,-3.52096) -- (-2.02821,-3.53143) -- (-1.99983,-3.54152) -- (-1.97142,-3.55123) -- (-1.94299,-3.56056) -- (-1.91454,-3.56952) -- (-1.88608,-3.57811) -- (-1.85762,-3.58632) -- (-1.82915,-3.59416) -- (-1.80068,-3.60164) -- (-1.77222,-3.60875) -- (-1.74376,-3.61549) -- (-1.71532,-3.62187) -- (-1.68689,-3.62789) -- (-1.65849,-3.63356) -- (-1.63011,-3.63886) -- (-1.60176,-3.64381) -- (-1.57344,-3.64841) -- (-1.54515,-3.65265) -- (-1.5169,-3.65655) -- (-1.4887,-3.66009) -- (-1.46054,-3.66329) -- (-1.43243,-3.66615) -- (-1.40437,-3.66866) -- (-1.37637,-3.67083) -- (-1.34843,-3.67266) -- (-1.32055,-3.67416) -- (-1.29273,-3.67532) -- (-1.26499,-3.67614) -- (-1.23731,-3.67663) -- (-1.20972,-3.6768) -- (-1.1822,-3.67663) -- (-1.15476,-3.67614) -- (-1.1274,-3.67532) -- (-1.10013,-3.67417) -- (-1.07296,-3.67271) -- (-1.04587,-3.67092) -- (-1.01888,-3.66882) -- (-0.99199,-3.6664) -- (-0.96521,-3.66366) -- (-0.93852,-3.66061) -- (-0.91195,-3.65725) -- (-0.88548,-3.65358) -- (-0.85913,-3.6496) -- (-0.83289,-3.64532) -- (-0.80677,-3.64073) -- (-0.78077,-3.63583) -- (-0.7549,-3.63064) -- (-0.72915,-3.62514) -- (-0.70353,-3.61934) -- (-0.67804,-3.61325) -- (-0.65269,-3.60687) -- (-0.62747,-3.60019) -- (-0.60239,-3.59322) -- (-0.57746,-3.58596) -- (-0.55266,-3.57841) -- (-0.52802,-3.57057) -- (-0.50352,-3.56245) -- (-0.47918,-3.55404) -- (-0.45499,-3.54536) -- (-0.43095,-3.53639) -- (-0.40708,-3.52714) -- (-0.38336,-3.51762) -- (-0.35981,-3.50782) -- (-0.33643,-3.49775) -- (-0.31321,-3.48741) -- (-0.29017,-3.47679) -- (-0.26729,-3.46591) -- (-0.2446,-3.45476) -- (-0.22208,-3.44334) -- (-0.19974,-3.43166) -- (-0.17758,-3.41972) -- (-0.15561,-3.40751) -- (-0.13383,-3.39505) -- (-0.11223,-3.38233) -- (-0.09083,-3.36935) -- (-0.06961,-3.35612) -- (-0.0486,-3.34263) -- (-0.02778,-3.3289) -- (-0.00717,-3.31491) -- (0.01325,-3.30067) -- (0.03346,-3.28619) -- (0.05346,-3.27147) -- (0.07325,-3.25649) -- (0.09284,-3.24128) -- (0.1122,-3.22583) -- (0.13136,-3.21014) -- (0.15029,-3.1942) -- (0.16901,-3.17804) -- (0.1875,-3.16164) -- (0.20577,-3.145) -- (0.22381,-3.12814) -- (0.24162,-3.11104) -- (0.2592,-3.09371) -- (0.27655,-3.07616) -- (0.29366,-3.05838) -- (0.31054,-3.04038) -- (0.32717,-3.02215) -- (0.34357,-3.0037) -- (0.35971,-2.98504) -- (0.37562,-2.96615) -- (0.39127,-2.94704) -- (0.40668,-2.92772) -- (0.42183,-2.90819) -- (0.43672,-2.88844) -- (0.45136,-2.86848) -- (0.46574,-2.84831) -- (0.47985,-2.82792) -- (0.49371,-2.80733) -- (0.50729,-2.78654) -- (0.52061,-2.76553) -- (0.53365,-2.74432) -- (0.54643,-2.72291) -- (0.55892,-2.7013) -- (0.57114,-2.67948) -- (0.58308,-2.65746) -- (0.59473,-2.63525) -- (0.6061,-2.61283) -- (0.61718,-2.59022) -- (0.62796,-2.56742) -- (0.63846,-2.54441) -- (0.64866,-2.52122) -- (0.65856,-2.49782) -- (0.66816,-2.47424) -- (0.67745,-2.45046) -- (0.68644,-2.4265) -- (0.69512,-2.40234) -- (0.70349,-2.37799) -- (0.71154,-2.35345) -- (0.71927,-2.32872) -- (0.72668,-2.30381) -- (0.73377,-2.2787) -- (0.74052,-2.25341) -- (0.74695,-2.22793) -- (0.75304,-2.20226) -- (0.7588,-2.1764) -- (0.76421,-2.15035) -- (0.76927,-2.12412) -- (0.77399,-2.0977) -- (0.77836,-2.07109) -- (0.78236,-2.04429) -- (0.78601,-2.0173) -- (0.78929,-1.99012) -- (0.79219,-1.96275) -- (0.79473,-1.93519) -- (0.79688,-1.90743) -- (0.79865,-1.87947) -- (0.80003,-1.85132) -- (0.80101,-1.82298) -- (0.80159,-1.79442) -- (0.80176,-1.76567) -- (0.80151,-1.73671) -- (0.80085,-1.70754) -- (0.79975,-1.67816) -- (0.79823,-1.64857) -- (0.79625,-1.61875) -- (0.79383,-1.58871) -- (0.79094,-1.55844) -- (0.78758,-1.52793) -- (0.78374,-1.49719) -- (0.7794,-1.46619) -- (0.77457,-1.43494) -- (0.76921,-1.40343) -- (0.76333,-1.37164) -- (0.7569,-1.33958) -- (0.7499,-1.30721) -- (0.74233,-1.27455) -- (0.73416,-1.24155) -- (0.72537,-1.20823) -- (0.71594,-1.17455) -- (0.70584,-1.14049) -- (0.69505,-1.10604) -- (0.68353,-1.07116) -- (0.67124,-1.03583) -- (0.65815,-1.00001) -- (0.64421,-0.96367) -- (0.62936,-0.92675) -- (0.61355,-0.88921) -- (0.59671,-0.85099) -- (0.57875,-0.812) -- (0.55957,-0.77216) -- (0.53905,-0.73136) -- (0.51705,-0.68946) -- (0.49336,-0.64628) -- (0.46776,-0.60161) -- (0.43993,-0.55514) -- (0.40944,-0.50646) -- (0.37565,-0.45497) -- (0.33761,-0.39975) -- (0.29372,-0.33923) -- (0.24086,-0.27037) -- (0.171,-0.18523) -- (0.171,-0.18523) -- (0.16229,-0.17504) -- (0.15307,-0.16435) -- (0.14323,-0.15306) -- (0.13266,-0.14103) -- (0.12114,-0.12808) -- (0.10839,-0.11391) -- (0.0939,-0.09801) -- (0.0767,-0.07941) -- (0.05425,-0.0556) -- (0.,0.) -- (-0.05425,0.05298) -- (-0.07671,0.07418) -- (-0.09392,0.09016) -- (-0.10842,0.10344) -- (-0.12118,0.11499) -- (-0.13272,0.12533) -- (-0.14331,0.13473) -- (-0.15317,0.1434) -- (-0.16242,0.15148) -- (-0.17117,0.15905) -- (-0.17117,0.15905) -- (-0.24152,0.21801) -- (-0.29521,0.2607) -- (-0.34025,0.29505) -- (-0.37977,0.32411) -- (-0.41538,0.34945) -- (-0.44802,0.37199) -- (-0.47833,0.39234) -- (-0.50673,0.4109) -- (-0.53355,0.42798) -- (-0.55903,0.44381) -- (-0.58334,0.45856) -- (-0.60664,0.47237) -- (-0.62906,0.48535) -- (-0.65069,0.4976) -- (-0.67161,0.50919) -- (-0.6919,0.52019) -- (-0.71161,0.53065) -- (-0.73081,0.54062) -- (-0.74952,0.55013) -- (-0.76781,0.55924) -- (-0.78569,0.56796) -- (-0.80321,0.57632) -- (-0.82038,0.58435) -- (-0.83725,0.59208) -- (-0.85382,0.59952) -- (-0.87012,0.60668) -- (-0.88617,0.61359) -- (-0.90199,0.62026) -- (-0.91759,0.62671) -- (-0.93298,0.63293) -- (-0.94819,0.63896) -- (-0.96322,0.64479) -- (-0.97808,0.65043) -- (-0.99279,0.6559) -- (-1.00735,0.6612) -- (-1.02178,0.66634) -- (-1.03608,0.67133) -- (-1.05026,0.67617) -- (-1.06433,0.68086) -- (-1.0783,0.68542) -- (-1.09217,0.68985) -- (-1.10595,0.69416) -- (-1.11965,0.69834) -- (-1.13327,0.7024) -- (-1.14682,0.70635) -- (-1.1603,0.71019) -- (-1.17372,0.71393) -- (-1.18709,0.71756) -- (-1.2004,0.72109) -- (-1.21366,0.72452) -- (-1.22688,0.72786) -- (-1.24006,0.7311) -- (-1.2532,0.73426) -- (-1.26632,0.73733) -- (-1.2794,0.74031) -- (-1.29246,0.74321) -- (-1.3055,0.74603) -- (-1.31852,0.74877) -- (-1.33153,0.75143) -- (-1.34453,0.75402) -- (-1.35751,0.75653) -- (-1.37049,0.75897) -- (-1.38347,0.76133) -- (-1.39645,0.76363) -- (-1.40943,0.76585) -- (-1.42242,0.76801) -- (-1.43541,0.7701) -- (-1.44842,0.77212) -- (-1.46143,0.77408) -- (-1.47447,0.77597) -- (-1.48752,0.7778) -- (-1.50059,0.77957) -- (-1.51368,0.78127) -- (-1.5268,0.78292) -- (-1.53995,0.7845) -- (-1.55312,0.78602) -- (-1.56632,0.78748) -- (-1.57956,0.78889) -- (-1.59284,0.79023) -- (-1.60615,0.79152) -- (-1.6195,0.79275) -- (-1.63289,0.79392) -- (-1.64633,0.79503) -- (-1.65982,0.79609) -- (-1.67335,0.79709) -- (-1.68693,0.79804) -- (-1.70056,0.79893) -- (-1.71425,0.79976) -- (-1.72799,0.80053) -- (-1.74179,0.80126) -- (-1.75565,0.80192) -- (-1.76957,0.80253) -- (-1.78356,0.80308) -- (-1.79761,0.80358) -- (-1.81173,0.80403) -- (-1.82592,0.80441) -- (-1.84018,0.80474) -- (-1.85452,0.80502) -- (-1.86893,0.80524) -- (-1.88342,0.8054) -- (-1.89799,0.80551) -- (-1.91264,0.80556) -- (-1.92737,0.80555) -- (-1.94219,0.80549) -- (-1.9571,0.80537) -- (-1.9721,0.80519) -- (-1.98719,0.80495) -- (-2.00238,0.80465) -- (-2.01766,0.8043) -- (-2.03304,0.80389) -- (-2.04852,0.80341) -- (-2.06411,0.80288) -- (-2.0798,0.80228) -- (-2.0956,0.80163) -- (-2.1115,0.80091) -- (-2.12752,0.80013) -- (-2.14366,0.79929) -- (-2.15991,0.79839) -- (-2.17628,0.79742) -- (-2.19278,0.79639) -- (-2.2094,0.79529) -- (-2.22614,0.79412) -- (-2.24302,0.79289) -- (-2.26002,0.79159) -- (-2.27717,0.79023) -- (-2.29445,0.78879) -- (-2.31187,0.78729) -- (-2.32943,0.78571) -- (-2.34714,0.78407) -- (-2.365,0.78235) -- (-2.38301,0.78055) -- (-2.40117,0.77869) -- (-2.4195,0.77675) -- (-2.43798,0.77473) -- (-2.45663,0.77263) -- (-2.47544,0.77046) -- (-2.49443,0.76821) -- (-2.51359,0.76588) -- (-2.53293,0.76346) -- (-2.55245,0.76097) -- (-2.57215,0.75839) -- (-2.59205,0.75572) -- (-2.61213,0.75297) -- (-2.63241,0.75013) -- (-2.65289,0.7472) -- (-2.67358,0.74419) -- (-2.69447,0.74108) -- (-2.71558,0.73788) -- (-2.7369,0.73458) -- (-2.75844,0.73119) -- (-2.78021,0.72771) -- (-2.80221,0.72412) -- (-2.82444,0.72044) -- (-2.84692,0.71665) -- (-2.86964,0.71276) -- (-2.8926,0.70877) -- (-2.91583,0.70467) -- (-2.93931,0.70047) -- (-2.96306,0.69615) -- (-2.98708,0.69173) -- (-3.01138,0.68719) -- (-3.03597,0.68254) -- (-3.06084,0.67778) -- (-3.08601,0.6729) -- (-3.11148,0.6679) -- (-3.13726,0.66277) -- (-3.16335,0.65753) -- (-3.18977,0.65217) -- (-3.21652,0.64668) -- (-3.2436,0.64106) -- (-3.27104,0.63531) -- (-3.29882,0.62944) -- (-3.32696,0.62343) -- (-3.32761,0.62329);

\draw[c0,line width=0.5mm] (-3.32761,0.62329) -- (-3.33016,0.6172) -- (-3.33889,0.59579) -- (-3.34738,0.57415) -- (-3.35563,0.55227) -- (-3.36361,0.53016) -- (-3.37133,0.50782) -- (-3.37877,0.48527) -- (-3.38593,0.4625) -- (-3.39279,0.43951) -- (-3.39934,0.41633) -- (-3.40558,0.39295) -- (-3.41149,0.36937) -- (-3.41707,0.34562) -- (-3.42231,0.3217) -- (-3.4272,0.29761) -- (-3.43172,0.27337) -- (-3.43589,0.24899) -- (-3.43968,0.22448) -- (-3.44309,0.19985) -- (-3.44611,0.17511) -- (-3.44874,0.15027) -- (-3.45097,0.12535) -- (-3.45281,0.10037) -- (-3.45424,0.07532) -- (-3.45526,0.05024) -- (-3.45588,0.02513) -- (-3.45588,-0.02513) -- (-3.45526,-0.05024) -- (-3.45424,-0.07532) -- (-3.45281,-0.10037) -- (-3.45097,-0.12535) -- (-3.44874,-0.15027) -- (-3.44611,-0.17511) -- (-3.44309,-0.19985) -- (-3.43968,-0.22448) -- (-3.43589,-0.24899) -- (-3.43172,-0.27337) -- (-3.4272,-0.29761) -- (-3.42231,-0.3217) -- (-3.41707,-0.34562) -- (-3.41149,-0.36937) -- (-3.40558,-0.39295) -- (-3.39934,-0.41633) -- (-3.39279,-0.43951) -- (-3.38593,-0.4625) -- (-3.37877,-0.48527) -- (-3.37133,-0.50782) -- (-3.36361,-0.53016) -- (-3.35563,-0.55227) -- (-3.34738,-0.57415) -- (-3.33889,-0.59579) -- (-3.33016,-0.6172) -- (-3.32761,-0.62329);

\draw[c1,line width=0.5mm] (-7.5,0.) -- (-7.26442,0.22829) -- (-7.02907,0.44223) -- (-6.79419,0.64224) -- (-6.56,0.82872) -- (-6.32674,1.00206) -- (-6.09464,1.16261) -- (-5.86393,1.31074) -- (-5.63483,1.44678) -- (-5.40757,1.57104) -- (-5.18237,1.68385) -- (-4.95947,1.78552) -- (-4.73907,1.87633) -- (-4.52139,1.95658) -- (-4.30666,2.02656) -- (-4.09507,2.08654) -- (-3.88685,2.13681) -- (-3.68219,2.17764) -- (-3.4813,2.2093) -- (-3.28437,2.23206) -- (-3.09161,2.24619) -- (-2.9032,2.25195) -- (-2.71932,2.24962) -- (-2.54016,2.23945) -- (-2.3659,2.22172) -- (-2.1967,2.1967) -- (-2.03274,2.16464) -- (-1.87417,2.12583) -- (-1.72115,2.08051) -- (-1.57384,2.02898) -- (-1.43237,1.97149) -- (-1.2969,1.90832) -- (-1.16754,1.83975) -- (-1.04443,1.76604) -- (-0.9277,1.68748) -- (-0.81745,1.60434) -- (-0.7138,1.5169) -- (-0.61684,1.42543) -- (-0.52668,1.33023) -- (-0.44339,1.23157) -- (-0.36708,1.12974) -- (-0.2978,1.02503) -- (-0.23563,0.9177) -- (-0.18062,0.80807) -- (-0.13285,0.6964) -- (-0.09234,0.583) -- (-0.05914,0.46814) -- (-0.03329,0.35212) -- (-0.0148,0.23523) -- (-0.0037,0.11776) -- (0.,0.) -- (-0.0037,-0.11776) -- (-0.0148,-0.23523) -- (-0.03329,-0.35212) -- (-0.05914,-0.46814) -- (-0.09234,-0.583) -- (-0.13285,-0.6964) -- (-0.18062,-0.80807) -- (-0.23563,-0.9177) -- (-0.2978,-1.02503) -- (-0.36708,-1.12974) -- (-0.44339,-1.23157) -- (-0.52668,-1.33023) -- (-0.61684,-1.42543) -- (-0.7138,-1.5169) -- (-0.81745,-1.60434) -- (-0.9277,-1.68748) -- (-1.04443,-1.76604) -- (-1.16754,-1.83975) -- (-1.2969,-1.90832) -- (-1.43237,-1.97149) -- (-1.57384,-2.02898) -- (-1.72115,-2.08051) -- (-1.87417,-2.12583) -- (-2.03274,-2.16464) -- (-2.1967,-2.1967) -- (-2.3659,-2.22172) -- (-2.54016,-2.23945) -- (-2.71932,-2.24962) -- (-2.9032,-2.25195) -- (-3.09161,-2.24619) -- (-3.28437,-2.23206) -- (-3.4813,-2.2093) -- (-3.68219,-2.17764) -- (-3.88685,-2.13681) -- (-4.09507,-2.08654) -- (-4.30666,-2.02656) -- (-4.52139,-1.95658) -- (-4.73907,-1.87633) -- (-4.95947,-1.78552) -- (-5.18237,-1.68385) -- (-5.40757,-1.57104) -- (-5.63483,-1.44678) -- (-5.86393,-1.31074) -- (-6.09464,-1.16261) -- (-6.32674,-1.00206) -- (-6.56,-0.82872) -- (-6.79419,-0.64224) -- (-7.02907,-0.44223) -- (-7.26442,-0.22829) -- (-7.5,0.);

\draw[c0,line width=0.5mm,dash pattern= on 3pt off 5pt] (-8.25,0) -- (-7.5,0);
\draw[c1,line width=0.5mm,dash pattern= on 3pt off 5pt,dash phase=4pt] (-8.25,0) -- (-7.5,0);

\node[black] at (-7.5 ,0.25){$-1$};
\node[black] at (1.875 ,0.25){$1/4$};
\node[black] at (0.5,3.75){$i/4$};
\node[black] at (0.5,-3.75){$-i/4$};

\end{tikzpicture}
\caption{Parameters $z = 1/\lambda$ for which the transfer matrix has two maximal eigenvalues (non-persistently) of the same norm for $C_2$ in blue and $C_4$ in red. These curves are accumulation points of the zeros of the polynomials $\{Z(C_n\square C_2;\lambda)\}_{n\geq 1}$ and $\{Z(C_n\square C_4;\lambda)\}_{n\geq 1}$ respectively. The other accumulation points in $\lambda$ coordinates are given by the real intervals with approximate bounds $[-1, -0.172]$ and $[-1, -0.126]$ respectively.}
\label{fig: acc_par}
\end{figure}

Let $\mathcal{T}_m = \mathbb{Z}_{m}^{d-1}$ and let $a_m \geq 2m$ be an even integer such that $Z(C_{a_m} \square \mathcal{T}_{2m};\lambda)$ has a zero with norm at least $m$. Such an $a_m$ exists by Lemma~\ref{lem: zer_acc_infty_basic}. Now $\{C_{a_m}\square\mathcal{T}_{2m}\}_{m \geq 1}$ is a sequence of tori whose sidelengths all converge to $\infty$ and whose zeros are unbounded. The first part of the main theorem, proved in the previous section, shows that for every $C > 0$ there are only finitely many $m$ such that $a_m \leq e^{C m}$, i.e. $\log(a_m) = \omega(m)$.
In the next section we will show that $\log(a_m)$ can be chosen to not grow faster than $m^{3(d-1)}$.

\subsection{Explicit bounds}

The remainder of this section is dedicated to proving a more quantitative version of Lemma~\ref{lem: zer_acc_infty_basic}. Let $\mathcal{T}$ be an even torus, $\alpha = |V(\mathcal{T})|/2$ and $N=|\mathcal{I}(\mathcal{T})|$. We shall prove the following.
\begin{restatable}{theorem}{explicitBounds}
    \label{thm: zeros_at_inf}
    Let $R > (6N^2)^{\alpha+2}$ and $n \geq 80 \cdot R^\alpha$ then $Z(C_n \square \mathcal{T}; \lambda)$ has at least $\frac{1}{16}n R^{-\alpha}$ distinct zeros with magnitude at least $R$.
\end{restatable}

Once we have proved the above, we quickly obtain a proof of the second part of the main  theorem:

\begin{theorem*}[Second part of Main Theorem]
Let $\mathcal{F}$ be a highly unbalanced family of even tori. The zeros of the independence polynomials $\{Z(\mathcal{T};\lambda): \mathcal{T} \in \mathcal{F}\}$ are not uniformly bounded.
\end{theorem*}
\begin{proof}
    For every $\mathcal{T} \in \mathcal{F}$ write $\ell(\mathcal{T})$ for the longest side length of $\mathcal{T}$. Furthermore, let $\mathcal{R}(\mathcal{T})$ be the torus for which $\mathcal{T} \cong C_{\ell(\mathcal{T})} \square \mathcal{R}(\mathcal{T}) $. Now define 
    \[
        \mathcal{F}' = \{\mathcal{T} \in \mathcal{F}: \ell(\mathcal{T}) \geq 80 \cdot 6^{3|\mathcal{R}(\mathcal{T})|^2} \cdot 2^{6|\mathcal{R}(\mathcal{T})|^3}\}.
    \]
    Because $\mathcal{F}$ is highly unbalanced $\mathcal{F}'$ contains infinitely many elements. We distinguish between the case where $\{\mathcal{R}(\mathcal{T}): \mathcal{T} \in \mathcal{F}\}$ is finite or infinite.

    In the former case there is a fixed torus $\mathcal{T}$ such that $\mathcal{F}$ contains infinitely many elements of the form $C_n \square \mathcal{T} $. Their zeros are unbounded according to Lemma~\ref{lem: zer_acc_infty_basic}.

    In the latter case let $\mathcal{T}_n \in \mathcal{F}'$ be a sequence for which $|\mathcal{R}(\mathcal{T}_n)|$ tends to infinity. Let $R_\mathcal{T} = 6^{3|\mathcal{R}(\mathcal{T})|}\cdot 2^{6|\mathcal{R}(\mathcal{T})|}$. Because $|\mathcal{I}(\mathcal{R}(\mathcal{T}))|< 2^{|\mathcal{R}(\mathcal{T})|}$ we can apply Theorem~\ref{thm: zeros_at_inf} to see that $Z(\mathcal{T};\lambda) = Z(C_{\ell(\mathcal{T})} \square \mathcal{R}(\mathcal{T});\lambda)$ has at least one zero with magnitude at least $R_\mathcal{T}$ for any $\mathcal{T} \in \mathcal{F}'$. The theorem now follows from the fact that $R_{\mathcal{T}_n}$ tends to infinity. 
\end{proof}

The remainder of this section focuses on proving Theorem \ref{thm: zeros_at_inf}.

\subsubsection{\texorpdfstring{The eigenvalues $q^+$ and $q^-$}{The eigenvalues q+ and q-}}

We again let $\mathcal{T}$ be a fixed torus whose sidelengths are all even. We recall that we defined the rescaled transfer-matrix $M_z$ with eigenvalues $q^+, q^-$ holomorphic in a neighborhood of $z=0$. We also recall the two independent sets $\Seven$ and $\Sodd$ of size $\alpha$. In this section we will investigate the series expansion of $q^{\pm}$. For example when $\mathcal{T} = C_8$ we have 
\[
 q^{+}(z) = 1 + 4 z + 6 z^2 + 8 z^3 + 44 z^4 + \mathcal{O}(z^5)
 \quad
 \text{ and }
 \quad
 q^{-}(z) = -1 - 4 z - 6 z^2 - 8 z^3 + 26 z^4 + \mathcal{O}(z^5).
\]
We will show that the coefficient of $z^m$ of $q^+$ is minus that of $q^{-}$ for $m = 0, \dots, \alpha-1$, while the coefficients of $z^\alpha$ differ in magnitude. This is done so that in the end we can get a handle on the map $\beta(z) = q^+(z)/q^{-}(z)$ and the branches of its inverse.

For any $k \in \{0, \dots, \alpha\}$ we define $Q_{k}$ as the projection of a vector on the subspace spanned by $\{e_S\}_{\|S\| = k}$, i.e. 
\[
    Q_{k} = \sum_{\substack{S \in \mathcal{I} \\ \|S\| = k}} e_{S}e_S^T.
\]
Observe that $Q_0 + Q_1 + \cdots + Q_\alpha = I_{|\mathcal{I}|}$. 

We define $v_0^{+} = e_{\Seven} + e_{\Sodd}$ and $v_0^{-} = e_{\Seven} - e_{\Sodd}$. We also define $q_0^{+} = 1$ and $q_{0}^{-} = -1$. For $n \geq 1$ recursively define the sequences of vectors $v_n^{\pm}$ and of integers $q_{n}^{\pm}$ by 
\begin{equation}
    \label{eq: recusion eigenvectors}
    v_n^{\pm} = q_0^{\pm}\Big(\sum_{k=1}^{\min{(n,\alpha)}} Q_{k}Av_{n-k}^{\pm} -  \sum_{i=1}^{n-1} q_i^{\pm} v_{n-i}^{\pm}\Big)
    \quad 
    \text{ and }
    \quad 
    q_{n}^{\pm} = e_{\Seven}^T A v_n^{\pm}.
\end{equation}
Observe that $q_{n}^{\pm} = e_{\Seven}^T A v_n^{\pm}$ also holds for $n=0$. We furthermore define the (formal) power series
\begin{equation}
    \label{eq: power_series_eigensystem}
    v^{\pm}(z) = \sum_{n=0}^\infty v_n^{\pm} z^n
    \quad 
    \text{ and }
    \quad 
    q^{\pm}(z) = \sum_{n=0}^\infty q_n^{\pm} z^n.
\end{equation}
We will show that $(q^{\pm}, v^{\pm})$ form two eigenvalue-eigenvector pairs corresponding to $q^{\pm}$ as defined in Lemma~\ref{lem: q definition}. This is will technically be an equality of formal power series until we prove that $q^{\pm}$ and and the entries of $v^{\pm}$ are analytic around $0$, which we will subsequently do. We first identify a certain symmetry in the entries of $v_n$. 

Let $\sigma \in \mathrm{Aut}(\mathcal{T})$. For any $S \in \mathcal{I}$ we define
\[
  S^{\sigma} = \{\sigma(v): v \in S\}.  
\]
The map $\epsilon: \mathrm{Aut}(\mathcal{T}) \to \{\pm 1\}$ given by $\epsilon(\sigma) = 1$ if $\Seven^\sigma = \Seven$ and $\epsilon(\sigma) = -1$ if $\Seven^\sigma = \Sodd$ is a group homomorphism. An autormorphism $\sigma \in \mathrm{Aut}(\mathcal{T})$ is called even or odd according to whether $\epsilon(\sigma) = 1$ or $\epsilon(\sigma) = -1$ respectively. We define the permutation matrix $P_\sigma$ by $P_\sigma e_S = e_{S^\sigma}$ and we observe that $P_\sigma Q_k = Q_k P_\sigma$ and $P_\sigma A = A P_{\sigma}$.

\begin{lemma}
    \label{lem: symmetry eigenvectors}
    Let $n \in \mathbb{Z}_{\geq 0}$ and $\sigma \in \mathrm{Aut}(\mathcal{T})$. If $\sigma$ is even then $P_\sigma v_{n}^{\pm} = v_n$, while if $\sigma$ is odd then $P_\sigma v_n^{\pm} = \pm v_n^{\pm}$.
\end{lemma}

\begin{proof}
    For $n=0$ the statement follows directly from the definitions. For $n \geq 1$ we have 
    \[
        P_\sigma v_n^\pm = 
        q_0^{\pm}\Big(\sum_{k=1}^{\min{(n,\alpha)}} Q_{k}AP_{\sigma}v_{n-k}^{\pm} -  \sum_{i=1}^{n-1} q_i^{\pm} P_\sigma v_{n-i}^{\pm}\Big).
    \]
    The statement follows inductively.
\end{proof}

We now prove that $(q^{\pm},v^{\pm})$ indeed form two eigenvalue-eigenvector pairs.

\begin{lemma}
    As power series in $z$ we have $M_z v^\pm(z) = q^\pm(z) v^\pm(z)$.
\end{lemma}

\begin{proof}
    We first claim that for any $n \in \mathbb{Z}_{\geq 0}$ we have $Q_0 A v_n^\pm = q_n^\pm v_0^\pm$. Let $\sigma \in \mathrm{Aut}(\mathcal{T})$ be an odd permutation. Then
    \begin{align*}
        Q_0 A v_n^\pm &= e_{\Seven} e_{\Seven}^T A v_n^{\pm} + e_{\Sodd} e_{\Sodd}^T A v_n^{\pm}\\
        &= e_{\Seven} e_{\Seven}^T A v_n^{\pm} + e_{\Sodd} e_{\Sodd}^T P_\sigma A P_{\sigma^{-1}} v_n^{\pm}\\
        &= (e_{\Seven}\pm e_{\Sodd}) e_{\Seven}^T A v_{n}^{\pm}\\
        &= q_n^\pm v_0^\pm,
    \end{align*}
    where we have used Lemma~\ref{lem: symmetry eigenvectors} to equate $P_{\sigma^{-1}} v_n^\pm$ with $\pm v_n$.
    
    We now prove the statement in the lemma. Observe that 
    \[  
        M_z v^{\pm}(z)
        = 
        \Big(\sum_{k=0}^\alpha Q_k A z^k\Big) \Big(\sum_{n=0}^{\infty} v_n^\pm z^n\Big) = \sum_{n=0}^\infty \Big[\sum_{k=0}^{\min(n,\alpha)} Q_k A v_{n-k}^\pm\Big]z^n.
    \]
    Moreover,
    \[
        q^\pm(z) v^\pm(z)
        = \Big(\sum_{n=0}^\infty q_n^{\pm} z^n\Big) \Big(\sum_{n=0}^{\infty} v_n^\pm z^n\Big)
        =
        \sum_{n=0}^\infty \Big[\sum_{i=0}^n q_i^\pm v_{n-i}^\pm\Big]z^n.
    \]
    It is thus sufficient to prove that for all $n$
    \[
        \sum_{i=0}^n q_i^\pm v_{n-i}^\pm =  \sum_{k=0}^{\min(n,\alpha)} Q_k A v_{n-k}^\pm.
    \]
    For $n=0$ the statement reads $q_0^\pm v_0^\pm = Q_0 A v_0^\pm$, which is equivalent to the claim above for $n=0$. For $n \geq 1$ we reason inductively as follows. 
    \begin{align*}
        \sum_{i=0}^n q_i^\pm v_{n-i}^\pm &= 
        \sum_{i=1}^{n-1} q_i^\pm v_{n-i}^\pm + q_0^\pm v_n^{\pm} + q_n^\pm v_0^{\pm}\\
        &= \sum_{i=1}^{n-1} q_i^\pm v_{n-i}^\pm + 
        \left(q_0^{\pm}\right)^2\Big(\sum_{k=1}^{\min{(n,\alpha)}} Q_{k}Av_{n-k}^{\pm} -  \sum_{i=1}^{n-1} q_i^{\pm} v_{n-i}^{\pm}\Big)
        + Q_0 A v_n^\pm\\
        &= \sum_{k=0}^{\min(n,\alpha)} Q_k A v_{n-k}^\pm.
    \end{align*}
    This concludes the proof of the lemma.
\end{proof}

Now we will prove that both $q^{\pm}$ and the entries of $v^{\pm}$ are indeed analytic around $z=0$. In what follows we let $N = |\mathcal{I}|$ so that the vectors $v_{n}^{\pm}$ are $N$-dimensional. We first prove an elementary lemma on a certain sequence of integers that will serve as an upper bound for the entries of $v_n$ and $q_n$.

\begin{lemma}
    \label{lem: sequence bound}
    Define the sequence $\{x_n\}_{n \geq 0}$ by $x_0 = 1$ and for $n \geq 1$
    \[
        x_n = N \cdot \Big(x_{n-1} + \sum_{i=1}^{n-1} x_i x_{n-i}\Big).
    \]
    Then $x_n \leq (6N^2)^n$.
\end{lemma}

\begin{proof}
    Let $y_n(N) = x_n/N^{2n}$. We observe that $y_0(N) = 1$ and
    \[
        y_n(N) = \frac{1}{N} y_{n-1}(N) + N \sum_{i=1}^{n-1} y_i(N)y_{n-i}(N).
    \]
    It follows that $y_1(N)=1/N$ and inductively $y_{n}(N)$ is a polynomial in $1/N$ with positive coefficients and constant term equal to zero. We can conclude that $y_n(N) \leq y_n(1)$ and thus it remains to show that $y_n(1) \leq 6^n$ for all $n \geq 0$. 
    
    We denote $y_n(1)$ by $y_n$ and prove that
    \begin{equation}\label{strong inequality}
    y_n \le \frac{6^n}{(n+1)^2},
    \end{equation}
    which of course implies the desired inequality. Computer computations show that \eqref{strong inequality} is satisfied for $n=1,\ldots, 199$. Suppose that \eqref{strong inequality} is satisfied for all values $0, \ldots, n-1$ for some $n \ge 200$. We observe
    \[
    y_n = y_{n-1} + \sum_{i=1}^{n-1} y_i y_{n-i}  \leq y_{n-1} + 2\sum_{i=1}^{99} y_i y_{n-i} + 2 \sum_{i=100}^{\lfloor n/2\rfloor} y_i y_{n-i}.
    \]
    Using the induction hypothesis we find that 
    \[
        \frac{(n+1)^2}{6^n}\Big(y_{n-1} + 2\sum_{i=1}^{99} y_i y_{n-i}\Big) \leq (n+1)^2 \Big(\frac{1}{6 n^2} + 2\sum_{i=1}^{99} \frac{y_i}{6^{i} (n+1-i)^2} \Big).
    \]
    The right-hand side is an explicit decreasing rational function in $n$ and thus upper bounded by the value obtained from plugging in $n=200$, yielding an upper bound of $0.87$. We also find 
    \[
    \frac{(n+1)^2}{6^n}\Big(2 \sum_{i=100}^{\lfloor n/2\rfloor} y_i y_{n-i}\Big) 
    \leq 2 \sum_{i=100}^{\lfloor n/2\rfloor} \left(\frac{n+1}{(i+1)(n+1-i)}\right)^2 \leq 8 \sum_{i=100}^\infty \frac{1}{(i+1)^2} \leq 0.08.
    \]
    Putting these two estimates together we conclude that $y_n \leq (0.87 + 0.08) \frac{6^n}{(n+1)^2} \leq \frac{6^n}{(n+1)^2}$.
\end{proof}

\begin{lemma}
    \label{lem: q bound}
    We have $|q^{\pm}_n| \leq N \cdot (6N^2)^n$ and $|(v^{\pm}_n)_S| \leq (6N^2)^n$ for all $n\geq 0$ and $S \in \mathcal{I}$.
\end{lemma}

\begin{proof}
For a vector $v$ let $|v|$ denote the vector whose entries are the magnitudes of the entries of $v$. For two vectors $v_1,v_2$ we write $v_1 \leq v_2$ if the inequality holds entrywise. We let $\mathds{1}$ denote the $N$-dimensional vector whose entries are all equal to $1$. We inductively prove that $|v^{\pm}_n| \leq x_n \cdot \mathds{1}$ and $|q_n^\pm| \leq N \cdot x_n$, where $x_n$ is defined as in Lemma~\ref{lem: sequence bound}. This is sufficient by the bound proved in that lemma. 

For $n=0$ this follows by definition. For larger $n$ we use the recursion in equation~(\ref{eq: recusion eigenvectors}) to obtain
\begin{align*}
    |v_n^{\pm}| &\leq \sum_{k=1}^{\min{(n,\alpha)}} |Q_{k}Av_{n-k}^{\pm}| +  \sum_{i=1}^{n-1} |q_i^{\pm} v_{n-i}^{\pm}| \\
    & \leq x_{n-1} \sum_{k=1}^{\min{(n,\alpha)}} Q_{k}A\ \mathds{1} + \Big(\sum_{i=1}^{n-1}N x_i x_{n-i} \Big) \mathds{1}\\
    &\leq N \cdot \Big(x_{n-1} + \sum_{i=1}^{n-1} x_i x_{n-i}\Big)\cdot \mathds{1} = x_{n} \cdot \mathds{1}.
\end{align*}
We also obtain $|q_{n}^{\pm}| = |e_{\Seven}^T A v_n^{\pm}| \leq e_{\Seven}^T A\ \mathds{1} \cdot x_n \leq N \cdot x_n$.
\end{proof}

\begin{corollary}
    The functions $q^{\pm}$ and the entries of $v^{\pm}$ define holomorphic functions in a disk of radius $1/(6N^2)$. On that disk they form two eigenvalue-eigenvector pairs.
\end{corollary}

\subsubsection[The sum q+ + q-]{The sum $q^+ + q^-$}
Define $u_n = \frac{1}{2}(v_{n}^{+} + v_n^-)$, $a_n = \frac{1}{2}(q_n^+ + q_n^-)$ and $b_n = \frac{1}{2}(q_n^+ - q_n^-)$. Our goal is to show that $a_n=0$ for $n=0,\dots,\alpha-1$ and $a_\alpha > 0$.
We will start by deriving a useful recurrence for the $u_n$.

\begin{lemma}
    \label{lem: recursion u_n}
    Let $\sigma \in \mathrm{Aut}(\mathcal{T})$ be an odd permutation. Then for all $n \geq 1$
    \[
        u_n = \sum_{k=1}^{\min{(n,\alpha)}} Q_{k}Au_{n-k}^{\sigma} -  \sum_{i=1}^{n-1} \left(a_i u_{n-i}^{\sigma} + b_i u_{n-i}\right),
    \]
    moreover,
    \[
        a_n = e_{\Seven}^T A u_n
        \quad
        \text{ and }
        \quad 
        b_n = e_{\Sodd}^T A u_n.
    \]
\end{lemma}
\begin{proof}
It follows from Lemma~\ref{lem: symmetry eigenvectors} that $u_n^\sigma = \frac{1}{2}(v_n^+ - v_n^-)$ and thus $v_n^\pm = u_n \pm u_n^\sigma$. We similarly have $q_n^{\pm} = a_n \pm b_n$. We now use the recursion for $v_n^\pm$ defined in {(\ref{eq: recusion eigenvectors})} to get a recursion for $u_n$:
\begin{align*}
u_n &= \frac{1}{2}\Big[\Big(\sum_{k=1}^{\min{(n,\alpha)}} Q_{k}Av_{n-k}^{+} -  \sum_{i=1}^{n-1} q_i^{+} v_{n-i}^{+}\Big) - \Big(\sum_{k=1}^{\min{(n,\alpha)}} Q_{k}Av_{n-k}^{-} -  \sum_{i=1}^{n-1} q_i^{-} v_{n-i}^{-}\Big)\Big]\\
&= \sum_{k=1}^{\min{(n,\alpha)}} Q_{k}Au_{n-k}^{\sigma} -  \sum_{i=1}^{n-1} \frac{1}{2}\left(q_i^{+} v_{n-i}^{+}-q_i^{-} v_{n-i}^{-}\right).
\end{align*}
The claimed recursive formula for $u_n$ now follows from the following equality:
\[
\frac{1}{2}\left(q_i^{+} v_{n-i}^{+}-q_i^{-} v_{n-i}^{-}\right) = \frac{1}{2}\left[(a_i + b_i)(u_{n-i} + u_{n-i}^\sigma) - (a_i - b_i)(u_{n-i} - u_{n-i}^\sigma) \right] = a_i u_{n-i}^{\sigma} + b_i u_{n-i}.
\]
We use the part of equation {(\ref{eq: recusion eigenvectors})} that defines $q_n^\pm$ to observe that $a_n = \frac{1}{2}\left(e_{\Seven}^T A v_n^{+} + e_{\Seven}^T A v_n^{-}\right) = e_{\Seven}^T A u_n$ and $b_n = \frac{1}{2}\left(e_{\Seven}^T A v_n^{+} - e_{\Seven}^T A v_n^{-}\right) = e_{\Seven}^T A u_n^\sigma = e_{\Seven}^T P_\sigma A P_{\sigma^{-1}} u_n^\sigma = e_{\Sodd}^T A u_n$.
\end{proof}

The goal is to write the elements of $u_n$ as weighted paths of independent sets of $\torus$; see Lemma~\ref{lem: eigenvalue sum coefficients}. To make this formal we introduce some notation from formal language theory. 

For any set $F$ let $F^*$ denote the set of finite words of elements of $F$ (including the empty word denoted by $\emptyset_F$).  For $f \in F$ and $w \in F^*$ we use $f \in w$ to indicate that $f$ is a letter in the word $w$. For concatenation of two words $w_1,w_2 \in F^*$ we write $w_1 \cdot w_2$.

Let $\mathcal{I}_{\geq 1} = \{S \in \mathcal{I}: \|S\| \geq 1\}$. We define 
\[
    \mathcal{P} = \mathcal{I}_{\geq 1} \times  \mathbb{Z}_{\geq 1}^*.
\]
For $r \in \mathbb{Z}_{\geq 1}^*$ we let $\|r\|$ denote the sum of its entries with $\|\emptyset_{\mathbb{Z}_{\geq 1}}\| = 0$. For $p \in \mathcal{P}$ of the form $(S,r)$ we define the length and weight of $p$ respectively as 
\[
    \ell(p) = \|S\| + \|r\|
    \quad
    \text{ and }
    \quad
    W(p) = \prod_{n \in r} (-b_n).
\]
For an element $w \in \mathcal{P}^*$ we define 
\[
    \ell(w) = \sum_{p \in w} \ell(p)
    \quad
    \text{ and }
    \quad
    W(w) = \prod_{p \in w} W(p).
\]
An empty sum or product we treat as $0$ or $1$ respectively. 

Fix an odd $\sigma \in \mathrm{Aut}(\mathcal{T})$ with the property that $\sigma^2 = \text{id}$, (e.g. $(n_1,n_2, \dots, n_d) \mapsto (1-n_1,n_2,\dots, n_d)$). Define the subset $\mathcal{Q} \subseteq \mathcal{P}^*$ by 
\[
    \mathcal{Q} = \{(S_1, r_1) \cdots (S_m,r_m) \in \mathcal{P}^*: \Sodd \sim S_1 \text{ and } S_i^\sigma \sim S_{i+1} \text{ for all $i =1, \dots, m-1$} \} \cup \{\emptyset_{\mathcal{P}}\}.
\]

For any $S \in \mathcal{I}_{\geq 1}$ we let $\mathcal{Q}[S]$ denote the elements in $\mathcal{Q}$ that end in $(S,r)$ for some $r$. We let $\mathcal{Q}[\Seven] = \{\emptyset_\mathcal{P}\}$.
\begin{lemma}
    \label{lem: even comp -> length >= alpha}
    Let $S \in \mathcal{I}$ such that $S \sim \Seven$. For any $w \in \mathcal{Q}[S]$ we have $\ell(w) \geq \alpha$, moreover if $\ell(w) = \alpha$ then $W(w) = 1$.
\end{lemma}
\begin{proof}
Because $\Seven$ is not compatible with itself $w$ is not the empty word and thus we can write $w = (S_1,r_1) \cdots (S_m,r_m)$. Let $v$ be a vertex of $\mathcal{T}$. If $v \in S_i$ for some $i$ then it follows from the requirement that $S_i^\sigma \sim S_{i+1}$ that $\sigma(v) \not \in S_{i+1}$. Applying this fact to $\sigma(v)$ and using that $\sigma^2$ is the identity we see that $\sigma(v) \in S_i$ implies that $v \not \in S_{i+1}$. The possible transitions for $(\mathds{1}(v \in S_i),\mathds{1}(\sigma(v) \in S_i))$ are thus given in the following diagram.
\[
\begin{tikzcd}
        & (1,1) \arrow[d,leftrightarrow]    & \\
 (1,0) \arrow[r,leftrightarrow] \arrow[loop left]  & (0,0) \arrow[l,leftrightarrow] \arrow[r,leftrightarrow] \arrow[out=-140, in=-40, loop,distance=2em] & (0,1) \arrow[l,leftrightarrow] \arrow[loop right]
\end{tikzcd}
\]
Assume that $v$ is an even vertex. Because $S_1 \sim \Sodd$ we see that $(\mathds{1}(v \in S_1),\mathds{1}(\sigma(v) \in S_1))$ is of the form $(*,0)$. Similarly, because $S_m \sim \Seven$, we see that $(\mathds{1}(v \in S_m),\mathds{1}(\sigma(v) \in S_m))$ is of the form $(0,*)$. It thus follows that $(\mathds{1}(v \in S_i),\mathds{1}(\sigma(v) \in S_i))$ takes on the value $(0,0)$ at least once more often than it takes on the value $(1,1)$. From this we can conclude that 
\begin{equation}
    \label{eq: lower bound loss}
    \sum_{i=1}^m \left[1 - \mathds{1}(v \in S_i) - \mathds{1}(\sigma(v) \in S_i)\right] \geq 1.
\end{equation}
We now find that
\begin{align*}
    \ell(w) &= \sum_{i=1}^m \|S_i\| + \|r_i\|
            = \sum_{i=1}^m \Bigg[\sum_{\substack{v \in \mathcal{T}\\ v\text{ even}}} 1 - \mathds{1}(v \in S_i) - \mathds{1}(\sigma(v) \in S_i) \Bigg] + \sum_{i=1}^m \|r_i\|\\
            &\geq \alpha + \sum_{i=1}^m \|r_i\| \geq \alpha,
\end{align*}
where we interchanged the two summations and used {(\ref{eq: lower bound loss})}. We see that indeed $\ell(w) \geq \alpha$. Moreover, if $\ell(w) = \alpha$ then the final two inequalities must be equalities and thus $r_i = \emptyset_{\mathbb{Z}_{\geq 1}}$ for all $i$, which implies that $W(w) = 1$. 
\end{proof}

For any $n \geq 0$ and $S \in \mathcal{I}$ define 
\[
    \mathcal{Q}_n[S] = \{p \in \mathcal{Q}[S]: \ell(p) = n\}.
\]

\begin{lemma}
    \label{lem: eigenvalue sum coefficients}
Let $0 \leq n \leq \alpha$ and $S \in \mathcal{I}$. Then 
\begin{equation}
    \label{eq: entries as weights}
  e_S^T u_n = \sum_{w \in \mathcal{Q}_n[S]} W(w).
\end{equation}
Moreover, if $n \neq \alpha$ then $a_n = 0$, while $a_\alpha \geq 1$.
\end{lemma}
\begin{proof}
By definition $a_0 = 0$ and $u_0 = e_{\Seven}$. Moreover, $\mathcal{Q}_0[S]$ is non-empty only if $S=\Seven$ in which case it consists of the empty word. Therefore we see that for $n=0$ both sides of equation {(\ref{eq: entries as weights})} are equal to $1$ if $S=\Seven$ and equal to $0$ otherwise.

We will now prove the statement inductively, i.e. we let $1\leq n \leq \alpha$ and we assume that for all values $k < n$ both {(\ref{eq: entries as weights})} holds and $a_k=0$. 

First suppose that either $\|S\| = 0$ or $\|S\| > n$. Then it follows that $\mathcal{Q}_n[S]$ is empty and thus the right-hand side of {(\ref{eq: entries as weights})} is equal to $0$. Because $e_S^T Q_k = 0$ for $k \neq \|S\|$ we inductively obtain by Lemma~\ref{lem: recursion u_n} that in this case indeed the left-hand side is equal to
\[
    e_S^T u_n  = - \sum_{i=1}^{n-1} b_i e_S^T u_{n-i} = 0.
\]

Now suppose $1 \leq \|S\| \leq n$. We inductively find that the left-hand side of {(\ref{eq: entries as weights})} is equal to 
\begin{align*}
    e_S^T u_n &= e_S^T A u_{n-\|S\|}^\sigma - \sum_{i=1}^{n-1} b_i e_S^T u_{n-i}\\
    &= \sum_{\substack{X \in \mathcal{I} \\ X^\sigma \sim S}} e_X^T u_{n-\|S\|} + \sum_{i=1}^{n-1} (-b_i) e_S^T u_{n-i}\\
    &= \sum_{\substack{X \in \mathcal{I} \\ X^\sigma \sim S}} \sum_{w \in \mathcal{Q}_{n-\|S\|}[X]} W(w) + \sum_{i=1}^{n-1} \sum_{w \in \mathcal{Q}_{n-i}[S]} (-b_i) W(w).
\end{align*}
For any $T \in \mathcal{I}$, $k \in \mathbb{Z}_{\geq 1}$ and $i \in \mathbb{Z}_{\geq 1}$ let $\mathcal{Q}_i[T,k]$ be those elements of $\mathcal{Q}_i[T]$ ending in $(T,r)$ with $r$ ending in $k$. Moreover, let $\mathcal{Q}_i[T,0]$ denote those elements ending in $(T,\emptyset_{\mathbb{Z}_{\geq 1}})$. For $w \in \mathcal{Q}_i[T]$ we can write $w = w'\cdot(T,r)$ for some $r$. We let $w \oplus k$ denote the element $w'\cdot(T,r \cdot k) \in \mathcal{Q}_{i+k}[T]$. We have
\[
\sum_{w \in \mathcal{Q}_n[S,0]} W(w)= \sum_{\substack{X \in \mathcal{I} \\ X^\sigma \sim S}} \sum_{w' \in \mathcal{Q}_{n-\|S\|}[X]} W(w' \cdot (S,\emptyset)) = \sum_{\substack{X \in \mathcal{I} \\ X^\sigma \sim S}} \sum_{w' \in \mathcal{Q}_{n-\|S\|}[X]} W(w').
\]
While, if $i \in \{1, \dots, n-1\}$ we have 
\[
\sum_{w \in \mathcal{Q}_n[S,i]} W(w)= \sum_{w' \in \mathcal{Q}_{n-i}[S]} W(w \oplus i) = (-b_i) \cdot \sum_{w' \in \mathcal{Q}_{n-i}[S]} W(w).
\]
We thus have 
\[
    \sum_{w \in \mathcal{Q}_n[S]} W(w) = \sum_{i=0}^{n-1} \sum_{w \in \mathcal{Q}_n[S,i]} W(w) = \sum_{\substack{X \in \mathcal{I} \\ X^\sigma \sim S}} \sum_{w \in \mathcal{Q}_{n-\|S\|}[X]} W(w) + \sum_{i=1}^{n-1} \sum_{w \in \mathcal{Q}_{n-i}[S]} (-b_i) W(w),
\]
which proves equality {(\ref{eq: entries as weights})}.

We now have to show that $a_n = 0$ if $n < \alpha$ and $a_\alpha \geq 1$. It follows from Lemma~\ref{lem: recursion u_n} that 
\[
    a_n = e_{\Seven}^T A u_n = \sum_{\substack{S \in \mathcal{I}\\S \sim \Seven}} e_{S}^T u_n = \sum_{\substack{S \in \mathcal{I}\\S \sim \Seven}} \sum_{w \in \mathcal{Q}_n[S]} W(w).
\]
In Lemma~\ref{lem: even comp -> length >= alpha} it is shown that if $S \sim \Seven$ and $w \in Q_{n}[S]$, then $n \geq \alpha$. This shows that $a_n = 0$ for $n < \alpha$. Moreover, if $n=\alpha$ the lemma states that $W(w) = 1$. This shows that $a_n \geq 0$. Because $\emptyset$ is compatible with both $\Seven$ and $\Sodd$ we see that $(\emptyset,\emptyset_{\mathcal{Z}_{\geq 1}}) \in Q_\alpha[\emptyset]$ and thus we can conclude that $a_\alpha \geq 1$.
\end{proof}

\subsubsection{The other eigenvalues}
In this section we study the other eigenvalues of the transfer matrix $M_z$, i.e. those not equal to $q^{\pm}(z)$. Recall from Section~\ref{sec: cons_width_tori} that $M_z = \hat{D}_z A$, where $A$ is the compatibility matrix of the independent sets of $\mathcal{T}$ and $\hat{D}_z$ is a diagonal matrix with $(\hat{D}_z)_{S,S} = z^{\|S\|}$. In this section it will be more convenient to look at the symmetric transfer-matrix $\hat{M}_z = D_{z^{1/2}} A D_{z^{1/2}}$, where (for now) we make an arbitrary choice of $z^{1/2}$ for each $z$. The symmetric transfer-matrix $\hat{M}_z$ is conjugate to $M_z$ and thus has the same eigenvalues.

Recall that the matrix $\hat{M}_z$ is $N$-dimensional. For this section we order the indices of the $N$-dimensional vectors, indexed by elements of $\mathcal{I}$, in such a way that $\Seven$ and $\Sodd$ correspond to the final two coordinates. The $2\times2$ submatrix of $\hat{M}_z$ induced by the final two coordinates therefore has $0$s on the diagonal and $1$s on the off diagonal. Every other non-zero entry of $\hat{M}_z$ is a strictly positive power of $z^{1/2}$.

For $\epsilon>0$ we define the forward and backward cones $C^+(\epsilon)$ and $C^-(\epsilon)$ by
$$
C^+(\epsilon) = \{(v_1, \ldots, v_N) \in \mathbb C^N \; : \; \|(v_1, \ldots, v_{N-2})\|_1 \le \epsilon \cdot \|(v_{N-1}, v_N)\|_1\}
$$
and
$$
C^-(\epsilon) = \{(v_1, \ldots, v_N) \in \mathbb C^N \; : \; \epsilon \cdot \|(v_1, \ldots, v_{N-2})\|_1 \ge \|(v_{N-1}, v_N)\|_1\}
$$
For $\epsilon < 1$ these two cones intersect only in the origin.
\begin{lemma}
The symmetric transfer-matrix $\hat{M}_z$ maps $\mathbb C^N \setminus C^-(\epsilon)$ into $C^+(\epsilon)$ whenever
$$
|z|<\frac{\epsilon^4}{N^2(1+\epsilon)^2}.
$$
\end{lemma}
\begin{proof}
    Let $v = (v_1, \ldots, v_N) \in \mathbb C^N \setminus C^-(\epsilon)$ and write $\hat{M}_z v = w = (w_1, \ldots, w_N)$. It follows that
    \begin{alignat*}{2}
    \|(w_1, \dots, w_{N-2})\|_1 \le & (N-2) \cdot \max_{j \le N-2} |w_j|
    &&\le  (N-2) |z|^{\frac{1}{2}}\cdot \sum_{i=1}^N |v_i|\\
    = & (N-2) |z|^{\frac{1}{2}} \|v\|_1
    &&\le  (N-2) |z|^{\frac{1}{2}} \frac{\epsilon+1}{\epsilon} \|(v_{N-1}, v_N)\|_1.
    \end{alignat*}
    On the other hand
    \begin{align*}
    \|(w_{N-1},w_N)\|_1 &\ge |v_{N-1}| + |v_N| - 2 \sum_{i=1}^{N-2} |z|^{\frac{1}{2}} |v_i|\\
    &= \|(v_{N-1}, v_N)\|_1 - 2|z|^{\frac{1}{2}} \|(v_1, \ldots, v_{N-2})\|_1\\
    &\ge \left( 1 - \frac{2|z|^\frac{1}{2}}{\epsilon}\right) \|(v_{N-1}, v_N)\|_1.
    \end{align*}
    The inclusion $M_z (\mathbb C^N\setminus C^-(\epsilon) )\subset C^+(\epsilon)$ is therefore satisfied whenever
    $$
    \epsilon\Big(1 - \frac{2|z|^\frac{1}{2}}{\epsilon}\Big) \ge (N-2) |z|^{\frac{1}{2}} \frac{\epsilon+1}{\epsilon},
    $$
    which is satisfied whenever
    $$
    |z| \le \frac{\epsilon^4}{N^2(1+\epsilon)^2}.
    $$
\end{proof}

From now on we fix $\epsilon = \frac{1}{3}$ so that the forward and backward cones $C^+(\frac{1}{3})$ and $C^-(\frac{1}{3})$ are forward respectively backward invariant whenever $|z|< \frac{1}{144 N^2}$.

\begin{corollary}
For $|z|< \frac{1}{144 N^2}$ the two eigenvectors $\hat{v}^+(z)$ and $\hat{v}^-(z)$ of $\hat{M}_z$ corresponding to the maximal eigenvalues $q^+(z)$ and $q^-(z)$ are contained in $C^+(\frac{1}{3})$, while all other (generalized) eigenvectors are contained in $C^-(\frac{1}{3})$.
\end{corollary}
\begin{proof}
    The statement clearly holds for $|z|$ sufficiently small. For any fixed $z$ the entries of the matrix $\hat{M}_{\sqrt{x} z^{1/2}}$ are continuous functions of $x$ for $x \in [0,1]$. The statement therefore follows for any $|z|< \frac{1}{144 N^2}$ from the previous lemma, using the continuity of the set of eigenvectors of $\hat{M}_{\sqrt{x} z^{1/2}}$.
\end{proof}

\begin{lemma}
    \label{lem: bound rest eigenvalues}
For $|z| < \frac{1}{144 N^2}$ the absolute values of the two eigenvalues $q^+(z)$ and $q^-(z)$ are at least twice as large as the absolute value of any other eigenvalue.
\end{lemma}
\begin{proof}
    Let us first write $v$ for one of the eigenvectors $\hat{v}^+(z)$ or $\hat{v}^-(z)$ of $\hat{M}_z$, and write $w = \hat{M}_z v$. Using that $v \in C^+(1/3)$ we obtain
    $$
    \|(w_{N-1}, w_N)\|_1 \ge \|(v_{N-1}, v_N)\|_1 - 2|z|^{\frac{1}{2}} \|(v_1, \ldots, v_{N-2})\|_1 = \|(v_{N-1}, v_N)\|_1 \cdot \left(1 - \tfrac{1}{18N}\right),
    $$
    which implies that $|q^+(z)|$ and $|q^-(z)|$ are bounded from below by $17/18$. 
    
    Now let $w = \hat{M}_z v$ for an eigenvector $v \in C^-(1/3)$. Then
    \begin{align*}
    \|(w_1, \ldots, w_{N-2})\|_1 &\le (N-2) \cdot \max_{j \le N-2} \|w_j\|\\
    &\le (N-2) \left( |z|^{\frac{1}{2}} \|(v_{N-1}, v_N)\|_1 + |z| \|(v_1,\ldots, v_{N-2})\|_1\right)\\
    &\le (N-2) \left(\tfrac{1}{36N} + \tfrac{1}{144 N^2} \right)\|(v_1,\ldots, v_{N-2})\|_1\\
    &\le \tfrac{1}{36} \|(v_1,\ldots, v_{N-2})\|_1.
    \end{align*}
    It follows that the corresponding eigenvalue is bounded above by $1/36$, which proves the statement for any $N \ge 1$.
\end{proof}

\subsubsection{Proof of the main theorem}

In this section we will again prove that zeros of $Z(C_n\square \mathcal{T};\lambda)$ accumulate at $\infty$, as is done in Lemma~\ref{lem: zer_acc_infty_basic}. Similar to the proof of that lemma, we use that $\frac{z^{n\alpha} \cdot Z(C_n \square \mathcal{T}; 1/z)}{q^+(z)^n} = 1 + q^{-}(z)/q^{+}(z) + \mathcal{O}(z)$. This culminates in a proof of Theorem~\ref{thm: zeros_at_inf}, which, as we showed in the beginning of this section, leads to a proof of the second part of the main theorem. We define $\beta(z) = q^{-}(z)/q^{+}(z)$.

\begin{lemma}
    \label{lem: image alpha}
    Suppose $z\in \mathbb{C}$ satisfies $|z| < \frac{1}{(6N^2)^{\alpha+2}}$ then $|\beta(z) + 1| \geq \frac{1}{2}|z|^\alpha$.
\end{lemma}
\begin{proof}
    We can assume that $z\neq 0$. We have 
    \begin{align*}
        |q^+(z)| = |1 + \sum_{n=1}^\infty q_n^+ z^n|
        &\leq 
        1  + \sum_{n=1}^\infty |q_n^+| |z|^n
        \leq
        1 + \sum_{n=1}^\infty N \cdot (6N^2)^n |z|^n \\
        &\leq
        1 + N \sum_{n=1}^\infty \Bigg(\frac{1}{(6N^2)^{\alpha+1}}\Bigg)^n
        =
        1  + \frac{N}{(6N^2)^{\alpha+1}-1} < \frac{3}{2},
    \end{align*}
    where we used Lemma~\ref{lem: q bound} for the bound on $|q_n^+|$.
    We now also have 
    \[
        |\beta(z) + 1| = \left|\frac{q^+(z)+q^{-}(z)}{q^+(z)}\right|
        \geq \frac{2}{3} \left|q^+(z)+q^{-}(z)\right| 
        = \frac{2}{3} |z|^\alpha \cdot \left|\frac{q^+(z)+q^{-}(z)}{z^\alpha}\right|.
    \]
    We now use Lemma~\ref{lem: eigenvalue sum coefficients}, which says that $q_n^++q_n^- = 0$ for $n < \alpha$, while $q_\alpha^+ + q_\alpha^- \geq 1$. We get
    \[
        \left|\frac{q^+(z)+q^{-}(z)}{z^\alpha}\right| \geq 1 - \sum_{n=1}^\infty
        (|q_{\alpha+n}^+| + |q_{\alpha+n}^-|) |z|^n
        \geq 1 - 2N \sum_{n=1}^\infty (6N^2)^{n+\alpha} |z|^n
        = 1 - \frac{2N (6N^2)^\alpha}{(6N^2)^{\alpha+1}-1} > \frac{3}{4}.
    \]
    We therefore find that $|\beta(z) + 1| > \frac{1}{2}|z|^\alpha$.
\end{proof}

The following is a purely geometric lemma that will be useful in the subsequent proof.

\begin{lemma}
    \label{lem: Sector in Disk}
    Let $0< \rho < 1$. The disk of radius $\rho$ around $-1$ contains the sector
    \[
        S_\rho = \{z \in \mathbb{C}: 1-\frac{1}{2}\rho \leq |z| \leq 1 + \frac{1}{2}\rho \text{ and } \pi-\frac{1}{5}\pi \rho \leq \arg(z) \leq \pi + \frac{1}{5}\pi \rho\}.
    \]
    Moreover, for an integer $n$ with $n \geq 40/\rho$ the sector $S_\rho$ contains at least $\frac{1}{8}n\rho + 2$ distinct $n$th roots of unity, i.e. $\zeta \in \mathbb{C}$ such that $\zeta^n = 1$.
\end{lemma}

\begin{proof}
    Take $z \in {S}_\rho$. We can write $-z = r(\cos(\theta) + i\sin(\theta))$ for real values $r,\theta$ with $|1-r| \leq \frac{1}{2}\rho$ and $|\theta| \leq \frac{1}{5}\pi\rho$. We thus find
    \[
        |1-z|^2 = 1 - 2r\cos(\theta) + r^2 \leq (1-r)^2 + r\theta^2 \leq \frac{1}{4} \rho^2 + \frac{3}{2}(\frac{1}{5}\pi)^2 \rho^2 < \rho^2, 
    \]
    where we used that $\cos(\theta) \geq 1-\theta^2/2$. We conclude that the distance from $-1$ to $z$ is indeed less than $\rho$.

    Now let $n \in \mathbb{Z}_{\geq 1}$. For even $n$ the distinct roots of unity inside $S_\rho$ are given by $-\exp(2\pi i k/n)$ for integer $k$ satisfying $|k| \leq \frac{1}{10} \rho n$. There are $2\lfloor \frac{1}{10} \rho n\rfloor + 1$ such $k$. For odd $n$ the distinct roots of unity inside $S_\rho$ are given by $-\exp(\pi i (2k+ 1)/n)$ for integer $k$ satisfying $|2k+1| \leq \frac{1}{5} n \rho$ there are $\lfloor \frac{1}{10}n\rho - \frac{1}{2}\rfloor + \lfloor \frac{1}{10}n\rho + \frac{1}{2}\rfloor + 1$ such $k$. In both cases there are at least 
    \[
        \frac{1}{5}n\rho - 1 = \frac{1}{8}n\rho + \frac{3}{40}n\rho  - 1 \geq \frac{1}{8}n\rho + 2
    \]
    roots of unity inside $S_\rho$.
\end{proof}
We can now prove Theorem~\ref{thm: zeros_at_inf}, which we restate here for convenience. 
\explicitBounds*

\begin{proof}
    Let $B_{1/R}$ denote the disk of radius $1/R$.
    By Lemma~\ref{lem: image alpha} the image $\beta(B_{1/R})$ contains a disk of radius $\frac{1}{2}R^{-\alpha}$ around $-1$. Furthermore, by Lemma~\ref{lem: Sector in Disk}, this disk contains a sector $S_{\frac{1}{2}R^{-\alpha}}$ as defined in that lemma. 
    
    Let $k = \lceil \frac{1}{16}n R^{-\alpha}\rceil$. It follows from from Lemma~\ref{lem: Sector in Disk} that there are at least $k+2$ angles $\theta_1, \dots, \theta_{k+2}$, ordered increasingly, such that $e^{in\theta_m} = 1$ and $e^{i\theta_m}$ is contained in $S_{\frac{1}{2}R^{-\alpha}}$ for all $m$. For $m = 1, \dots, k+1$ define 
    \[
        T_{m} = \{z \in \mathbb{C}: 1-\frac{1}{4} R^{-\alpha} \leq |z| \leq 1 + \frac{1}{4}R^{-\alpha} \text{ and } \theta_m \leq \arg(z) \leq \theta_{m+1}\}.
    \]
    Observe that $T_m\subseteq \beta(B_{1/R})$ for all $m$. 
    
    We claim that for any $w \in \partial T_m$ we have $|1+w^n| > \left(\tfrac{1}{2}\right)^n N$. Because $n \gg N$ clearly $\left(\tfrac{1}{2}\right)^n N < \frac{1}{2}$, so it will be sufficient to prove that $|1+w^n| > \frac{1}{2}$. On the radial arcs of $\partial T_m$ we have $w^n = |w|^n$ so $|1 + w^n| = 1 + |w|^n > \frac{1}{2}$. If $w$ lies in the inner circular arc of $\partial T_m$ we have 
    \[
        |1+w^n| \geq 1 - |w|^n = 1 - \left(1-\frac{1}{4} R^{-\alpha}\right)^n \geq 1 - \exp\left[-\frac{n}{4}R^{-\alpha}\right] \geq 1 - e^{-20} > \frac{1}{2}.
    \]
    If $w$ lies on the outer circular arc of $\partial T_m$ we have
    \[
        |1+w^n| \geq |w|^n - 1 = \left(1+\frac{1}{4} R^{-\alpha}\right)^n - 1 \geq 1 + \frac{n}{4} R^{-\alpha} - 1 \geq 20 > \frac{1}{2}.
    \]
    This proves the claim.

    We now recall that 
    \[
        \frac{z^{n\alpha} \cdot Z(C_n \square \mathcal{T}; 1/z)}{q^+(z)^n} = 1 + \beta(z)^n + \sum_{s \neq q^{\pm}(z)} \left(\frac{s}{q^+(z)}\right)^n,
    \]
    where the sum runs over the eigenvalues of $M_z$ not equal to $q^{\pm}(z)$. Let $Q(z)$ denote this latter sum. By Lemma~\ref{lem: bound rest eigenvalues} we have that $|Q(z)| \leq \left(\frac{1}{2}\right)^n N$ for all $z \in B_{1/R}$. Note that $T_m$ contains an element $w_0$ such that $w_0^n = -1$. Consider a connected component $C_m$ of $\beta^{-1}(T_m)$ inside $B_{1/R}$. By the maximum modulus principle $C_m$ is simply connected and $\partial C_m$ is mapped to $\partial T_m$ by $\beta$. Moreover, $C_m$ contains an element $z_0$ in its interior with $\beta(z_0) = w_0$. For $z \in \partial C_m$ we thus have 
    \[
        |1 + \beta(z)^n| > \left(\tfrac{1}{2}\right)^n N \geq |Q(z)|,
    \]
    while $1+\beta(z_0)^n = 0$. It follows from Rouch\'e's theorem that $1 + \beta(z)^n + Q(z)$ contains a zero inside the interior of $C_m$. Therefore $z^{n\alpha} \cdot Z(C_n \square \mathcal{T}; 1/z)$ has $k+1$ distinct zeros inside $B_{1/R}$. As long as such a zero $z$ is itself nonzero then $\lambda = 1/z$ is a zero of $Z(C_n \square \mathcal{T}; \lambda)$ with norm at least $R$. We conclude that $Z(C_n \square \mathcal{T}; \lambda)$ has at least $k = \lceil \frac{1}{16}n R^{-\alpha}\rceil$ such zeros. 
\end{proof}

This theorem leads to a proof of the second part of the main theorem as is shown in the beginning of this section.

\section{An FPTAS for balanced tori}\label{sec:algorithms}
In this section we give a proof of Proposition~\ref{prop:algorithm}.
We will require the Newton identities that we recall here for convenience of the reader.
Let $p(x)=\sum_{j=0}^n a_j x^j$ be a polynomial with positive constant term and let $\log p(x)=\log(a_0)+\sum_{j\geq 1}-p_j \tfrac{x^j}{j}$ be the series expansion of the logarithm of $p$ around $0$. 
Then the Newton identities yield (cf.~\cite[Proposition 2.2]{PatelRegts17})
\begin{equation}\label{eq:newton}
    ka_k=-\sum_{i=0}^{k-1}a_ip_{k-i}
\end{equation}
for each $k\geq 1$, where $a_i = 0$ for $i > n$.

Proposition~\ref{prop:algorithm} immediately follows from the following more detailed result.
\begin{prop}
Let $d\in \mathbb{Z}_{\geq 2}$ and let $C>0$. Let $\delta(d,C)$ be the constant from Theorem~\ref{thm:main zero-freeness result contour version}.
For each $\lambda$ such that $|\lambda|>1/\delta(d,C)$ there exists an FPTAS for approximating $Z_{\torus}(\lambda)$ for $\torus\in \Balancedtori(d,C)$.
\end{prop}

\begin{proof}
Let us write $p_1(z):=Z^{\text{even}}_{\text{match}}(\torus;z)+Z^{\text{large}}_{\text{match}}(\torus;z)$, $p_2(z):=Z^{\text{odd}}_{\text{match}}(\torus;z)$ and $p(z)=p_1(z)+p_2(z).$
Taking $z=1/\lambda$, it suffices to approximate $p(z)$ by Corollary~\ref{cor:Zmatch = Zind}. 

Since $p$ has no zeros in the disk of radius $\delta(d,C)$ it suffices by Barvinok's interpolation method (\cite[Section 2.2]{BarBook}) to compute an $\varepsilon$-approximation to $\log p(z)$. This can be done by computing the first $O(\log (n/\varepsilon))$ coefficients of the Taylor series of $\log p(z)$. 
By \eqref{eq:newton} we can compute the first $m$ coefficients of the Taylor series of $\log p(z)$ from the first $m$ coefficients of the polynomial $p(z)$ in $O(m^2)$ time.
These coefficients in turn can be obtained from the first $m$ coefficients of $p_1(z)$ and $p_2(z)$, which in turn, using \eqref{eq:newton} again, can be computed from the first $m$ coefficients of the Taylor series of $\log p_1(z)$ and $\log p_2(z)$ in $O(m^2)$ time.
To obtain an FPTAS it thus suffices to compute the first $O(\log (n/\varepsilon))$ of the Taylor series of $\log p_1(z)$ and $\log p_2(z)$ in time polynomial in $n/\varepsilon$.

By the cluster expansion we have power series expressions for $\log p_1(z)$ given in \eqref{eq:cluster expansion large even} and for $\log p_2(z)$ given in \eqref{eq:cluster expansion trunc} using Theorem~\ref{thm:stable contours are all contours}.
From these we can extract the coefficients of the respective Taylor series.
Indeed, we can restrict the sum~\eqref{eq:cluster expansion large even} to clusters $X=\{\gamma_1,\ldots, \gamma_k\}$ such that  $\sum_{i=1}^k \|\gamma_i\|\leq  m$ and compute the coefficients of $z^j$ for $j\leq m$ of this restricted series.
The idea is to do this iteratively, since the weights appearing in the sum, $w(\gamma;z)$ are ratios of partition functions of smaller domains for which we can assume that we have already computed the first $m$ coefficients of its Taylor expansion around $0$.

To make this precise we need to combine some ingredients from~\cite{PirogovSinaiWillGuusTyler}.
We wish to apply Theorem 2.2 from~\cite{PirogovSinaiWillGuusTyler}\footnote{In the published version there is an error in the proof of that result, but this is corrected in a later arXiv version \texttt{{arXiv:1806.11548v3}}}.
For this we should view both $p_1$ and $p_2$ as polymer partition functions of a collection of bounded degree graphs. For us this collection will be the collection of all induced closed subgraphs of tori contained in $\Balancedtori_d(C)$ and denoted by $\mathfrak{G}$. (Here we maintain the information of the torus containing the closed induced subgraph.)
For $p_2$ this is clear but for $p_1$ this is a bit more subtle since in~\cite{PirogovSinaiWillGuusTyler} supports of polymers are connected subgraphs of graphs in $\mathfrak{G}$.
We would like to view our contours as polymers, but large contours may have disconnected support.
With this change, there are some potentials issues with the proof of Theorem 2.2.
We first indicate how to circumvent these issues and then verify the assumptions of that theorem.

One potential issue is in the use of \cite[Lemma 2.4]{PirogovSinaiWillGuusTyler}.
We sidestep this in a similar way as in the proof of Lemma~\ref{lem:large contours bound}.

Let $G\in \mathfrak{G}$. 
We know that $G$ is an induced closed subgraph of some torus $\torus$ in $\Balancedtori_d(C)$. 
Let $\ell_1$ be the shortest side length of $\torus$. Then the number of vertices of $G$, denoted by $n$, is at most $\exp(C\ell_1)$.
We need to list all subgraphs $H$ of $G$ such that either $H$ is connected or that each component of $H$ has size at least $\ell_1$ (since any component of a large contour has at least $\ell_1$ vertices) in time $\exp(O(m))$.
For connected graphs $H$ this follows directly from~\cite[Lemma 2.4]{PirogovSinaiWillGuusTyler}.
We now address the listing of subgraphs $H$ that are not necessarily connected. 
The number of components of such $H$ is at most $m/\ell_1$. 
By \cite[Lemma 2.4]{PirogovSinaiWillGuusTyler} it takes time $n \exp(O(m_i))$ to list all connected subgraphs $H_i$ of size $m_i$ and therefore it takes time $n^t \exp(\sum_{i=1}^t O(m_i))$ to list all subgraphs $H$ with $t$ components of sizes $m_1,\ldots,m_t$ respectively.
Let us denote $k:=\lceil m/\ell_1\rceil$. 
Putting this together this gives a running time bound of 
\begin{align*}
& \sum_{\substack{m_1, \ldots, m_k \\ \sum m_i = m \text{ and } m_i \geq \ell_1}}\prod_{i=1}^kn\exp(O(m_i))\leq \binom{m+k}{m}n^k\exp(O(m))
\\
&=n^k\exp(O(m))\leq \exp(kC\ell_1)\exp(O(m))=\exp(O(m)),
\end{align*}
for listing these graphs, as desired.

Another potential issue is in the use of cluster graphs in the proof of \cite[Theorem 2.2]{PirogovSinaiWillGuusTyler}.
In~\cite{PirogovSinaiWillGuusTyler} cluster graphs are assumed to be connected, but for us they may be disconnected (in case one of the contours in the cluster is large). 
In that case we have a lower bound of $\ell_1$ on the size of each component. 
So as above we can construct the list of all cluster graphs of size $O(m)$ in time $\exp(O(m))$.
With these modifications the proof of Theorem 2.2 given in~\cite{PirogovSinaiWillGuusTyler} still applies.

We next verify all the assumptions of (the modification of) Theorem 2.2 in~\cite{PirogovSinaiWillGuusTyler}.

The first assumption in the theorem is clearly satisfied, since $\|\gamma\|\leq |\overline{\gamma}|$ for any contour $\gamma$.

Our weight functions indeed satisfy Assumption 1 in~\cite{PirogovSinaiWillGuusTyler} by Lemma~\ref{lem:Peierls condition}.
It follows from the proof of~\cite[Lemma 3.3]{PirogovSinaiWillGuusTyler} that the first $m$ coefficients of the weights $w(\gamma;z)$ can be computed in time $\exp(m+\log |\gamma|).$ 
Here we need to take into account that large contours may consist of more than one component, and they should come first in the ordering of contours that is created in the proof of that lemma.

In our setting the third requirement translates that for a subgraph $H$ of some $G\in \mathfrak{G}$ we need to be able to list all polymers whose support is equal to $H$ in time $\exp(O(|V(H)|)$.
Let $\torus$ be the torus containing $G$. Let $\ell_1$ denote its smallest side length.
In case $H$ is not connected we know that we are dealing with a potentially large contour, while if $H$ is connected we have to compute its box-diameter to check whether or not the contour is large or small. This can be done in time polynomial in $|V(H)|$.
If $H$ is a candidate large contour it must have size at least $\ell_1$ and since the number of vertices of $G$ is at most $\exp(O(\ell_1)=\exp(O(|V(H)|)$, it follows that we can determine all components of $\torus\setminus V(H)$ in time $\exp(O(|V(H)|)$. 
If $H$ is a candidate small contour, we can determine all components of $\torus\setminus V(H)$ of size bounded by $|V(H)|^d$ in time polynomial in $|V(H)|$, by breadth first search. The remaining component must then be the exterior of the candidate contour.
We then go over all possible ways of assigning $0,1$ to the vertices of $V(H)$ and types to the components, i.e. select even or odd and check whether this yields a valid configuration. 
For this we need to check that vertices of $H$ are incorrect as per Definition~\ref{def:correct}.
Since the number of components is at most $O(|V(H)|)$ this takes time $\exp(O(|V(H)|)$.

The fourth assumption requires zero-freeness, which follows from convergence of the cluster expansion given in Theorem~\ref{thm:logtrunc converges_tori} in combination with Theorem~\ref{thm:stable contours are all contours}  for $p_2$ and in Theorem~\ref{thm:logtrunc converges for even plus big} for $p_1$.

This finishes the proof.
\end{proof}

\bibliographystyle{alpha}
\bibliography{biblio}

\end{document}